\documentclass{siamart}
\usepackage{amsmath}
\usepackage{amssymb}
\usepackage{enumerate}
\usepackage{lipsum}
\usepackage{amsfonts}
\usepackage{graphicx}
\usepackage{epstopdf}
\usepackage{algorithmic}
\newcommand{\real}{\mathbb{R}}  % The real numbers.
\def \comp{\mathbb{C}} % the complex number

\def \eqn {\begin{equation}}
\def \eeqn {\end{equation}}
\def \pa {\partial}
\def \hardy {\mathbb{H}} % hardy space
\def \e {\pmb e}
\def \eb {\overline{\pmb e}}
\newtheorem{thm}{Theorem}[section]

\newtheorem{lem}[thm]{Lemma}
\begin{document}

\title{Jost Solutions and the Direct Scattering Problem of the Benjamin--Ono Equation}

\maketitle

\author{Yilun Wu}
%\address{Mathematics Department, Brown University, Providence, RI 02912}
%\email{yilun_wu@brown.edu}
%\urladdr{http://www-personal.umich.edu/~yilunwu/} 

% REQUIRED
\begin{abstract}
In this paper, we present a rigorous study of the direct scattering problem that arises from the complete integrability of the Benjamin--Ono (BO) equation. In particular, we establish existence, uniqueness, and asymptotic properties of the Jost solutions to the scattering operator in the Fokas--Ablowitz inverse scattering transform (IST). Formulas relating different scattering coefficients are proven, together with their asymptotic behavior with respect to the spectral parameter. This work is an initial step toward the construction of general solutions to the BO equation by IST.
\end{abstract}

% REQUIRED

% REQUIRED

%%%%%%%%%%%%%%%%%%%%%%%%%%%%%%%%%%%%%%%%%%%%%%%%%%%%%%%%%%%%%%%%%%%%%%
\section{Introduction}
%%%%%%%%%%%%%%%%%%%%%%%%%%%%%%%%%%%%%%%%%%%%%%%%%%%%%%%%%%%%%%%%%%%%%%

The Benjamin--Ono (BO) equation may be written as
\eqn\label{eq: BO}
u_t+2uu_x-Hu_{xx}=0.
\eeqn
Here we consider $u=u(x,t)$ a real-valued function of space and time, both one-dimensional, and $H$ is the Hilbert transform defined by
\eqn
Hf(x) = \text{P.V.}\frac1\pi \int_{-\infty}^{\infty}\frac{f(y)}{x-y}~dy.
\eeqn
Formulated by Benjamin \cite{benjamin1967internal} and Ono \cite{ono1975algebraic}, the BO equation \cref{eq: BO} is used to model long internal gravity waves in a two-layer fluid. Typical setup of the models requires the wave amplitudes to be much smaller than the depth of the upper layer, which in turn is small compared with the wavelengths, while the lower layer has infinite depth. See Davis and Acrivos \cite{davis1967solitary}, Choi and Camassa \cite{choi1999fully} and Xu \cite{xu2010asymptotic} for more details on the derivation of \cref{eq: BO}. One can also observe (see \cite{ablowitz1991solitons}) that the BO equation \cref{eq: BO} can be formally obtained from the Intermediate Long Wave (ILW) equation by passing to the deep water limit, whereas the shallow water limit of the ILW equation gives the Korteweg--de Vries (KdV) equation. 

The BO equation \cref{eq: BO} is known to be well-posed for initial data in a Sobolev space $H^s(\real)$. Local and global well-posedness of \cref{eq: BO} were obtained by Saut \cite{saut1979sur}, I{\'o}rio \cite{jose1986cauchy}, Ponce \cite{ponce1991global}, Koch and Tzvetkov \cite{koch2003local}, Kenig and Koenig \cite{kenig2003local} and Tao \cite{tao2004global}. The best known result in \cite{tao2004global} establishes global well-posedness in $H^s(\real)$ for $s\ge 1$.

The BO equation \cref{eq: BO} was also found to be completely integrable. The Lax pair of \cref{eq: BO} was discovered by Nakamura \cite{nakamura1979direct} and Bock and Kruskal \cite{bock1979two}. An equivalent but formally different Lax pair was presented in Wu \cite{wu2016simplicity}. Fokas and Ablowitz \cite{fokas1983inverse} formulated the direct and inverse scattering problems for \cref{eq: BO} and obtained soliton solutions. See also Kaup and Matsuno \cite{kaup1998inverse} and Xu \cite{xu2010asymptotic}. As is the case for many other completely integrable equations, one expects to be able to construct solutions to the Cauchy problem of the BO equation using the Fokas--Ablowitz inverse scattering transform (IST). Even though the BO equation is known to be well-posed in $H^s(\real)$, a solution by IST makes full use of the integrability structure of the equation, and will provide key tools and insights for stability and asymptotic analysis. This plan was carried out by Coifman and Wickerhauser \cite{coifman1990scattering} for sufficiently small initial data. It turns out that the Fokas--Ablowitz IST does not behave well enough to be solved by iteration (contraction mapping principle) even under a small potential assumption, so Coifman and Wickerhauser actually used a more complicated regularized IST and solved it by iteration. Up to the present time, a rigorous analysis of the Fokas--Ablowitz or related IST without a small potential assumption is still lacking, and as a result, no rigorous IST solution to the large data Cauchy problem of the BO equation has been proven.

As a first step toward this goal, the author \cite{wu2016simplicity} studied the $L_u$ operator in the Lax pair of the BO equation, and proved that its discrete spectrum is finite and simple. These are some key spectral assumptions made by Fokas and Ablowitz in their definition of the scattering data of the IST. A few other useful properties about the eigenfunctions were also established. 

In this paper, we will examine the full spectrum of the $L_u$ operator and provide a complete study of the direct scattering problem in the Fokas--Ablowitz IST. We will also investigate the asymptotic and regularity properties of the scattering data thus constructed. Such investigations may provide directions to the correct setup and future study of the inverse problem. The paper is organized as follows. We present the essential ingredients of the Fokas--Ablowitz IST in \Cref{sec: FA IST}. It will be evident that the central objects of study for the direct scattering problem are certain eigenfunctions of the $L_u$ operator in the Lax pair. These are the so-called Jost solutions (or Jost functions). In \Cref{sec: existence}, we prove the existence and uniqueness of these Jost solutions. This will provide basis for the construction of the scattering data. As we will see in \Cref{sec: existence}, what we need to solve are certain Fredholm integral equations, and the main difficulty is to prove a vanishing lemma for the corresponding integral operator. In \Cref{sec: scattering data}, we construct the scattering coefficients in the Fokas-Ablowitz IST from the Jost solutions, and prove certain important relations between them that are known only on the formal level in the literature. In \Cref{sec: k 0 limit}, we prove asymptotic formulas for the Jost solutions and scattering coefficients as the spectral parameter $k$ approaches $0$. These very useful asymptotic formulas obtained formally in \cite{fokas1983inverse} and \cite{kaup1998inverse} help clarify the global behavior of the scattering coefficients, and may provide insight into the study of the inverse scattering problem. The key to proving these formulas is to solve a regularized Fredholm integral equation at $k=0$, and the crucial difficulty is again to prove a vanishing lemma for a regularized integral operator. In \Cref{sec: k infty limit}, we prove asymptotic formulas as the spectral parameter $k$ approaches infinity. Finally, we discuss the time evolution of the scattering data in \Cref{sec: time evo}. This point is worth discussing particularly because the operator that is used to define the Jost solutions is actually slightly different than the $L_u$ operator in the Lax pair.

We now set up standard spaces and notations used throughout the paper. The following convention is employed for the Fourier transform and inverse Fourier transform:
\begin{align}
F(f)(\xi)&=\hat{f} (\xi) = \int_{\mathbb{R}}e^{-i\xi x}f(x)~dx,\\
F^{-1}(f)(x) &= \check{f}(x) = \frac{1}{2\pi}\int_{\mathbb{R}}e^{ix\xi }f(\xi)~d\xi,
\end{align}
with their usual extension to tempered distributions.
The Cauchy projections $C_{\pm}$ are defined in terms of the Hilbert transform as
\eqn
C_{\pm}f = \frac{\varphi \pm iHf}{2}.
\eeqn
In other words, $\widehat{C_{\pm}f}= \chi_{\mathbb{R}_{\pm}}\hat{f}$. We denote the $L^p$ Hardy space of the upper half plane by $\hardy^{p,+}$. More specifically, $f(x)\in \hardy^{p,+}$ for $1<p\le\infty$ if it is the $L^p$ (and almost everywhere) boundary value of an analytic function $F(x+iy)$ for $z=x+iy$ in the upper half plane $\{y>0\}$, such that $\sup_{y>0}\|F(\cdot + iy)\|_p<\infty$. We denote $\hardy^{2,+}$ also by $\hardy^+$. Observe that $C_+(L^2(\real))=\hardy^+$. We fix the notation for weighted $L^p$ spaces and weighted Sobolev spaces as follows. Let $w(x)=1+|x|$ be the weight function. We define for $1\le p\le \infty$ and $s\in \real$
\eqn
L^p_s(\real)=\{f~|~w^{s}f\in L^p(\real)\},
\eeqn
and
\eqn
H^s_s(\real)=\{f~|~ f\in L^2_s, \text{ and }\hat f\in L^2_s\},
\eeqn
with norms $\|f\|_{L^p_s(\real)}=\|w^sf\|_p$ and $\|f\|_{H^s_s(\real)}=\|w^s f\|_2+\|w^s\hat f\|_2$. We denote the $L^p(\real)$ norm by $\|\cdot\|_p$. When doing estimates, we use $C$ to mean a generic constant, whose value may be enlarged from step to step.

%%%%%%%%%%%%%%%%%%%%%%%%%%%%%%%%%%%%%%%%%%%%%%%%%%%%%%%%%%%%%%%%%%%%%%
\section{The Fokas--Ablowitz inverse scattering transform}\label{sec: FA IST}
%%%%%%%%%%%%%%%%%%%%%%%%%%%%%%%%%%%%%%%%%%%%%%%%%%%%%%%%%%%%%%%%%%%%%%

Throughout this section, we assume $u(x,t)$ is sufficiently smooth with sufficiently rapid decay in $x$ for each $t$, and present the Fokas-Ablowitz IST formulated in \cite{fokas1983inverse}. Since the current paper provides rigorous analysis of the direct scattering problem, we will freely quote results in the later sections when describing the direct problem, and take note that the inverse problem calls for more analysis in future works. Since time is frozen when performing the IST, we drop the $t$ dependence of $u$ in the discussion.

We start by recalling the Lax pair of the BO equation \cref{eq: BO} presented in \cite{wu2016simplicity}. There we see that when $u$ is real, as is the case considered in this paper, we only need to take the Lax pair to be operators defined on $\hardy^+$:
\begin{align}
L_u\varphi &= \frac{1}{i}\varphi_x -C_+(uC_+\varphi), \label{def: Lu 0}\\
B_u\varphi &= \frac{1}{i}\varphi_{xx} +2[(C_+u_x)(C_+\varphi)-C_+((uC_+\varphi)_x)]. \label{def: Bu 0}
\end{align} 
Since $C_+$ acts as the identity on $\hardy^+$, we simplify the Lax pair further by dropping the $C_+$ in $C_+\varphi$ and write
\begin{align}
L_u\varphi &= \frac{1}{i}\varphi_x -C_+(u\varphi), \label{def: Lu}\\
B_u\varphi &= \frac{1}{i}\varphi_{xx} +2[(C_+u_x)\varphi-C_+((u\varphi)_x)]. \label{def: Bu}
\end{align} 
Notice that in \cref{def: Lu} and \cref{def: Bu}, $\varphi$ may be allowed to have moderate growth at infinity. For instance, $\varphi$ could be a function in a weighted $L^p$ space. On the other hand, the equivalence of the BO equation with the Lax equation does cling to the particular form \cref{def: Lu 0} and \cref{def: Bu 0}. By dropping $C_+$ from the equations, we run a potential risk of destroying the equivalence of the BO equation with the Lax equation, when $\varphi$ is not a function in $\hardy^+$. We will address this issue in \Cref{sec: time evo}, since its effect only comes into play when time evolution is concerned.

According to \cite{wu2016simplicity}, $L_u$ given in \cref{def: Lu}, regarded as an operator on $\hardy^+$, is self-adjoint with finitely many negative simple eigenvalues $\lambda_j$, $j=1,\dots,N$, and $[0,\infty)$ as the essential spectrum. We denote the resolvent set of $L_u$ by $\rho(L_u)=\comp\setminus \{\lambda_1,\dots,\lambda_N\}\setminus [0,\infty)$. By \Cref{lem: equiv int diff} and \Cref{thm: existence of Jost soln}, for each $k\in \rho(L_u)$, there exists a unique Jost solution $m_1(x,k)$ in $L^{\infty}(\real)$ to the following equation
\eqn
L_u m_1 = k(m_1-1)
\eeqn
such that $m_1(x,k)-1\to 0$ as $x\to \pm \infty$. $m_1(x,k)$ depends analytically on $k$. Furthermore, as $k$ approaches a positive real $\lambda$ from above or from below, $m_1(x,k)$ has limits $m_1(x,\lambda\pm 0i)\in L^{\infty}(\real)$. We abbreviate $m_1(x,k)$ as $m_1(k)$ when convenient. By Proposition 2.1 and Corollary 2.2 in \cite{wu2016simplicity}, for each negative simple eigenvalue $\lambda_j$, and normalized eigenfunction $\phi_j$, there exists a number $\gamma_j$, such that the Laurent expansion of $m_1(k)$ around $\lambda_j$ is
\eqn\label{eq: Laurent}
m_1(k) = -\frac{i}{k-\lambda_j}\phi_j + (x+\gamma_j)\phi_j+(k-\lambda_j)h(k,\lambda_j),
\eeqn
where $h(k,\lambda_j)$ is analytic in $k$ around $\lambda_j$. $\gamma_j$ is called the phase constant in the literature.

The scattering data of the Fokas-Ablowitz IST consist of the eigenvalues $\{\lambda_j\}_{j=1}^N$, the phase constants $\{\gamma_j\}_{j=1}^N$ and the scattering coefficient 
\eqn
\beta(\lambda) = i\int_\real u(x)m_1(x,\lambda+0i)e^{-i\lambda x}~dx
\eeqn
for $\lambda>0$.

The discussion above provides a minimal description of the direct scattering problem. However, to understand the connection to the inverse problem, we need to express the jump of $m_1(k)$ on the positive real line. To accomplish that we introduce another Jost function $m_e(x,\lambda- 0i)\in L^{\infty}(\real)$ which for $\lambda>0$ solves uniquely
\eqn
L_um_e=\lambda m_e
\eeqn
with asymptotic condition $m_e(x,\lambda-0i)\to 0$ as $x\to \infty$. The notation $\lambda-0i$ in $m_e(x,\lambda-0i)$ is natural in the integral equation it satisfies. The existence of $m_e$ is established in \Cref{thm: existence of Jost soln}. By \Cref{lem: def scattering coeff} and \Cref{lem: rel scattering coeff}, 
\eqn
m_1(\lambda+0i)-m_1(\lambda-0i)=\beta(\lambda)m_e(\lambda-0i),
\eeqn
and
\eqn\label{eq: deriv jump}
\e(\lambda)\partial_\lambda(\eb(\lambda) m_e(\lambda-0i))= \frac{\overline{\beta(\lambda)}}{2\pi i\lambda}m_1(\lambda-0i),
\eeqn
where $\e(\lambda)=\e(x,\lambda) = e^{i\lambda x}$. By \Cref{thm: lim k 0},
\eqn
\lim_{\lambda\searrow 0}m_1(\lambda-0i) = \lim_{\lambda\searrow 0}m_e(\lambda-0i).
\eeqn
Denoting the limit above by $m_1(0-0i)=m_e(0-0i)$, we obtain from \cref{eq: deriv jump}
\eqn
\eb(\lambda) m_e(\lambda-0i) = m_1(0-0i)+\int_0^{\lambda}\frac{\overline{\e(\mu)\beta(\mu)}}{2\pi i\mu}m_1(\mu-0i)~d\mu.
\eeqn
By the study performed in \Cref{sec: k 0 limit}, for a large class of potential $u$ called generic potentials, $m_1(0-0i)$ is actually equal to $0$.
Finally, by \Cref{thm: lim infty schwartz}, 
\eqn
C_+u = \lim_{k\to \infty}k(1-m_1(k)),
\eeqn
where the limit holds in $L^{\infty}(\real)$ in $x$.

Summarizing the above discussion, it is natural to cast the inverse scattering problem as follows. Given the negative eigenvalues $\{\lambda_j\}_{j=1}^N$, the phase constants $\{\gamma_j\}_{j=1}^N$ and suitable scattering coefficient $\beta(\lambda)$ for $\lambda>0$, find an analytic mapping $k\mapsto m_1(k)$ from the resolvent set $\comp\setminus \{\lambda_1,\dots,\lambda_N\}\setminus [0,\infty)$ to a suitable function space in $x$ such that
\begin{enumerate}[(a)]
\item Around every $\lambda_j$, the Laurent expansion of $m_1(k)$ has the form \cref{eq: Laurent} for some fixed function $\phi_j$ and mapping $h(k,\lambda_j)$ analytic for $k$ close to $\lambda_j$.
\item $m_1(k)$ has limits $m_1(\lambda\pm 0i)$ in suitable function spaces as $k$ approaches the positive real line from above and from below, such that
\begin{align}
&~m_1(\lambda+0i)-m_1(\lambda-0i) \notag\\
=&~ \beta(\lambda)\left(\e(\lambda) m_1(0-0i)+\int_0^{\lambda}\frac{\e(\lambda-\mu)\overline{\beta(\mu)}}{2\pi i\mu}m_1(\mu-0i)~d\mu\right). \label{eq: nonlocal jump}
\end{align}
\item $m_1(k)\to 1$ as $k\to \infty$.
\end{enumerate}
Once $m_1(x,k)$ is obtained by solving the inverse problem, $u(x)$ may be recovered by
\eqn
u= 2~\text{Re}\lim_{k\to \infty}k(1-m_1(k)).
\eeqn
This completes the formulation of the inverse scattering problem. 

The inverse problem is often called a nonlocal Riemann-Hilbert problem. \Cref{eq: nonlocal jump} is known as the nonlocal jump condition, in comparison with the usual jump condition appearing in a standard Riemann-Hilbert problem, where the integral in \cref{eq: nonlocal jump} is replaced by straightforward multiplication.

%%%%%%%%%%%%%%%%%%%%%%%%%%%%%%%%%%%%%%%%%%%%%%%%%%%%%%%%%%%%%%%%%%%%%%
\section{Existence and uniqueness of Jost solutions}\label{sec: existence}
%%%%%%%%%%%%%%%%%%%%%%%%%%%%%%%%%%%%%%%%%%%%%%%%%%%%%%%%%%%%%%%%%%%%%%

In this section, we solve certain modified eigenvalue equations for the operator $L_u = \frac1i\pa_x - C_+u$, with specified asymptotic conditions at $\pm \infty$. These are the Jost solutions that play a central role in the Fokas--Ablowitz IST. They encode properties of the spectrum of $L_u$, which, according to \cite{wu2016simplicity}, has the form $\{\lambda_1,\dots,\lambda_N\}\cup\{0\}\cup \real^+$, where $\real^+=(0,\infty)$.

In the following, two Jost functions $m_1(x,k)$ and $m_e(x,\lambda\pm 0i)$ will be considered. These are solutions to the following equations, with suitable asymptotic conditions at infinity:
\begin{align}
\frac1i \pa _x m_1 - C_+(um_1 )= k(m_1-1), \label{eq: m1 diff 0}
%m_1^{\pm}(x) - 1 &\to 0 \quad (x\to \pm\infty). \label{eq: m1 limit}
\end{align}
\begin{align}
\frac1i \pa _x m_e - C_+(um_e) = \lambda m_e.\label{eq: me diff 0}
%m_e^{\pm}(x) - e^{i\lambda x} &\to 0 \quad (x\to \pm\infty). \label{eq: me limit}
\end{align}
Here $\lambda\pm 0i\in \real^+\pm 0i$, and 
\eqn
k\in \rho(L_u)\cup (\real^+\pm 0i)=(\mathbb{C}\setminus\{\lambda_1,\dots\lambda_N\}\setminus [0,\infty))\cup (\real^+\pm 0i), 
\eeqn
which is the resolvent set glued with two copies of the positive real line. Later on, we will see that $m_1(x,k)$ has limits as $k$ approaches the positive real line from above and below. %Notice that the operator on the left hand side of \eqref{eq: m1 diff 0} and \eqref{eq: me diff 0} is slightly different from $L_u$ in that the $C_+$ before multiplication by $u$ is lost. Such a change has no effect if the function being operated on is in $\hardy^+$, since $C_+$ acts as identity on $\hardy^+$. However, for $m_1(x,\lambda\pm0i)$ and $m_e(x,\lambda\pm 0i)$, which will be seen to belong only to $L^{\infty}$, such a change does give a slightly different operator. The extent to which such a change affects the Lax pair and the time evolution of scattering data is elucidated further in Section \ref{sec: time evo}. 
The notation of $m_1(x,k)$ and $m_e(x,\lambda\pm 0i)$ is adapted to the asymptotic conditions at infinity, and may be abbreviated as $m_1(k)$, $m_e(\lambda\pm 0i)$, $m_e(\lambda\pm)$, or simply $m_1$ and $m_e$. In \cite{fokas1983inverse}, a different notation is used. We provide the translation of notation as follows: 
\begin{align}
M(x,\lambda) &= m_1(x,\lambda+0i),~ \overline{M}(x,\lambda)=m_e(x,\lambda+0i),\\
N(x,\lambda) &= m_e(x,\lambda-0i), ~\overline{N}(x,\lambda)=m_1(x,\lambda-0i).
\end{align}
The Jost functions can equivalently be described as solutions to certain Fredholm integral equations. To express these equations, we introduce the convolution kernels
\eqn\label{def: G}
G_k(x) = \frac{1}{2\pi}\int_0^{\infty}\frac{e^{ix\xi}}{\xi-k}~d\xi
\eeqn
for $k\in \mathbb{C}\setminus [0,\infty)$, and
\eqn\label{def: tilde G}
\widetilde{G}_k(x) = \frac{1}{2\pi}\int_{-\infty}^0\frac{e^{ix\xi}}{\xi-k}~d\xi
\eeqn
for $k\in \mathbb{C}\setminus (-\infty,0]$.
We have
\begin{align}
G_{\lambda\pm i\epsilon}(x)&= \frac{1}{2\pi}\int_{-\infty}^{\infty}\frac{e^{ix\xi}}{\xi-(\lambda\pm i\epsilon)}~d\xi -\widetilde{G}_{\lambda\pm i\epsilon}(x)  \notag\\
&= \pm i e^{\mp\epsilon x}e^{i\lambda x}\chi_{\real^{\pm}}(x) - \widetilde{G}_{\lambda\pm i\epsilon}(x) \label{eq: G_lambda ep decomp}
\end{align}
with
\begin{equation}\label{eq: G_lambda decomp}
G_{\lambda\pm 0 i} (x)= \lim_{\epsilon\searrow 0}G_{\lambda \pm i\epsilon}(x)
= \pm i e^{i \lambda x}\chi_{\real^{\pm}}(x) -\widetilde{G}_{\lambda}(x) ,
\end{equation}
for $\lambda>0$.
The limit in \cref{eq: G_lambda decomp} holds in the following sense: the first term in \cref{eq: G_lambda ep decomp} converges pointwise, and the second term in \cref{eq: G_lambda ep decomp} converges in $L^{p'}$ for every $p'\in[2,\infty)$. To see the latter, observe that $\widetilde{G}_{\lambda\pm i\epsilon}$ is the inverse Fourier transform of $\frac{\chi_{\real^-}(\xi)}{\xi-(\lambda\pm i\epsilon)}$, which converges to $\frac{\chi_{\real^-}(\xi)}{\xi-\lambda}$ in every $L^p$ for $p\in(1, 2]$, assuming $\lambda>0$.

We are ready to describe the Fredholm integral equations satisfied by the Jost solutions.

\begin{lem}\label{lem: equiv int diff}
Let $p>1$ and $s>s_1>1-\frac1p$ %$s\ge 0$ 
be given, and let %$u\in w^{-s}L^1(\real)\cap w^{-s}L^p(\real)$. 
$u\in L^p_s(\real)$. Suppose $m_1(x,k), m_e(x,\lambda\pm 0i) \in L^{\infty}_{-(s-s_1)}(\real)$ for fixed $k\in (\comp\setminus[0,\infty))\cup(\real^+\pm 0i)$ and $\lambda\in \real^+$, then the following are equivalent:
\begin{enumerate}[(a)]
\item $m_1(x,k), m_e(x,\lambda\pm 0i)$ solve 
\eqn\label{eq: m1 diff}
\frac1i \pa _x m_1 - C_+(um_1 )= k(m_1-1),
\eeqn
\eqn\label{eq: me diff} 
\frac1i \pa _x m_e - C_+(um_e) = \lambda m_e,
\eeqn
together with the asymptotic conditions 
\eqn\label{eq: m1 limit}
m_1(x,k) - 1 \to 0 \begin{cases}\text{ as } |x|\to \infty\quad &\text{ if }k\in \mathbb{C}\setminus [0,\infty),\\ \text{ as } x \to \mp \infty &\text{ if }k=\lambda\pm 0i\in \real^+\pm 0i.\end{cases}
\eeqn
\eqn\label{eq: me limit}
m_e(x,\lambda\pm 0i) -e^{i\lambda x} \to 0 \quad \text{ as }x\to \mp\infty.
\eeqn
The above asymptotic conditions should be read with either the upper sign or the lower sign.
\item $m_1(x,k), m_e(x,\lambda\pm 0i)$ solve the following integral equations:
\eqn\label{eq: m1 int}
m_1(x,k) = 1+G_k*(um_1(\cdot,k))(x),
\eeqn
\eqn\label{eq: me int}
m_e(x,\lambda\pm 0i) = \pmb{e}(x,\lambda) + G_{\lambda\pm 0i}*(um_e(\cdot, \lambda\pm0i))(x),
\eeqn
where $\pmb{e}(x,\lambda)$ denotes $e^{i\lambda x}$. 
\end{enumerate}
In addition, if either (a) or (b) holds, we have the stronger bounds 
\eqn
m_1(x,k) - 1\in  L^{\infty}(\real)\cap \hardy^{p,+}
\eeqn
for fixed $k\in \mathbb{C}\setminus [0,\infty)$, and
\eqn
m_1(x,\lambda\pm 0i), m_e(x,\lambda\pm 0i)\in  L^{\infty}(\real)
\eeqn
for fixed $\lambda\in \real^+$.
\end{lem}
\begin{proof}
First of all, we notice from the conditions on $u$, $m_1$, and $m_e$ that $$um_1,um_e\in L^p_{s_1}\subset L^1\cap L^p.$$ Since $L^q\subset L^1\cap L^p$ for every $1<q<p$, $f\in L^q$ for some $1<q\le 2$. The terms $C_+(um_1)$, $C_+(um_e)$ in \cref{eq: m1 diff} and \cref{eq: me diff} are well-defined as $C_+$ is bounded on $L^p$. To see that the convolution in \cref{eq: m1 int} and \cref{eq: me int} are well-defined and belong to $L^{\infty}$, we notice by \cref{eq: G_lambda decomp} that $G_k\in L^{p'}$ if $k\in \comp\setminus [0,\infty)$, and $G_{\lambda \pm 0i}\in L^{\infty}+ L^{p'}$ where $\frac1p+\frac1{p'}=1$. 

We now study $m(x,k)$ for $k\in \comp\setminus [0,\infty)$. In this case, we can actually show \cref{eq: m1 diff} is equivalent to \cref{eq: m1 int} without using the asymptotic condition \cref{eq: m1 limit}. To see this, we take the Fourier transform of \cref{eq: m1 diff} to get
\eqn
\xi \widehat{m_1}-\chi_{\real^+}\widehat{um_1} = k \widehat{m_1}-k\hat 1,
\eeqn
or 
\eqn\label{eq: hat m1 hardy}
\widehat{m_1} =  \hat 1 + \frac{\chi_{\real^+}}{\xi-k}\widehat{um_1}.
\eeqn
Now take the inverse Fourier transform to get \cref{eq: m1 int}. The convolution formula for inverse Fourier transform can be justified using the fact that $um_1\in L^q$ for some $1<q\le 2$. The above calculation can be reversed. Hence \cref{eq: m1 int} also implies \cref{eq: m1 diff}. To obtain the limiting condition \cref{eq: m1 limit} when $k\in \comp\setminus[0,\infty)$, we just observe that $\frac{\chi_{\real^+}}{\xi-k}\widehat{um_1}\in L^1$.
\Cref{eq: hat m1 hardy} also implies $m_1-1\in \hardy^{p,+}$. To see this, we apply the Marcinkiewicz multiplier theorem to the multiplier $\mu(\xi)=\frac{\chi_{\real^+}(\xi)e^{-y\xi}}{\xi-k}$ for every $y>0$. In fact
\begin{align}
\sup_{j\in\mathbb{Z}}\int_{2^j}^{2^{j+1}}|\mu'(\xi)|~d\xi &\le \sup_{j\in\mathbb{Z}}\int_{2^j}^{2^{j+1}}\left(ye^{-y\xi}\frac1{|\xi-k|}+\frac1{|\xi-k|^2}\right)~d\xi \notag\\
&\le C_k\left(1+\sup_{j\in\mathbb{Z}}\int_{2^j}^{2^{j+1}}ye^{-y\xi}\frac1{|\xi|+|k|}~d\xi\right) \notag\\
&\le C_k\left(1+\sup_{j\in\mathbb{Z}}ye^{-y2^j}\log\left(\frac{2^{j+1}+|k|}{2^j+|k|}\right)\right)\notag\\
&\le C_k\left(1+\sup_{j\in\mathbb{Z}}ye^{-y2^j}2^j\right)\notag\\
&\le C_k \left(1+\sup_{y\ge 0}ye^{-y}\right)\le C_k,\label{est: multiplier}
\end{align}
where $C_k$ is a generic constant depending only on $k$. Estimate \cref{est: multiplier} implies that the $L^p$ norm of 
\eqn
F(x+iy)=\frac1{2\pi}\int_0^{\infty}\frac{e^{i\xi(x+iy)}}{\xi-k}\widehat{um_1}(\xi)~d\xi
\eeqn
is uniformly bounded for $y>0$. On the other hand $F(x+iy)$ converges pointwise to $F(x+i0)=m_1(x,k)$ as $y\searrow 0$ by the dominated convergence theorem. Hence $m_1(x,k)-1\in \hardy^{p,+}$. 

We now work on $m_1(x,\lambda\pm 0i)$ and $m_e(x,\lambda\pm 0i)$. To simplify notation, we suppress the $x$ variable and $0i$, and only work on the case with the plus sign. The case with the minus sign can be treated similarly. 

We first prove the passage from \cref{eq: m1 diff} and \cref{eq: m1 limit} to \cref{eq: m1 int}. In fact, the Fourier transform of \cref{eq: m1 diff} gives
\eqn
\xi \widehat{m_1(\lambda+)} = \lambda \widehat{m_1(\lambda+)} - \lambda \widehat 1 +\chi_{\mathbb{R}^+}F(um_1(\lambda+)).
\eeqn
For every $\epsilon>0$, we divide by $\xi-(\lambda+i\epsilon)$ to get
\begin{align}\label{eq: epsilon fourier}
\widehat{m_1(\lambda+)} =&~ -\frac{i\epsilon}{\xi-(\lambda+i\epsilon)}F(m_1(\lambda+)-1)-\frac{\lambda+i\epsilon}{\xi-(\lambda+i\epsilon)}\widehat 1\notag\\
&\quad +\frac{1}{\xi-(\lambda+i\epsilon)}\chi_{\mathbb{R}^+}F(um_1(\lambda+)).
\end{align}
Since $\widehat 1$ is a multiple of $\delta$, we have
\begin{equation}
-\frac{\lambda+i\epsilon}{\xi-(\lambda+i\epsilon)}\widehat 1=\widehat 1.
\end{equation}
Now inverse Fourier transform \cref{eq: epsilon fourier} to get
\begin{equation}\label{eq: pre int}
m_1(\lambda+) = F^{-1}\left(-\frac{i\epsilon}{\xi-(\lambda+i\epsilon)}F(m_1(\lambda+)-1) \right)+1+G_{\lambda+i\epsilon}*(um_1(\lambda+)).
\end{equation}
By the decomposition \cref{eq: G_lambda decomp}, and the dominated convergence theorem,
\eqn\label{eq: justify 1}
\lim_{\epsilon \searrow 0}G_{\lambda+i\epsilon}*(um_1(\lambda+))=G_{\lambda+0i}*(um_1(\lambda+))
\eeqn
pointwise.
Since 
\begin{equation}
F^{-1}\left(-\frac{i\epsilon}{\xi-(\lambda+i\epsilon)}\right) = \epsilon\chi_{\mathbb{R}^+}(x)e^{i(\lambda+i\epsilon)x},
\end{equation}
we have
\begin{equation}\label{eq: fourier convolution}
F^{-1}\left(-\frac{i\epsilon}{\xi-(\lambda+i\epsilon)}F(m_1(\lambda+)-1) \right) = \epsilon\chi_{\mathbb{R}^+}(x)e^{i(\lambda+i\epsilon)x}*(m_1(\lambda+)-1).
\end{equation}
\Cref{eq: fourier convolution} can be justified as follows. First of all, it is true if $m_1(\lambda+)-1$ is replaced by a Schwartz class function. By the conditions on $m_1(\lambda+)$, it is obvious that $w^{-2-(s-s_1)}(m_1(\lambda+)-1)\in L^1$. Approximate $w^{-2-(s-s_1)}(m_1(\lambda+)-1)$ in $L^1$ by a sequence of Schwartz class functions $f_n$, and take the limit as $n\to \infty$ of
\eqn\label{eq: fourier conv 1}
F^{-1}\left(-\frac{i\epsilon}{\xi-(\lambda+i\epsilon)}\widehat{g_n} \right) = \epsilon\chi_{\mathbb{R}^+}(x)e^{i(\lambda+i\epsilon)x}*g_n(x),
\eeqn
where $g_n = w^{2+s-s_1}f_n$. The left hand side of \cref{eq: fourier conv 1} converges as tempered distributions to the left hand side of \eqref{eq: fourier convolution}. To study convergence of the right hand side, observe that 
\begin{align}
|\chi_{\mathbb{R}^+}(x)e^{i(\lambda+i\epsilon)x}*g_n| &= \left|\int_{-\infty}^x e^{-\epsilon(x-y)}e^{i\lambda(x-y)}g_n(y)~dy\right|\notag\\
&\le e^{-\epsilon x}(\sup_{y\le x}e^{\epsilon y}[w(y)]^{2+s-s_1}) \|w^{-2-(s-s_1)}g_n\|_{L^1}
\end{align}
It follows that the right hand side of \cref{eq: fourier conv 1} converges locally uniformly to the right hand side of \eqref{eq: fourier convolution}. Thus \eqref{eq: fourier convolution} holds. Now
\begin{align}
&~\epsilon\chi_{\mathbb{R}^+}(x)e^{i(\lambda+i\epsilon)x}*(m_1(\lambda+)-1) \notag \\
= &~\epsilon\int_0^{\infty}e^{i(\lambda+i\epsilon)y}(m_1(\lambda+)-1)(x-y)~dy \notag\\
= &~\int_0^{\infty}e^{i\lambda y/\epsilon}e^{-y}(m_1(\lambda+)-1)\left(x-\frac{y}{\epsilon}\right)~dy \label{eq: limit int}
\end{align}
We take the limit as $\epsilon \searrow 0$ of \cref{eq: limit int}. By \cref{eq: m1 limit}, $(m_1(\lambda+)-1)$ is bounded on $(-\infty,x]$ and approaches 0 as $x\to -\infty$, hence \cref{eq: limit int} tends to 0 for every $x$ by the dominated convergence theorem. By \cref{eq: justify 1}, \cref{eq: fourier convolution}, and the above discussion about \cref{eq: limit int}, the right hand side of \cref{eq: pre int} tends to the right hand side of \cref{eq: m1 int} as $\epsilon \searrow 0$. 

We can work similarly on $m_e(\lambda+)$. In this case, \cref{eq: pre int} is replaced by
\eqn
m_e(\lambda+) = F^{-1}\left(-\frac{i\epsilon}{\xi-(\lambda+i\epsilon)}\widehat{m_e(\lambda+)} \right)+G_{\lambda+i\epsilon}*(um_e(\lambda+)).
\eeqn
Again, 
\eqn
\lim_{\epsilon \searrow 0}G_{\lambda+i\epsilon}*(um_e(\lambda+)) =G_{\lambda+i0}*(um_e(\lambda+))\eeqn
pointwise, and $F^{-1}\left(-\frac{i\epsilon}{\xi-(\lambda+i\epsilon)}\widehat{m_e(\lambda+)} \right)$ equals
\begin{align}
&~\epsilon\chi_{\mathbb{R}^+}(x)e^{i(\lambda+i\epsilon)x}*m_e(\lambda+) \notag \\
= &~\epsilon\int_{-\infty}^{x}e^{i(\lambda+i\epsilon)(x-y)}m_e(y,\lambda+)~dy \notag\\
= &~e^{i(\lambda + i\epsilon)x}\epsilon \int_{-\infty}^xe^{\epsilon y}e^{-i\lambda y}m_e(y,\lambda+)~dy \notag \\
= &e^{i(\lambda + i\epsilon)x}\int_{-\infty}^{\epsilon x}e^y e^{-i\lambda y/\epsilon}m_e(y/\epsilon,\lambda+)~dy \label{eq: limt int 2}
\end{align}
By \cref{eq: me limit}, $m_e(x,\lambda+)$ is bounded on $(-\infty, |x|]$, and $e^{-i\lambda x}m_e(x,\lambda+) \to 1$ as $x\to -\infty$, therefore \cref{eq: limt int 2} tends to $e^{i\lambda x}$ for every $x$ by the dominated convergence theorem. \Cref{eq: me int} then follows as above.

We now prove that \cref{eq: m1 int} implies \cref{eq: m1 diff} and \cref{eq: m1 limit}. By the decomposition \cref{eq: G_lambda decomp}, we can write \cref{eq: m1 int} as
\eqn\label{eq: m1 int 1}
m_1(\lambda+) = 1 +ie^{i\lambda x}\int_{-\infty}^x e^{-i\lambda y}u(y)m_1(y,\lambda+)~dy - \widetilde{G}_{\lambda}*(um_1(\lambda+)).
\eeqn
Weakly differentiate \cref{eq: m1 int 1} to get
\begin{align} \label{eq: m1 diff 1}
\frac1i \partial_x m_1(\lambda+) =&~ i\lambda e^{i\lambda x}\int_{-\infty}^x e^{-i\lambda y}u(y)m_1(y,\lambda+)~dy + um_1(\lambda+)\notag\\
&\quad -\frac1i \partial_x\widetilde{G}_{\lambda}*(um_1(\lambda+)).
\end{align}
To compute $\frac1i \partial_x\widetilde{G}_{\lambda}*(um_1(\lambda+))$, we take its Fourier transform
\begin{align}
F\left(\frac1i \partial_x\widetilde{G}_{\lambda}*(um_1(\lambda+))\right) &= \xi\frac{\chi_{\real^-}}{\xi-\lambda}F(um_1(\lambda+))\notag\\
&= \chi_{\real^-}F(um_1(\lambda+)) + \lambda \frac{\chi_{\real^-}}{\xi-\lambda}F(um_1(\lambda+))\notag\\
&= F[C_-(um_1(\lambda+))+\lambda \widetilde{G}_{\lambda}*(um_1(\lambda+))].
\end{align}
All of the above steps can be justified using the fact that $um_1(\lambda+)\in L^p$. It follows that
\eqn
\frac1i \partial_x\widetilde{G}_{\lambda}*(um_1(\lambda+))=C_-(um_1(\lambda+))+\lambda \widetilde{G}_{\lambda}*(um_1(\lambda+)).
\eeqn
\Cref{eq: m1 diff 1} thus gives
\begin{align}
\frac1i\partial_x m_1(\lambda+) & = i\lambda e^{i\lambda x}\int_{-\infty}^x e^{-i\lambda y}um_1(y,\lambda+)~dy +um_1(\lambda+) -C_-(um_1(\lambda+))\notag\\
&\qquad - \lambda \widetilde{G}_{\lambda}*(um_1(\lambda+))\notag \\
&= C_+(um_1(\lambda+)) +\lambda G_{\lambda+0i}*(um_1(\lambda+))\notag \\
&= C_+(um_1(\lambda+)) +\lambda (m_1(\lambda+)-1). \label{eq: m1 diff 2}
\end{align}
To get the last step, we have used \cref{eq: m1 int} again. This proves \cref{eq: m1 diff}. To get \cref{eq: m1 limit}, we take the limit of \cref{eq: m1 int 1} as $x\to -\infty$. It suffices to show 
\eqn\label{eq: m1 limit 1}
\lim_{x\to -\infty}\widetilde{G}_{\lambda}*(um_1(\lambda+))(x)=0.
\eeqn
To see this, we write $\widetilde{G}_{\lambda}*(um_1(\lambda+))$ using the Fourier inversion formula as
\eqn
\widetilde{G}_{\lambda}*(um_1(\lambda+)) = F^{-1}(F(\widetilde{G}_{\lambda}*(um_1(\lambda+)))) = F^{-1}\left(\frac{\chi_{\real^-}}{\xi-\lambda}F(um_1(\lambda+))\right).
\eeqn
Recall that $\frac{\chi_{\real^-}(\xi)}{\xi-\lambda}F(um_1(\lambda+))\in L^1$, since $F(um_1(\lambda+))\in L^{p'}$ by the Hausdorff--Young inequality, and $\frac{\chi_{\real^-}}{\xi-\lambda}\in L^p$. Thus \cref{eq: m1 limit 1} follows by the Riemann-Lebesgue lemma. That \cref{eq: me int} implies \cref{eq: me diff} and \cref{eq: me limit} can be obtained in a similar way.
\end{proof}

To describe the Fredholm nature of the integral equations \cref{eq: m1 int} and \cref{eq: me int}, define
\begin{equation}\label{def: T lambda}
T_k\varphi=G_k*(u\varphi)
\end{equation}
for $k\in (\comp\setminus [0,\infty))\cup (\real^+\pm 0i)$.
The integral equations \cref{eq: m1 int} and \cref{eq: me int} are of the form $(I-T_k)\varphi = g$ where $g\in L^{\infty}$. The existence and uniqueness of Jost solutions follow from the invertibility of $I-T_k$ on suitable spaces. In the following, we first prove that $T_k$ are compact on certain weighted $L^{\infty}$ spaces. Thus the Fredholm alternative theorem will reduce the question to a vanishing lemma.

\begin{lem}\label{lem: compact}
Let $p>1$ and $s>s_1>1-\frac1p$ %$s\ge 0$ 
be given, and let %$u\in w^{-s}L^1(\real)\cap w^{-s}L^p(\real)$. 
$u\in L^p_s(\real)$. Then $T_k:L^{\infty}_{-(s-s_1)}(\real)\to L^{\infty}_{-(s-s_1)}(\real)$ is compact for every $k\in (\comp\setminus [0,\infty))\cup (\real^+\pm 0i)$.
\end{lem}
\begin{proof}
We only provide argument for $T_{\lambda+0i}$. The cases $T_{\lambda-0i}$ and $T_k$ for $k\in \comp\setminus [0,\infty)$ can obtained analogously. Let $\{\varphi_n\}$ be a bounded sequence in $L^{\infty}_{-(s-s_1)}$. Recall from the proof of \Cref{lem: equiv int diff} that $G_{\lambda+0i}\in L^{\infty}+L^{p'}$, and $u\varphi_n \in L^1\cap L^p$ with suitable estimates. Hence there exists $C_1 = C_1(u,\lambda, p,s_1)$ such that
\eqn\label{est: l infty T phi}
\|T_{\lambda+0i}\varphi_n\|_{\infty}\le C_1\|\varphi_n\|_{L^{\infty}_{-(s-s_1)}}.
\eeqn
Also, one can compute the weak derivative of $T_{\lambda+0i}\varphi_n$ as in \cref{eq: m1 diff 2} to get
\eqn
\frac{1}{i}\partial_x T_{\lambda+0i}\varphi_n = i\lambda e^{i\lambda x}\int_{-\infty}^x e^{-i\lambda y}u\varphi_n(y)~dy +u\varphi_n^- -C_-(u\varphi_n) - \lambda \widetilde{G}_{\lambda}*(u\varphi_n).
\eeqn
The four terms above are in $L^{\infty}$, $L^p$, $L^p$ and $L^{\infty}$ respectively. As a consequence, for every natural number $N$, there exists $C_2 = C_2(u,\lambda, p, s_1,N)$ such that
\eqn\label{est: l p d T phi}
\|\partial_x T_{\lambda+0i}\varphi_n\|_{L^p(-N,N)}\le C_2\|\varphi_n\|_{L^{\infty}_{-(s-s_1)}}.
\eeqn
From \cref{est: l infty T phi} and \cref{est: l p d T phi} we conclude that there exists $C=C(u,\lambda, p ,s_1,N)$ such that
\eqn
\|T_{\lambda+0i}\varphi_n\|_{W^{1,p}(-N,N)}\le C\|\varphi_n\|_{L^{\infty}_{-(s-s_1)}}.
\eeqn
By the Sobolev embedding thoerem, the sequence $\{T_{\lambda+0i}\varphi_n\}$ is uniformly bounded in every $C^{0,\frac{p-1}p}[-N,N]$, which is compactly embedded in $C^0[-N,N]$. By passing to a subsequence and a Cantor diagonal argument, we can assume that $\{T_{\lambda+0i}\varphi_n\}$ converges uniformly on any compact subset of $\mathbb{R}$ to a continuous function $f$. Obviously 
\eqn
\|f\|_{\infty}\le \|T_{\lambda+0i}\varphi_n\|_{\infty}\le C_1\sup_{n}\|\varphi_n\|_{L^{\infty}_{-(s-s_1)}}.
\eeqn 
Hence $f\in L^{\infty}_{-(s-s_1)}$.
For any $\epsilon>0$, choose $N$ large enough so that 
$$w^{s_1-s}(x)C_1\sup_{n}\|\varphi_n\|_{L^{\infty}_{-(s-s_1)}}< \frac{\epsilon}{2}$$ 
for $|x|> N$. We then have 
$$w^{s_1-s}(x)|f(x)-T_{\lambda+0i}\varphi_n(x)|<\epsilon$$
for $|x|>N$.
For $|x|\le N$, $T_{\lambda+0i}\varphi_n$ converges uniformly to $f$, we obviously have 
$$w^{s_1-s}(x)|f(x)-T_{\lambda+0i}\varphi_n(x)|<\epsilon$$
for $n$ sufficiently large. Therefore $T_{\lambda+0i}\varphi_n$ converges to $f$ in $L^{\infty}_{-(s-s_1)}$.
\end{proof}

By the Fredholm alternative theorem, what is left to show is that $(I-T_k)\varphi =0$ and $\varphi \in L^{\infty}_{-(s-s_1)}$ imply $\varphi=0$. We accomplish this in two steps. First we prove that, in suitable spaces, any such function $\varphi$ must be in $L^2$. After that an $L^2$ vanishing lemma will close the argument. In fact, we can prove the following decay estimate for functions in the kernel of $I-T_k$.

\begin{lem}\label{lem: decay estimate}
Suppose $s>s_1>\frac12$, $u\in L^2_s(\real)$ and $k\in (\comp\setminus[0,\infty))\cup(\real^+\pm 0i)$. If $\varphi\in L^{\infty}_{-(s-s_1)}(\real)$ and $\varphi = T_k\varphi$, then %$\varphi\in L^2$.
there exists $C=C(u,k, s, s_1)$ and $r=r(s)>\frac12$ such that
\eqn\label{est: decay}
|\varphi(x) |\le C[w(x)]^{-r}.
\eeqn 
In particular, $\varphi\in L^2(\real)$.
\end{lem}
\begin{proof}
We first assume $k\in \comp\setminus [0,\infty)$. In this case we have
\eqn\label{eq: G_k conv}
\varphi(x) = \int_{\real}G_k(x-y)u\varphi(y)~dy
\eeqn
where $G_k\in L^2$. As before, the conditions on $u$ and $\varphi$ imply $u\varphi \in L^1\cap L^2$. Hence $\varphi=G_k*(u\varphi)\in L^{\infty}$. To prove the decay estimate for $|x|$ large, we split the integral in \cref{eq: G_k conv} into two pieces: one on $\{|y-x|\le \frac{|x|}{2}\}$, and the other on $\{|y-x|> \frac{|x|}{2}\}$. For the former, we have
\begin{align}
&~\left|\int_{|y-x|\le \frac{|x|}{2}}G_k(x-y)u\varphi(y)~dy\right|\notag\\
\le &~C\|G_k\|_2\|\varphi\|_{\infty}\left(\int_{|y-x|\le \frac{|x|}{2}}w^{-2s}(y)w^{2s}(y)u^2(y)~dy\right)^{\frac12}\notag\\
\le &~C\|G_k\|_2w^{-s}\left(\frac{|x|}{2}\right)\|u\|_{L^2_s}\le C|x|^{-s}.
\end{align}
To estimate the other piece, we use the pointwise bound $|G_k(x)|\le \frac{C}{|x|}$ for $|x|>1$, which is an easy consequence of \cref{def: G} from integration by parts. Therefore
\eqn
\left|\int_{|y-x|> \frac{|x|}{2}}G_k(x-y)u\varphi(y)~dy\right|
\le C\frac{1}{|x|}\|u\varphi\|_1\le \frac{C}{|x|}\|u\|_{L^2_s}\|\varphi\|_{\infty}.
\eeqn
This completes the proof when $k\in \comp\setminus[0,\infty)$. Next, we study the case when $k\in \real^+\pm 0i$. Again, for simplicity, we work on $T_{\lambda+0i}$ only. Using \cref{eq: G_lambda decomp}, we have
\eqn\label{eq: decay 1}
\varphi(x) = T_{\lambda+0i}\varphi(x) = ie^{i\lambda x}\int_{-\infty}^x e^{-i\lambda y}u\varphi(y)~dy-\int_{\mathbb{R}}\widetilde{G}_{\lambda}(x-y)u\varphi(y)~dy .
\eeqn
By the same reason as above, $u\varphi\in L^1\cap L^2$, so we get $\varphi\in L^{\infty}$. The decay estimate for the last term in \cref{eq: decay 1} can be proved in a similar way as above, as $\widetilde{G}_{\lambda}\in L^2$. We now prove the decay estimate on the first integral in \cref{eq: decay 1}. To that end, we need the crucial identity \cref{eq: vanish int}, which follows from an integration calculation detailed in \Cref{lem: IBP iden}. Using that lemma, we have
\begin{align}
&~\langle \varphi, u\varphi\rangle 
=\langle T_{\lambda+0i}\varphi,u\varphi\rangle \notag\\
=&~\langle G_{\lambda + 0i}*u\varphi, u\varphi\rangle = i |\langle u\varphi, \e\rangle|^2 + \langle u\varphi, G_{\lambda+0i}*u\varphi\rangle \notag\\
=&~i|\langle u\varphi, \e\rangle|^2 + \langle u\varphi, T_{\lambda+0i}\varphi\rangle \notag\\
=&~i|\langle u\varphi, \e\rangle|^2 + \langle u\varphi, \varphi\rangle.
\end{align}
Therefore $\langle u\varphi, \e\rangle=0$, or
\begin{equation}\label{eq: vanish int}
\int_{\mathbb{R}}e^{-i\lambda x}u\varphi(x)~dx=0.
\end{equation}
Using \cref{eq: vanish int} on \cref{eq: decay 1}, we have
\begin{align}
\varphi(x) =T_{\lambda+0i}\varphi(x)&= i\int_{-\infty}^x e^{i\lambda(x-y)}u\varphi(y)~dy - \widetilde{G}_{\lambda}*(u\varphi) (x) \label{eq: left vanish}\\
&= -i\int_x^{\infty} e^{i\lambda(x-y)}u\varphi(y)~dy - \widetilde{G}_{\lambda}*(u\varphi) (x)\label{eq: right vanish}
\end{align}
Denote
\eqn\label{def: I(x)}
I(x) = i\int_{-\infty}^x e^{i\lambda(x-y)}u\varphi(y)~dy = -i\int_x^{\infty} e^{i\lambda(x-y)}u\varphi(y)~dy.
\eeqn
We want to show that $I(x)$ has $\frac{1}{|x|^r}$ decay at infinity for some $r>\frac12$. Let us now use the first expression in \cref{def: I(x)} to study the decay of $I(x)$ as $x$ tends to $-\infty$. Since $\varphi(x) = I(x) -\widetilde{G}_{\lambda}*(u\varphi)(x)$, we have
\eqn\label{eq: I bootstrap}
I(x) = i\int_{-\infty}^x e^{i\lambda (x-y)}u\varphi(y)~dy = i\int_{-\infty}^x e^{i\lambda (x-y)}u(y)[I(y)-\widetilde{G}_{\lambda}*(u\varphi)(y)]~dy.
\eeqn
Since $\widetilde{G}_{\lambda}*(u\varphi)\in L^2$ and $u\in L^2_s$, we get
$u\widetilde{G}_{\lambda}*(u\varphi)\in L^1_s$. Thus
\eqn\label{eq: left vanish 1}
\left|\int_{-\infty}^x e^{i\lambda(x- y)}u(y)\widetilde{G}_{\lambda}*(u\varphi)(y)~dy\right|\le C\|w^su\widetilde{G}_{\lambda}*(u\varphi)\|_{1}w^{-s}(x)\le Cw^{-s}(x)
\eeqn
for $x<0$, as $w(y)>w(x)$ for $y<x<0$. Since $s>\frac12$, \cref{eq: left vanish 1} already has the required decay as $x$ tends to $-\infty$. We next use \cref{eq: I bootstrap} to bootstrap decay estimates on $I(x)$. Recall that $I(x)\in L^{\infty}$. Suppose $I(x)\in L^{\infty}_r(-\infty, 0]$ for some $r\in [0,\frac12]$. We have
\begin{align}
\int_{-\infty}^x |u(y)I(y)|~dy &\le C\int_{-\infty}^x w^{-s-r}(y)|w^su(y)|~dy\notag\\
&= C\int_{-\infty}^x w^{-r-\frac12(s-\frac12)}(y)w^{-\frac12(s+\frac12)}(y)|w^su(y)|~dy\notag\\
&\le Cw^{-r-\frac12(s-\frac12)}(x) \int_{\real}w^{-\frac12(s+\frac12)}(y)|w^su(y)|~dy\notag\\
&\le Cw^{-(r+\frac12(s-\frac12))}(x) \|w^{-\frac12(s+\frac12)}\|_2\|u\|_{L^2_s}\label{eq: left vanish 2}
\end{align}
for $x<0$. 
By \cref{eq: I bootstrap}, \cref{eq: left vanish 1}, and \cref{eq: left vanish 2}, we get
$I(x) \in L^{\infty}_{r+\frac12(s-\frac12)}(-\infty,0]$, which has a little more decay than what we started with. Finitely many iterations of this argument will bring the decay exponent $r$ above $\frac12$. A similar argument using the second expression in \cref{def: I(x)} shows that $I(x)$ has the required decay as $x$ tends to $\infty$.
\end{proof}

The next result is the $L^2$ vanishing lemma alluded to in the previous discussion. It provides the key step for the proof of invertibility of $I-T_k$. It means, among other implications, that there is no embedded eigenvalues in the essential spectrum of $L_u$.

\begin{lem}\label{lem: crucial iden}
Suppose $s>s_1>\frac12$, $u\in L^2_s(\real)$ and $k\in (\comp\setminus[0,\infty))\cup(\real^+\pm 0i)$. If $\varphi\in L^{\infty}_{-(s-s_1)}(\real)$ and $\varphi = T_k\varphi$,
then
\begin{enumerate}
\item If $k\in\comp\setminus \real$, $\varphi=0$.
\item If $k\in (\real\setminus [0,\infty))\cup (\real^+\pm 0i)$,
\begin{equation}\label{eq: crucial iden}
\left|\int_{\mathbb{R}}u\varphi ~dx\right|^2=-2\pi k\int _{\mathbb{R}}|\varphi|^2~dx .
\end{equation}
\end{enumerate}
In particular, $\varphi = 0$ if $k\in \real^+\pm 0i$, or if $k$ is in the resolvent set of $L_u=\frac1i\partial_x-C_+uC_+$, regarded as an operator on $\hardy^+$.
\end{lem}

Identity \cref{eq: crucial iden} is reminiscent of Lemma 2.5 in \cite{wu2016simplicity}. However, the proof of \cref{eq: crucial iden} is much more delicate when $k=\lambda\pm 0i$, as $\lambda>0$ may introduce a singularity to $\hat{\varphi}$ and in particular make it non-differentiable at $\lambda$.
\begin{proof}
Let us first assume $k\in \comp\setminus\real$. By the same proof as in \Cref{lem: equiv int diff}, $\varphi=T_k\varphi$ implies
\eqn\label{eq: varphi diff complex case}
\frac1i\partial_x\varphi-C_+(u\varphi)=k\varphi,
\eeqn
and 
\eqn\label{eq: varphi fourier side}
\chi_{\real^+}\widehat{u\varphi}=(\xi-k)\hat\varphi.
\eeqn
Using \cref{eq: varphi fourier side} with the fact that $\varphi\in L^2$, proved in \Cref{lem: decay estimate}, we have $\varphi\in \hardy^+$. Thus $C_+\varphi=\varphi$. Multiply \cref{eq: varphi diff complex case} by $\bar{\varphi}$ and take the imaginary part to get
\eqn\label{eq: varphi diff complex case 1}
-\frac12 |\varphi|^2_x -\text{Im }(C_+(u\varphi)\bar{\varphi}) = (\text{Im }k)|\varphi|^2
.
\eeqn
Integrating \cref{eq: varphi diff complex case 1} and using the decay estimate \cref{est: decay} on $\varphi$, we get
\eqn\label{eq: varphi diff complex case 2}
0=-\text{Im }\int_{\real} u|\varphi|^2~dx = -\text{Im }\int_{\real}C_+(u\varphi)\bar{\varphi}~dx=  (\text{Im }k)\int_{\real}|\varphi|^2~dx.
\eeqn
To get the middle equality, we used the self-adjointness of $C_+$ and $C_+\varphi=\varphi$. \Cref{eq: varphi diff complex case 2} implies $\varphi=0$, as $k\in \comp\setminus\real$.

Next, if $k\in \real\setminus[0,\infty)$, we obtain \cref{eq: varphi diff complex case}, \cref{eq: varphi fourier side}, and $\varphi\in \hardy^+$ as before. In addition, \Cref{lem: decay estimate} gives $\varphi\in L^{\infty}_r$ for some $r>\frac12$. This together with the condition $u\in L^2_s$ for $s>\frac12$ imply $xu\varphi\in L^2$. Therefore $\widehat{u\varphi}\in H^1(\real)$. By the Sobolev embedding theorem, $\widehat{u\varphi}$ is continuous, and so is $\hat{\varphi}$ on $[0,\infty)$. Weakly differentiate \cref{eq: varphi fourier side} to get
\eqn
\widehat{u\varphi}'=\hat{\varphi}+(\xi-k)\hat{\varphi}'
\eeqn
when $\xi>0$.
Multiply by $\overline{\hat{\varphi}}$ and take the real part to get
\eqn
\text{Re }(\widehat{u\varphi}'\overline{\hat{\varphi}})=|\hat{\varphi}|^2+(\xi-k)\text{Re }(\varphi'\overline{\hat{\varphi}}),
\eeqn
or
\eqn
\text{Re }(\widehat{u\varphi}'\overline{\hat{\varphi}})= |\hat{\varphi}|^2 +\frac {\xi-k}2(|\hat{\varphi}|^2)'.
\eeqn
Now integrate between $0$ and $\infty$ to get
\eqn\label{eq: fourier k neg}
\text{Re }\int_{\real}\widehat{u\varphi}'\overline{\hat{\varphi}}~d\xi = \int_{\real}|\hat{\varphi}|^2~d\xi + \frac k2|\hat{\varphi}|^2(0+) -\frac12 \int_{\real}|\hat{\varphi}|^2~d\xi.
\eeqn
To obtain \cref{eq: fourier k neg}, we took the freedom to rewrite the integration domain as $\real$ whenever the integral involves $\hat{\varphi}$, a function supported on $[0,\infty)$, and have integrated the last term by part. To compute the boundary term for that step, we used the fact that $\lim_{\xi\to \infty}(\xi-k)|\hat{\varphi}|^2=0$, which is a consequence of \cref{eq: varphi fourier side} and the fact that $u\varphi \in L^1$. Now observe that the integral on the left hand side of \cref{eq: fourier k neg} is purely imaginary, by the Plancherel identity:
\eqn
\int_{\real}\widehat{u\varphi}'\overline{\hat{\varphi}}~d\xi = -2\pi i\int_{\real}xu|\varphi|^2~dx.
\eeqn
Hence \cref{eq: fourier k neg} gives
\eqn\label{eq: varphi L2}
2\pi \int_{\real}|\varphi|^2~dx = \int_{\real}|\hat{\varphi}|^2~d\xi = -k|\hat{\varphi}|^2(0+).
\eeqn
Using \cref{eq: varphi fourier side} to write $-k\hat{\varphi}(0+) = \widehat{u\varphi}(0)$, \cref{eq: crucial iden} follows.

Finally, assume $k=\lambda\pm0i$, where $\lambda\in \real^+$. Checking the signs of both sides of \cref{eq: crucial iden}, one easily sees that $\varphi=0$ is the only way to avoid a contradiction. Therefore the key is to prove \cref{eq: crucial iden}.
By the same proof as in \Cref{lem: equiv int diff}, $\varphi = T_{\lambda\pm 0i}\varphi$ implies
\begin{equation}\label{eq: diff eig}
\frac{1}{i}\partial_x\varphi-C_+(u\varphi)=\lambda \varphi,
\end{equation}
The Fourier transform of \cref{eq: diff eig} gives
\begin{equation}\label{eq: fourier eig}
\widehat{u\varphi}\chi_{\mathbb{R}^+} = (\xi-\lambda)\hat{\varphi}.
\end{equation}
This implies that $\varphi$ has its frequencies supported on $\real^+$, and thus belongs to $\hardy^+$.
Let $\psi_n^2$ be a smooth partition of unity on $(0,\infty)$:
\begin{equation}
\chi_{(0,\infty)}(\xi)=\sum_n\psi_n^2(\xi).
\end{equation}
Here the $\psi_n$'s are compactly supported smooth functions on $(0,\infty)$. An easy way to construct them is to make dyadic dilations of a fixed function. Let $P_n=F^{-1}\psi_nF$ be the Littlewood-Paley type projection associated with $\psi_n$. Letting $P_n$ act on \cref{eq: diff eig}, we have
\begin{equation}\label{eq: proj diff eig}
\frac{1}{i}(P_n\varphi)_x-P_n(u\varphi)=\lambda P_n\varphi.
\end{equation}
Multiply by $ix\overline{P_n\varphi}$ and take the real part to get
\begin{equation}
\text{Re } x(P_n\varphi)_x\overline{P_n\varphi} +\text{Im } xP_n(u\varphi)\overline{P_n\varphi}=0,
\end{equation}
or
\begin{equation}\label{eq: re im}
 \frac{1}{2}x\left(|P_n\varphi|^2\right)_x +\text{Im } xP_n(u\varphi)\overline{P_n\varphi}=0.
\end{equation}
We claim that $xP_n(u\varphi)\in L^2$. Indeed, 
\begin{align}
xP_n(u\varphi)(x) &= x\int_{\real}\check{\psi_n}(x-y)u(y)\varphi(y)~dy\notag\\
&= \int_{\real}(x-y)\check{\psi_n}(x-y)u(y)\varphi(y)~dy+\int_{\real}\check{\psi_n}(x-y)yu(y)\varphi(y)~dy\notag\\
&= (x\check{\psi_n})*(u\varphi)+ \check{\psi_n}*(xu\varphi).\label{eq: split schwartz}
\end{align}
Since $x\check{\psi_n}$ is in Schwartz class, and the conditions on $u$ and $\varphi$ imply $u\varphi\in L^2$, we conclude that the first term in \cref{eq: split schwartz} is in $L^2$. The second term is also in $L^2$ because $xu\varphi$ is, as is shown above. Integrate \cref{eq: re im} on $\mathbb{R}$ and use the Plancherel identity on the last term to get
\begin{equation}\label{eq: re im 2}
\frac{1}{2}x|P_n\varphi(x)|^2\bigg|_{-\infty}^{\infty}-\frac{1}{2}\int_{\mathbb{R}}|P_n\varphi|^2~dx +\frac{1}{2\pi}\text{Re }\int_0^{\infty}(\psi_n \widehat{u\varphi})'\psi_n\bar{\hat\varphi}~d\xi=0.
\end{equation}
We claim that the first term in \cref{eq: re im 2} vanishes. In fact, 
\begin{align}
|P_n\varphi(x)| &=\left| \int_{\real}\check{\psi_n}(x-y)\varphi(y)~dy\right|\notag\\
& \le \int_{|y-x|\le \frac{|x|}{2}}|\check{\psi_n}(x-y)\varphi(y)|~dy+ \int_{|y-x|> \frac{|x|}{2}}|\check{\psi_n}(x-y)\varphi(y)|~dy\notag\\
&\le \sup_{\frac{|x|}{2}\le |y|\le \frac{3|x|}{2}}|\varphi(y)|\|\check{\psi_n}\|_1+ \|\varphi\|_{\infty}\int_{|y|>\frac{|x|}{2}}|\check{\psi_n}(y)|~dy\notag\\
&\le C|x|^{-r}+C|x|^{-1}\notag
\end{align}
when $|x|$ is large. The last inequality above follows from estimate \cref{est: decay} and the fact that $\check{\psi_n}$ is in Schwartz class. We can now rewrite \cref{eq: re im 2} as
\begin{equation}
-\frac{1}{2}\int_{\mathbb{R}}|P_n\varphi|^2~dx +\frac{1}{2\pi}\text{Re }\int_0^{\infty}\left(\left(\frac{\psi_n^2}{2}\right)' \widehat{u\varphi}\bar{\hat\varphi}+\psi_n^2\widehat{u\varphi}'\bar{\hat\varphi}\right)~d\xi=0.
\end{equation}
Take the sum over $n$ to get
\begin{equation}\label{eq: sum re im 0}
-\frac{1}{2}\int_{\mathbb{R}}|\varphi|^2~dx + \frac{1}{2\pi}\text{Re }\sum_n\int_0^{\infty}\left(\frac{\psi_n^2}{2}\right)' \widehat{u\varphi}\bar{\hat\varphi}~d\xi+\frac{1}{2\pi}\text{Re }\int_{\real}\widehat{u\varphi}'\bar{\hat\varphi}~d\xi=0.
\end{equation}
The frequency integration domain of the first and last term in \cref{eq: sum re im 0} was changed from $\real^+$ to $\real$. This is allowed because $\varphi$ has frequencies supported on $\real^+$. The last integral in \cref{eq: sum re im 0} is purely imaginary, as can be seen by the Plancherel identity, hence the real part vanishes. Since $\sum_n\left(\frac{\psi_n^2}{2}\right)'=0$, where the sum is locally finite, we may insert into the second integral in \cref{eq: sum re im 0} a function $\chi$ that is compactly supported on $(0,\infty)$ and identically equal to 1 in a neighborhood of $\lambda$:
\begin{equation}\label{eq: sum re im}
-\frac{1}{2}\int_{\mathbb{R}}|\varphi|^2~dx + \frac{1}{2\pi}\text{Re }\sum_n\int_0^{\infty}\left(\frac{\psi_n^2}{2}\right)' \widehat{u\varphi}\bar{\hat\varphi}(1-\chi)~d\xi=0.
\end{equation}
Since $xu\varphi \in L^2$ as observed above, $\widehat{u\varphi}\in H^1$. Using \cref{eq: fourier eig} and the fact that $\chi=1$ in a neighborhood of $\lambda$, we get $\hat{\varphi}(1-\chi)\in H^1(0,\infty)$. Therefore we can integrate the second term in \cref{eq: sum re im} by parts and get
\begin{equation}
-\frac{1}{2}\int_{\mathbb{R}}|\varphi|^2~dx - \frac{1}{4\pi}\text{Re }\int_0^{\infty} \left(\widehat{u\varphi}\bar{\hat\varphi}(1-\chi)\right)'~d\xi=0,
\end{equation}
or
\begin{equation}\label{eq: sum re im 2}
-\frac{1}{2}\int_{\mathbb{R}}|\varphi|^2~dx + \frac{1}{4\pi}\text{Re }\widehat{u\varphi}(0)\bar{\hat\varphi}(0+)=0.
\end{equation}
The application of the fundamental theorem of calculus can be justified by the Sobolev embedding theorem. The fact that $\widehat{u\varphi}\bar{\hat{\varphi}}$ vanishes at infinity follows from $u\varphi \in L^1$, and \cref{eq: fourier eig}. \Cref{eq: fourier eig} also implies $\widehat{u\varphi}(0)=-\lambda\hat\varphi(0+)$. \Cref{eq: crucial iden} then follows from \cref{eq: sum re im 2}.
\end{proof}

By \cite{wu2016simplicity}, if $u\in L^{\infty}\cap L^2_s$ for some $s>\frac12$, then $L_u=\frac1i\partial_x-C_+uC_+$ has finitely many negative simple eigenvalues and the resolvent set has the form $\rho(L_u)=\mathbb{C}\setminus\{\lambda_1,\dots\lambda_N\}\setminus [0,\infty)$. In fact, \cite{wu2016simplicity} required $u\in L^{\infty}\cap L^2_1$, but the same result is true with the slightly weaker decay assumption, if one uses the same kind of bootstrap argument in the proof of \Cref{lem: decay estimate} to provide additional decay estimates on the eigenfunctions.

We are now ready to establish existence and uniqueness of the Jost solutions. 

\begin{thm}\label{thm: existence of Jost soln}
Let $s>s_1>\frac12$, and $u\in  L^2_s(\real)$. Let $\rho(L_u)$ be the resolvent set of $L_u=\frac1i\partial_x -C_+uC_+$, regarded as an operator on $\hardy^+$. Then for every $k\in \rho(L_u)\cup (\real^+\pm 0i)$, and every $\lambda>0$, there exists unique $m_1(x,k)$ and $m_e(x,\lambda\pm 0i)\in L^{\infty}_{-(s-s_1)}(\real)$ solving \cref{eq: m1 int} and \cref{eq: me int} respectively, with improved bounds $m_1(x,k), m_e(x,\lambda\pm 0i)\in L^{\infty}(\real)$. Furthermore, $k\mapsto m_1(k)$ is analytic from $\rho(L_u)$ to $L^{\infty}_{-(s-s_1)}(\real)$, and $m_1(k)\in C^{0,\gamma}_{loc}((\rho(L_u)\cup(\real^+\pm 0i)),L^{\infty}_{-(s-s_1)}(\real))$, while $m_e(\lambda\pm 0i)\in C^{0,\gamma}_{loc}((0,\infty),L^{\infty}_{-(s-s_1)}(\real)) $. Here $\gamma$ is some number between $0$ and $1$.
\end{thm}
\begin{proof}
\Cref{lem: compact}, \Cref{lem: crucial iden}, and the Fredholm alternative theorem imply existence and uniqueness of $m_1(x,k)$ and $m_e(x,\lambda\pm 0i)$. The improved $L^{\infty}$ bounds were proved in \Cref{lem: equiv int diff}. The analytic dependence of $m_1(k)$ on $k$ follows from the analytic dependence of $T_k$ on $k$. That in turn follows from the fact that $\frac1h(G_{k+h}-G_k)$ converges in $L^2$ to $-\frac{1}{2\pi}\int_0^{\infty}\frac{1}{(\xi-k)^2}~d\xi$, a result that is easy to see.
What is left to show is the H\"{o}lder continuity of $m_1(k)$ as $k$ approaches the positive half line from above and below, and of $m_e(\lambda\pm 0i)$. We write $m_1(k) = (I-T_k)^{-1}1$, and $m_e(\lambda\pm 0i) = (I-T_{\lambda\pm 0i})^{-1}\e(\lambda)$. Using the identity
\eqn
(I-T_{k+h})^{-1}-(I-T_{k})^{-1}=(I-T_k)^{-1}(T_{k+h}-T_{k})(I-T_{k+h})^{-1},
\eeqn
we reduce the problem to showing the following three points:
\begin{enumerate}[(a)]
\item $\e(\lambda)\in C^{0,\gamma}_{loc}((0,\infty),L^{\infty}_{-(s-s_1)}(\real)) $.
\item The $L^{\infty}_{-(s-s_1)}$ operator norm of $T_{k+h}-T_{k}$ is bounded by $C|h|^{\gamma}$ for fixed $k\in\real^+\pm 0i$ and small $h$.
\item The $L^{\infty}_{-(s-s_1)}$ operator norm of $(I-T_{k+h})^{-1}$ is uniformly bounded for fixed $k\in \real^+\pm 0i$ and small $h$.
\end{enumerate}
In the above, if $k=\lambda+0i$, then Im $h\ge 0$, while if $k=\lambda-0i$, then Im $h\le 0$.
To prove (a), we assume Im $h=0$, and estimate
\begin{align}
|w^{s_1-s}(x)(e^{i(\lambda+h)x}-e^{i\lambda x})|&= |w^{s_1-s}(x)(e^{ihx}-1)|\notag\\
&\le w^{s_1-s}(x)\min\{|hx|,2\}\label{eq: e lambda 1}
\end{align}
If $|x|<1/\sqrt h$, \cref{eq: e lambda 1} is bounded by $|hx|\le \sqrt{h}$. If $|x|\ge 1/\sqrt h$, \cref{eq: e lambda 1} is bounded by $2w^{s_1-s}(1/\sqrt h)\le Ch^{\frac{s-s_1}{2}}$. Hence $\|\e(\lambda+h)-\e(\lambda)\|_{L^{\infty}_{-(s-s_1)}}\le Ch^{p_1}$ for $p_1=\min(\frac{s-s_1}{2},\frac12)$. This proves (a). Notice that (b) implies (c), as $I-T_{k}$ is already shown to be invertible. For simplicity, we only work on $T_{\lambda+0i+h}-T_{\lambda+0i}$ with Im $h\ge 0$. To prove (b), we estimate
\begin{align}
&~(T_{\lambda+0i+h}-T_{\lambda+0i})\psi(x)\notag\\
= &~i\int_{-\infty}^xe^{i\lambda(x-y)}[e^{ih(x-y)}-1]u(y)\psi(y)~dy - (\widetilde{G}_{\lambda+h}-\widetilde{G}_{\lambda})*(u\psi)(x)
\end{align}
Since $u\in L^2_s$, and $\psi\in L^{\infty}_{-(s-s_1)}$, we can write $u=w^{-s}u_1$, with $\|u_1\|_2= \|u\|_{L^2_s}$, and $\psi= w^{s-s_1}\psi_1$, with $\|\psi_1\|_{\infty}=\|\psi\|_{L^{\infty}_{-(s-s_1)}}$. Hence 
\eqn
u\psi = w^{-s_1}u_1\psi_1=w^{-p_2}w^{-\frac12-p_2}u_1\psi_1,
\eeqn 
where $p_2=\frac12(s_1-\frac12)>0$. Since $w^{-\frac12-p_2}\in L^2$, we get $u_2=w^{-\frac12-p_2}u_1\in L^1$. Notice that $|e^{ih(x-y)}-1|\le C|h(x-y)|$ when $|h(x-y)|\le 1$, and $|e^{ih(x-y)}-1|\le 2$, since $\text{Im }h\ge 0$ and $x-y\ge 0$. Therefore
\begin{align}
&~|w^{s_1-s}(x)(T_{\lambda+0i+h}-T_{\lambda+0i})\psi(x)|\notag\\
\le &~C\bigg(w^{s_1-s}(x)\int_{-\infty}^x\min\{|h(x-y)|,2\}w^{-p_2}(y)|u_2(y)|~dy\notag\\
&\qquad+\|\widetilde{G}_{\lambda+h}-\widetilde{G}_{\lambda}\|_2\|u_1\|_2\bigg)\|\psi\|_{L^{\infty}_{-(s-s_1)}}\notag\\
\le &~C\bigg(\int_{-\infty}^x\min\{|h(x-y)|,2\}w^{-p_3}(x-y)|u_2(y)|~dy+|h|\bigg)\|\psi\|_{L^{\infty}_{-(s-s_1)}}
\end{align}
for $p_3=\min(s-s_1,p_2)$. Here we have used the Plancherel identity to estimate $\|\widetilde{G}_{\lambda+h}-\widetilde{G}_{\lambda}\|_2$, and have used the elementary inequality $w(x-y)\le w(x)w(y)$. The term $(\min\{|h(x-y)|,2\}w^{-p_3}(x-y))$ can be estimated as follows. If $|x-y|<1/\sqrt{|h|}$, $$\min\{|h(x-y)|,2\}w^{-p_3}(x-y)\le |h(x-y)|\le \sqrt{|h|}.$$
On the other hand, if $|x-y|\ge 1/\sqrt{|h|}$,
$$\min\{|h(x-y)|,2\}w^{-p_3}(x-y)\le 2w^{-p_3}(1/\sqrt {|h|})\le C|h|^{\frac{p_3}{2}}$$
for $h$ small. Therefore 
\eqn
|w^{s_1-s}(x)(T_{\lambda+0i+h}-T_{\lambda+0i})\psi(x)|\le C|h|^{p_4}\|\psi\|_{L^{\infty}_{-(s-s_1)}}
\eeqn
for $p_4=\min(\frac{p_3}2,\frac12)$. This proves (b) with $\gamma=p_4$. 
\end{proof}

%%%%%%%%%%%%%%%%%%%%%%%%%%%%%%%%%%%%%%%%%%%%%%%%%%%%%%%%%%%%%%%%%%%%%%
\section{Scattering coefficients between Jost solutions}\label{sec: scattering data}

Now that the Jost solutions are obtained, we may proceed to study relations between them that give rise to the scattering coefficients of the Fokas--Ablowitz IST. Such relations were obtained formally by Fokas and Ablowitz in \cite{fokas1983inverse}. In addition, there are also relations between different scattering coefficients, many of which are stated in \cite{kaup1998inverse}. However, the arguments used in \cite{kaup1998inverse} are formal as well and depend on certain identities involving the inverse scattering problem. Here we will prove these relations and construct the scattering data directly using the setup in \Cref{sec: existence}. 

First, we want to establish differentiability with respect to $\lambda$ of the function $\overline{\pmb{e}}(\lambda)m_e(\lambda\pm 0i)$. The $\lambda$ derivative of $\overline{\pmb{e}}(\lambda)m_e(\lambda\pm 0i)$ will help produce an important scattering coefficient. In fact, we will show that $\overline{\pmb{e}}(\lambda)m_e(\lambda\pm 0i)$ is differentiable as a map into the weighted $L^{\infty}$ spaces used in \Cref{sec: existence}. It is curious that differentiability of the particular combination $\overline{\pmb{e}}m_e$ can be proven under the same decay assumptions on $u$ as in \Cref{sec: existence}, whereas any slightly different function, such as $m_e(\lambda\pm 0i)$ alone, $m_1(\lambda\pm 0i)$, or $\overline{\pmb{e}}(\lambda)m_1(\lambda\pm 0i)$, will require significantly stronger decay conditions on $u$ to be differentiable in the above sense. The basic reason is that the term $\partial_{\lambda}e^{i\lambda x} = xe^{i\lambda x} $ comes out when we differentiate \cref{eq: m1 int} and \cref{eq: me int} with respect to $\lambda$. We would need $xe^{i\lambda x} $ to belong to $L^{\infty}_{-(s-s_1)}$, which forces $s-s_1\ge 1$. Combined with the condition $s>s_1>\frac12$ in \Cref{thm: existence of Jost soln}, this implies $s>\frac32$. Certain other considerations seem to even require $s>\frac52$. The special favor found only by $\overline{\pmb{e}}(\lambda)m_e(\lambda\pm 0i)$ can be explained as follows. To simplify notation, we suppress the $\lambda$ dependence when it is clear from the context. By \cref{eq: G_lambda decomp} and \cref{def: tilde G}, we formally have
\begin{equation}
\partial_{\lambda}G_{\lambda\pm 0i}(x)=-\frac{1}{2\pi\lambda}+ixG_{\lambda \pm 0i}(x).
\end{equation}
Rewrite \cref{eq: me int} as
\begin{equation}\label{eq: e bar me int}
\overline{\pmb e}m_e(\lambda\pm) = 1+\overline{\pmb e} G_{\lambda \pm 0i}*(u\pmb e\overline{\pmb e} m_e(\lambda\pm)),
\end{equation}
and differentiate with respect to $\lambda$ formally:
\begin{align}
&~\partial_{\lambda}(\overline{\pmb e}m_e(\lambda\pm)) \notag\\
= &~-ix\overline{\pmb e} \left(G_{\lambda \pm 0i}*(u\pmb e\overline{\pmb e}m_e(\lambda\pm))\right)+\overline{\pmb e} \left(-\frac{1}{2\pi \lambda}+ixG_{\lambda \pm 0i}\right)*(u\pmb e\overline{\pmb e} m_e(\lambda\pm))\notag \\
&~+\overline{\pmb e}G_{\lambda \pm 0i}*(iyu\pmb e\overline{\pmb e} m_e(\lambda\pm))+\overline{\pmb e}G_{\lambda \pm 0i}*(u\pmb e\partial_{\lambda}(\overline{\pmb e} m_e(\lambda\pm)))\notag \\
= &~ -\frac{\overline{\pmb e}}{2\pi\lambda}\int_{\mathbb{R}}u(y)m_e(y,\lambda\pm)~dy+\overline{\pmb e}G_{\lambda \pm 0i}*(u\pmb e\partial_{\lambda}(\overline{\pmb e} m_e(\lambda\pm))). \label{eq: diff lambda me}
\end{align}
Multiply both sides of \cref{eq: diff lambda me} by $\pmb e$ to get
\begin{equation}
\pmb e\partial_{\lambda}(\overline{\pmb e}m_e(\lambda\pm))=-\frac{1}{2\pi\lambda}\int_{\mathbb{R}}u(y)m_e(y,\lambda\pm)~dy+G_{\lambda \pm 0i}*(u\pmb e\partial_{\lambda}(\overline{\pmb e} m_e(\lambda\pm))).\label{eq: diff lambda me 2}
\end{equation}
Notice \cref{eq: diff lambda me 2} no longer involves any extra factor of $x$. In fact, the cancelation happening in \cref{eq: diff lambda me} removed all extra factors of $x$. The proof of differentiability is to show that a similar, although no longer exact, cancelation happens on the level of difference quotients.

\begin{lem}\label{lem: me differentiable}
Let $s>s_1>\frac12$, and $u\in L^2_s(\real)$. Let $m_e(\lambda\pm 0i)$ be the Jost functions constructed in \Cref{thm: existence of Jost soln}. Then $\overline{\pmb e}(\lambda) m_e(\lambda\pm 0i)\in C^{1,\gamma}_{loc}((0,\infty), L^{\infty}_{-(s-s_1)}(\real))$ for some $0<\gamma<1$, and 
\eqn\label{eq: diff me m1 rel}
\partial_{\lambda}(\overline{\pmb e}(\lambda)m_e(\lambda\pm 0i)) = -\frac{\eb}{2\pi \lambda}\left(\int_{\real}u(y)m_e(y,\lambda\pm 0i)~dy\right)m_1(\lambda\pm 0i).
\eeqn
\end{lem}
\begin{proof}
We denote the shift operator by $(\tau_h f)(\lambda) = f(\lambda + h)$, and the difference quotient operator by $D_h f= \frac1h(\tau_h f-f)$. One has the product rule:
\eqn
D_h (fg) = (D_h f)g + (\tau_h f)(D_h g).
\eeqn
For simplicity, we only work on $m_e(\lambda-0i)$ and write it simply as $m_e$.
$D_h$ acting on \cref{eq: e bar me int} gives
\begin{align}
D_h(\overline{\pmb e}m_e) &= (D_h \eb)\left(G_{\lambda - 0i}*(u\e\eb m_e)\right)+(\tau_h \eb)[(D_hG_{\lambda-0i})*(u\e\eb m_e)]\notag\\
&\quad + (\tau_h\eb)\left[(\tau_h G_{\lambda-0i})*\left(u(D_h\e)\eb m_e\right)\right]\notag\\
&\quad +(\tau_h \eb)\left[(\tau_h G_{\lambda-0i})*\left(u(\tau_h \e)D_h(\eb m_e)\right)\right].\label{eq: diff quo 1}
\end{align}
We add up the first three terms in \cref{eq: diff quo 1} as follows:
\begin{align}
&~\frac{e^{-i(\lambda+h)x}-e^{-i\lambda x}}{h}\int_{\real}G_{\lambda-0i}(x-y)um_e(y)~dy \notag\\
&\quad + e^{-i(\lambda+h)x}\int_{\real}\frac{G_{\lambda+h-0i}(x-y)-G_{\lambda-0i}(x-y)}{h}um_e(y)~dy\notag\\
&\quad + e^{-i(\lambda+h)x}\int_{\real}G_{\lambda+h-0i}(x-y)\frac{e^{i(\lambda+h)y}-e^{i\lambda y}}{h}e^{-i\lambda y}um_e(y)~dy\notag\\
=&~ e^{-i(\lambda+h)x}\int_{\real}\frac{G_{\lambda+h-0i}(x-y)-G_{\lambda-0i}(x-y)e^{ih(x-y)}}{h}e^{ihy}um_e(y)~dy\notag\\
=&~\frac{e^{-i(\lambda+h)x}}{2\pi}\int_{\real}\frac1h\left( \int_0^h \frac{e^{i(x-y)\xi}}{\xi-\lambda-h}~d\xi \right) e^{ihy}um_e(y)~dy.\label{eq: diff quo 2}
\end{align}
The last equality above follows from \cref{eq: G_lambda decomp} and \cref{def: tilde G}. Denote \cref{eq: diff quo 2} by $S_h(um_e)$. We get from \cref{eq: diff quo 1} that
\eqn\label{eq: diff quo 3}
\e D_h(\eb m_e) = \e S_h(um_e) + \e(\tau_h \eb)\left[(\tau_h G_{\lambda-0i})*\left(u(\tau_h \e)\eb \e D_h(\eb m_e)\right)\right].
\eeqn
Let $\varphi$ be the solution to 
\eqn\label{eq: diff quo 4}
\varphi = -\frac{1}{2\pi \lambda}\int_{\real}um_e(y)~dy + G_{\lambda-0i}*(u\varphi)
\eeqn
whose existence is guaranteed by \Cref{thm: existence of Jost soln}. We want to show $\partial_{\lambda} (\eb m_e) = \eb\varphi$, which, as we will see in the following, is equivalent to \cref{eq: diff me m1 rel}. To that end,
take the difference of \cref{eq: diff quo 3} and \cref{eq: diff quo 4} and rearrange terms to get
\begin{align}
&~[\e D_h(\eb m_e)-\varphi] - \e(\tau_h \eb)\left[(\tau_h G_{\lambda-0i})*\left(u(\tau_h \e)\eb [\e D_h(\eb m_e)-\varphi]\right)\right]\notag\\
=&~ \e S_h(um_e) +\frac{1}{2\pi \lambda}\int_{\real}um_e(y)~dy\notag\\
&~~ +\e(\tau_h \eb)\left[(\tau_h G_{\lambda-0i})*\left(u(\tau_h \e)\eb \varphi\right)\right]-G_{\lambda-0i}*(u\varphi).\label{eq: diff quo 5}
\end{align}
Denote $T_{\lambda, h}\psi= \e(\tau_h \eb)\left[(\tau_h G_{\lambda-0i})*\left(u(\tau_h \e)\eb\psi\right)\right]$, and recall the definition of $T_{\lambda-0i}$ by \cref{def: T lambda}, \cref{eq: diff quo 5} can be written as
\begin{align}
&~(I-T_{\lambda,h})[\e D_h(\eb m_e)-\varphi] \notag\\
= &~\e S_h(um_e) +\frac{1}{2\pi \lambda}\int_{\real}um_e(y)~dy + (T_{\lambda,h}-T_{\lambda-0i})\varphi. \label{eq: diff quo 6}
\end{align}
In view of \cref{eq: diff quo 6}, it suffices to show the following three points:
\begin{enumerate}[(a)]
\item For $\lambda>0$ fixed, $(I-T_{\lambda,h})^{-1}$ has uniformly bounded $L^{\infty}_{-(s-s_1)}$ operator norm for small $h$.
\item $\|\e S_h(um_e) +\frac{1}{2\pi \lambda}\int_{\real}um_e(y)~dy\|_{L^{\infty}_{-(s-s_1)}}\to 0$ as $h\to 0$.
\item $\|(T_{\lambda,h}-T_{\lambda-0i})\varphi\|_{L^{\infty}_{-(s-s_1)}}\to 0$ as $h\to 0$.
\end{enumerate}
In fact, we claim that the $L^{\infty}_{-(s-s_1)}$ operator norm of $T_{\lambda,h}-T_{\lambda-0i}$ tends to $0$ as $h$ tends to $0$. This will imply both (a) and (c), as $(I-T_{\lambda-0i})$ is invertible by \Cref{thm: existence of Jost soln}. We write 
\begin{align}
&~(T_{\lambda,h}-T_{\lambda-0i})\psi \notag\\
= &~ \e(\tau_h \eb)\left[(\tau_h G_{\lambda-0i})*\left(u(\tau_h \e)\eb \psi\right)\right]-G_{\lambda-0i}*(u\psi)\notag\\
=&~[\e(\tau_h \eb)-1]\left[(\tau_h G_{\lambda-0i})*\left(u(\tau_h \e)\eb \psi\right)\right]\notag\\
&~+(\tau_h G_{\lambda-0i}-G_{\lambda-0i})*\left(u(\tau_h \e)\eb \psi\right)\notag\\
&~+G_{\lambda-0i}*\left(u[(\tau_h \e)\eb-1] \psi\right)\notag\\
= &~I(x) + II(x) + III(x). \label{eq: I II III}
\end{align}
We estimate the three terms separately. By \cref{eq: G_lambda decomp} and \cref{def: tilde G}, there exists $C=C(\lambda,u)$ such that
\begin{align}
|w^{s_1-s}(x)I(x)|&\le Cw^{s_1-s}(x)|e^{-ihx}-1|\left(\|u\psi\|_{\infty}+\|\widetilde{G}_{\lambda+h}\|_2\|u\psi\|_2\right)\notag\\
&\le Cw^{s_1-s}(x)\min\{|hx|,2\}\|\psi\|_{L^{\infty}_{-(s-s_1)}}
\end{align}
If $|x|<1/\sqrt{h}$, $|hx|<\sqrt h$. On the other hand if $|x|>1/\sqrt h$, $$w^{s_1-s}(x)\le w^{s_1-s}(1/\sqrt h) \le Ch^{\frac{s-s_1}{2}}.$$
Therefore
\eqn\label{est: I(x)}
|w^{s_1-s}(x)I(x)|\le Ch^{p_1}\|\psi\|_{L^{\infty}_{-(s-s_1)}},
\eeqn
where $p_1=\min(\frac{s-s_1}{2}, \frac12)$. Let $p_2 = \frac12(s_1-\frac12)>0$. By the conditions on $u$ and $\psi$, we have 
\begin{align}
|u(y)\psi(y)|&\le w^{-s_1}(y)u_1(y)\|\psi\|_{L^{\infty}_{-(s-s_1)}} \notag\\
&= w^{-p_2}(y)w^{-\frac12-p_2}u_1(y)\|\psi\|_{L^{\infty}_{-(s-s_1)}} \notag\\
&=w^{-p_2}(y)u_2(y)\|\psi\|_{L^{\infty}_{-(s-s_1)}},
\end{align}
where $u_1 = w^{s}u\in L^2$, and $u_2=w^{-\frac12-p_2}u_1\in L^1 $.
Letting $p_3=\min(s-s_1, p_2)$, and using the relation $w(x-y)\le Cw(x)w(y)$, we have
\begin{align}
&~|w^{s_1-s}(x)II(x)|\notag\\
\le&~ Cw^{s_1-s}(x)\left(\int_x^{\infty}|e^{ih(x-y)}-1||u(y)\psi(y)|~dy + \|\widetilde{G}_{\lambda+h}-\widetilde{G}_{\lambda}\|_2\|u\psi\|_2\right)\notag\\
\le &~Cw^{s_1-s}(x)\left(\int_x^{\infty}\min\{|h(x-y)|,2\}w^{-p_2}(y)|u_2(y)|~dy + h\right)\|\psi\|_{L^{\infty}_{-(s-s_1)}}\notag\\
\le &~ C\left(\int_x^{\infty}\min\{|h(x-y)|,2\}w^{-p_3}(x-y)|u_2(y)|~dy + h\right)\|\psi\|_{L^{\infty}_{-(s-s_1)}}.
\end{align}
By the same argument as above, we get $\min\{|h(x-y)|,2\}w^{-p_3}(x-y)\le Ch^{p_4}$ for $p_4=\min(\frac{p_3}{2}, \frac12)$. Hence
\eqn\label{est: II(x)}
|w^{s_1-s}(x)II(x)|\le Ch^{p_4}\|\psi\|_{L^{\infty}_{-(s-s_1)}}.
\eeqn
Similarly, 
\begin{align}
&~|w^{s_1-s}(x)III(x)| \le |III(x)|\notag\\
\le &~C\left(\int_{\real}|u(y)\psi(y)||e^{ihy}-1|~dy+ \|\widetilde{G}_{\lambda}\|_2\left(\int_{\real}|u(y)\psi(y)|^2|e^{ihy}-1|^2\right)^{1/2} \right)\notag\\
\le &~C\bigg(\int_{\real}w^{-p_2}(y)u_2(y)\min\{|hy|,2\}~dy \notag\\
&\qquad + \left(\int_{\real}|w^{-s_1}(y)u_1(y)|^2(\min\{|hy|,2\})^2~dy\right)^{1/2}\bigg)\|\psi\|_{L^{\infty}_{-(s-s_1)}}\notag\\
\le&~ Ch^{p_4}\|\psi\|_{L^{\infty}_{-(s-s_1)}}. \label{est: III(x)}
\end{align}
By \cref{eq: I II III}, \cref{est: I(x)}, \cref{est: II(x)}, and \cref{est: III(x)}, we have
\eqn
\|(T_{\lambda,h}-T_{\lambda-0i})\psi \|_{L^{\infty}_{-(s-s_1)}}\le Ch^{p_4}\|\psi\|_{L^{\infty}_{-(s-s_1)}}. 
\eeqn
This proved points (a) and (c) mentioned above. To prove (b), we recall that $S_h(um_e)$ was defined by \cref{eq: diff quo 2}. So
\begin{align}
&~\left(\e S_h(um_e)\right)(x)+\frac{1}{2\pi\lambda}\int_{\real}um_e(y)~dy\notag\\
= &~\frac{1}{2\pi}\int_{\real}\left[\frac1h\left( \int_0^h \frac{e^{i(x-y)(\xi-h)}}{\xi-\lambda-h}~d\xi \right) +\frac1{\lambda}\right]um_e(y)~dy\notag\\
= &~\frac{1}{2\pi}\int_{\real}\frac1h\left( \int_0^h \frac{\lambda [e^{i(x-y)(\xi-h)}-1]+\xi-h}{\lambda(\xi-h-\lambda)}~d\xi \right) um_e(y)~dy.
\end{align}
Hence
\begin{align}
&~\left|w^{s_1-s}(x)\left(\left(\e S_h(um_e)\right)(x)+\frac{1}{2\pi\lambda}\int_{\real}um_e(y)~dy\right)\right|\notag\\
\le &~ Cw^{s_1-s}(x)\int_{\real}\left(\min\{|2h(x-y)|,2\}+h\right)|u(y)m_e(y)|~dy\notag\\
\le &~Cw^{s_1-s}(x)\int_{\real}\left(\min\{|2h(x-y)|,2\}+h\right)w^{-p_2}(y)|u_2(y)|~dy\notag\\
\le &~C\int_{\real}\left(\min\{|2h(x-y)|,2\}+h\right)w^{-p_3}(x-y)|u_2(y)|~dy\notag\\
\le &~Ch^{p_4}.
\end{align}
We have thus proved $\partial_{\lambda}(\eb m_e)=\eb \varphi$. Since $\varphi$ satisfies \cref{eq: diff quo 4},  $\partial_{\lambda}(\eb m_e)$ satisfies \cref{eq: diff lambda me 2}. In other words, 
\begin{align}
\partial_{\lambda}(\eb m_e(\lambda-0i))& = \eb(I-T_{\lambda-0i})^{-1}\left(-\frac{1}{2\pi\lambda}\int_{\mathbb{R}}um_e(y,\lambda-0i)~dy\right)\notag\\
&=-\frac{\eb}{2\pi \lambda}\left(\int_{\real}u(y)m_e(y,\lambda-0i)~dy\right)m_1(\lambda-0i).
\end{align} 
Its H\"older continuity follows from that of $\e(\lambda)$, $m_e(\lambda-0i)$ and $m_1(\lambda-0i)$, which was established in the proof of \Cref{thm: existence of Jost soln}. 
\end{proof}

We are now ready to described the scattering coefficients for the Fokas--Ablowtiz IST. 

\begin{lem}\label{lem: def scattering coeff}
Let $m_1(x,k)$ and $m_e(x,\lambda\pm 0i)$ be the Jost solutions constructed in \Cref{thm: existence of Jost soln}. Define
\begin{align}\label{def: Gamma}
\Gamma(\lambda) &= 1+i\int_{\real}u(x)m_e(x,\lambda+0i)e^{-i\lambda x}~dx \\
&= \frac{1}{1-i\int_{\real}u(x)m_e(x,\lambda-0i)e^{-i\lambda x}~dx},
\end{align}
\eqn\label{def: beta}
\beta(\lambda) = i\int_{\real}u(x)m_1(x,\lambda+0i)e^{-i\lambda x}~dx,
\eeqn
and
\eqn\label{def: f}
f(\lambda) = -\frac{1}{2\pi \lambda}\int_{\real}u(x)m_e(x,\lambda- 0i)~dx.
\eeqn
%and
%\eqn
%g(\lambda) = -\frac{1}{2\pi \lambda}\int_{\real}u(x)m_e(x,\lambda+ 0i)~dx.
%\eeqn
Then the following relations between Jost solutions hold:
\eqn\label{eq: me pm rel}
m_e(\lambda+0i)=\Gamma(\lambda)m_e(\lambda-0i),
\eeqn
\eqn\label{eq: m1 jump}
m_1(\lambda+0i)-m_1(\lambda-0i) = \beta(\lambda) m_e(\lambda-0i),
\eeqn
and
\eqn\label{eq: diff me m1}
\e\partial_{\lambda}(\eb m_e(\lambda-0i))=f(\lambda)m_1(\lambda-0i).
\eeqn
%and
%\eqn
%\e\partial_{\lambda}(\eb m_e(\lambda+0i))=f(\lambda)\Gamma(\lambda)m_1(\lambda+0i),
%\eeqn
\end{lem}
\begin{proof}
By \cref{eq: G_lambda decomp},
\eqn
G_{\lambda+0i}(x)-G_{\lambda-0i}(x)=i\e(x,\lambda).
\eeqn
Therefore \cref{eq: me int} implies
\begin{align}
m_e(\lambda+ 0i) &= \e(\lambda)+(G_{\lambda-0i}+i\e(\lambda))*(um_e(\lambda+0i))\notag\\
&= \e(\lambda) + G_{\lambda-0i}*(um_e(\lambda-0i))+ i\e(\lambda)\int_{\real}u(x)m_e(x,\lambda+0i)e^{-i\lambda x}~dx\notag\\
&=\left(1+i\int_{\real}u(x)m_e(x,\lambda+0i)e^{-i\lambda x}~dx\right)\e(\lambda)+G_{\lambda-0i}*(um_e(\lambda-0i)).
\end{align}
By \cref{eq: me int} and uniqueness of Jost solutions, we get
\eqn
m_e(\lambda+ 0i) =\left(1+i\int_{\real}u(x)m_e(x,\lambda+0i)e^{-i\lambda x}~dx\right)m_e(\lambda- 0i).
\eeqn
A similar calculation starting with the integral equation of $m_e(\lambda-0i)$ gives
\eqn
m_e(\lambda- 0i) =\left(1-i\int_{\real}u(x)m_e(x,\lambda-0i)e^{-i\lambda x}~dx\right)m_e(\lambda+ 0i).
\eeqn
This proves \cref{eq: me pm rel}. Take the difference of the integral equations of $m_1(\lambda\pm 0i)$ given in \cref{eq: m1 int} to get
\begin{align}
m_1(\lambda+0i)-m_1(\lambda-0i)=&~i\e(\lambda)\int_{\real}u(x)m_1(x,\lambda+0i)e^{-i\lambda x}~dx\notag\\
&\quad +G_{\lambda-0i}*[u(m_1(\lambda+0i)-m(\lambda-0i))].
\end{align}
By \cref{eq: me int} and uniqueness of Jost solutions, we get
\eqn
m_1(\lambda+0i)-m_1(\lambda-0i) = m_e(\lambda-0i)~i\int_{\real}u(x)m_1(x,\lambda+0i)e^{-i\lambda x}~dx .
\eeqn
This proves \cref{eq: m1 jump}. Finally, \cref{eq: diff me m1} is just the minus sign case of \cref{eq: diff me m1 rel}.
\end{proof}

Our next goal is to establish relations between different scattering coefficients. The following identity proves very useful in showing these relations.

\begin{lem}\label{lem: IBP iden}
Denote $\int_{\real}f(x)\overline{g(x)}~dx$ by $\langle f, g\rangle$. If $f,g\in L^1(\real)\cap L^2(\real)$, $\lambda>0$, then
\eqn
\langle G_{\lambda\pm 0i}*f, g\rangle = i\langle f, \e\rangle \overline{\langle  g, \e\rangle} + \langle f, G_{\lambda\pm 0i}*g\rangle.
\eeqn
\end{lem} 
\begin{proof}
We only present the calculation for $G_{\lambda+0i}$. Using \cref{eq: G_lambda decomp}, we see that
\eqn
\overline{G_{\lambda+0i}(-x)}=-ie^{i\lambda x}\chi_{\real^-}(x)-\tilde{G}_{\lambda}(x),
\eeqn
or 
\eqn
G_{\lambda+0i}(x) = \overline{G_{\lambda+0i}(-x)} + ie^{i\lambda x}.
\eeqn
Thus
\begin{align}
&~\langle G_{\lambda+ 0i}*f, g\rangle
=\int_{\mathbb{R}}\int_{\real}G_{\lambda+0i}(x-y)f(y)\overline{g(x)}~dy~dx \notag\\
=&~\int_{\mathbb{R}}\int_\real \left(\overline{G_{\lambda+0i}(y-x)}+ie^{i\lambda(x-y)}\right)f(y)\overline{g(x)}~dy~dx\notag\\
=&~i\int_\real f(y)e^{-i\lambda y}~dy \int_\real \overline{g(x)e^{-i\lambda x}}~dx + \int_\real\int_\real f(y) \overline{G_{\lambda+0i}(y-x)g(x)}~dx~dy\notag\\
=&~i\langle f, \e\rangle \overline{\langle  g, \e\rangle} + \langle f, G_{\lambda+ 0i}*g\rangle.
\end{align}

All of the above calculations are justified if $f,g\in L^1\cap L^2$.
\end{proof}

\begin{lem}\label{lem: rel scattering coeff}
Let $s>s_1>\frac12$, and $u\in L^2_s(\real)$. Define the Jost solutions as above, and let $\Gamma(\lambda)$, $\beta(\lambda)$ and $f(\lambda)$ be defined as in \Cref{lem: def scattering coeff}. Then $\Gamma\in C^{1,\gamma}_{loc}(0,\infty)$, $\beta,f\in C^{0,\gamma}_{loc}(0,\infty)$ for some $0<\gamma<1$, and the following relations hold:
\eqn
|\Gamma(\lambda)|=1,
\eeqn
\eqn\label{eq: f beta rel}
f(\lambda)=\frac{\overline{\beta(\lambda)}}{2\pi i \lambda},
\eeqn
\eqn\label{eq: beta square rel}
|\beta(\lambda)|^2 = 2~\textnormal{Im}\int_\real u(x)m_1(x,\lambda+0i)~dx,
\eeqn
and
\eqn\label{eq: diff Gamma rel}
\partial_{\lambda}\Gamma(\lambda) = \frac{|\beta(\lambda)|^2}{2\pi i\lambda}\Gamma(\lambda).
\eeqn
\end{lem}
\begin{proof}
The regularity of $\Gamma$, $\beta$ and $f$ follows easily from the corresponding regularity of the Jost solutions. What are left to show are the relations between them.
We start by multiplying the integral equation of $m_e(\lambda+)$ by $u\overline{m_e(\lambda+)}$ and integrate on $\real$. Using \Cref{lem: IBP iden}, we have
\begin{align}
&~\langle m_e(\lambda+),um_e(\lambda+)\rangle \notag\\
=&~ \langle \e, um_e(\lambda+)\rangle+\langle G_{\lambda+}*(um_e(\lambda+)), um_e(\lambda+)\rangle\notag\\
=&~ \langle \e, um_e(\lambda+)\rangle + i|\langle um_e(\lambda+,\e)\rangle|^2 + \langle um_e(\lambda+),G_{\lambda+}*(um_e(\lambda+))\rangle\notag\\
=&~\langle \e, um_e(\lambda+)\rangle + i|\langle um_e(\lambda+,\e)\rangle|^2 + \langle um_e(\lambda+),m_e(\lambda+)-\e\rangle.
\end{align}
Since $\langle m_e(\lambda+),um_e(\lambda+)\rangle=\langle um_e(\lambda+),m_e(\lambda+)\rangle$, we get
\eqn\label{eq: ume e}
i|\langle um_e(\lambda+),\e\rangle|^2 -\langle um_e(\lambda+),\e\rangle + \overline{\langle um_e(\lambda+),\e\rangle} =0.
\eeqn
By the definition of $\Gamma(\lambda)$ given in \cref{def: Gamma}, $\Gamma(\lambda) = 1+i\langle um_e(\lambda+),\e\rangle$. Hence \cref{eq: ume e} means
\eqn
i|\Gamma(\lambda)-1|^2 + i (\Gamma(\lambda)-1)+ i \overline{\Gamma(\lambda)-1}=0,
\eeqn
from which it follows that $|\Gamma(\lambda)|=1$. 

Next we multiply the integral equation of $m_1(\lambda+)$ given in \cref{eq: m1 int} by $u\overline{m_e(\lambda+)}$ and integrate on $\real$. Use \Cref{lem: IBP iden} again to get
\begin{align}
&~\langle m_1(\lambda+),um_e(\lambda+)\rangle = \langle 1 +G_{\lambda+0i}*(um_1(\lambda+)), um_e(\lambda+)\rangle\notag\\
=&~\int_\real u\overline{m_e(\lambda+)}~dx + i\langle um_1(\lambda+),\e\rangle \overline{\langle um_e(\lambda+), \e\rangle} \notag\\
&\quad + \langle um_1(\lambda+), G_{\lambda+0i}*(um_e(\lambda+)\rangle\notag\\
=&~\int_\real u\overline{m_e(\lambda+)}~dx + i\langle um_1(\lambda+),\e\rangle \overline{\langle um_e(\lambda+), \e\rangle} \notag\\
&\quad + \langle um_1(\lambda+), m_e(\lambda+)-\e\rangle.
\end{align}
Since $\langle m_1(\lambda+),um_e(\lambda+)\rangle =\langle um_1(\lambda+), m_e(\lambda+)\rangle$, this implies
\eqn
\int_\real u\overline{m_e(\lambda+)}~dx +\langle um_1(\lambda+),\e\rangle\left(i\overline{\langle um_e(\lambda+), \e\rangle}-1\right)=0.
\eeqn
By the definition of $\Gamma(\lambda)$, $\beta(\lambda)$, $f(\lambda)$, and the relation \cref{eq: me pm rel}, we get
\eqn
-2\pi \lambda\overline{f(\lambda)\Gamma(\lambda)}- \frac{1}{i}\beta(\lambda)\overline{\Gamma(\lambda)}=0.
\eeqn
Divide both sides by $\overline{\Gamma(\lambda)}$ to obtain \cref{eq: f beta rel}. This is allowed as $|\Gamma(\lambda)|=1$. 

Next we multiply the integral equation of $m_1(\lambda+)$ by $u\overline{m_1(\lambda+)}$ and integrate on $\real$. Use \Cref{lem: IBP iden} to get
\begin{align}
&~\langle m_1(\lambda+),um_1(\lambda+)\rangle = \langle 1 + G_{\lambda+0i}*(um_1(\lambda+)), um_1(\lambda+)\rangle\notag\\
=&~ \int_\real u\overline{m_1(\lambda+)}~dx + i|\langle um_1(\lambda+), \e\rangle |^2 + \langle um_1(\lambda+), G_{\lambda+0i}*(um_1(\lambda+))\rangle\notag\\
=&~\int_\real u\overline{m_1(\lambda+)}~dx + i|\langle um_1(\lambda+), \e\rangle |^2 + \langle um_1(\lambda+),m_1(\lambda+)-1\rangle
\end{align}
Since $\langle m_1(\lambda+),um_1(\lambda+)\rangle = \langle um_1(\lambda+),m_1(\lambda+)\rangle$, we have
\eqn
-2i~\text{Im}\int_\real um_1(\lambda+)~dx + i|\langle um_1(\lambda+), \e\rangle |^2=0,
\eeqn
from which \cref{eq: beta square rel} follows.

Finally, to get \cref{eq: diff Gamma rel}, we differentiate \cref{def: Gamma} using the plus sign case of \cref{eq: diff me m1 rel} and apply \cref{eq: me pm rel} to get
\begin{align}
&~\partial_{\lambda}\Gamma(\lambda)\notag\\
=&~ i \int_\real u(x) e^{-i\lambda x}\left(-\frac{1}{2\pi \lambda}\int_\real {u(y)m_e(y,\lambda+)}~dy\right)m_1(x,\lambda+)~dx\notag\\
= &f(\lambda)\beta(\lambda)\Gamma(\lambda). 
\end{align}
\Cref{eq: diff Gamma rel} now follows from \cref{eq: f beta rel}.
\end{proof}

%%%%%%%%%%%%%%%%%%%%%%%%%%%%%%%%%%%%%%%%%%%%%%%%%%%%%%%%%%%%%%%%%%%%%%

%%%%%%%%%%%%%%%%%%%%%%%%%%%%%%%%%%%%%%%%%%%%%%%%%%%%%%%%%%%%%%%%%%%%%%
\section{Asymptotic behavior near $k=0$}\label{sec: k 0 limit}
%%%%%%%%%%%%%%%%%%%%%%%%%%%%%%%%%%%%%%%%%%%%%%%%%%%%%%%%%%%%%%%%%%%%%%

In this section, we discuss the asymptotic behavior of the Jost solutions and scattering coefficients as $k$ approaches $0$ within the set $\rho(L_u)\cup(\real^{+}\pm 0i)$. It turns out that the convolution kernel $G_k$ has a logarithmic singularity at $k=0$, and so does the operator $T_k$. We employ the well-known method of subtracting a rank one operator from $T_k$ so that the modified operator has a limit at $k=0$. The limiting modified operator also has the form of identity plus a compact operator. We then obtain its invertibility through a vanishing lemma. The asymptotic behavior of the Jost functions can be recovered from the modified Jost functions. The asymptotics presented in this section was formally obtained in \cite{fokas1983inverse} and \cite{kaup1998inverse}.

Let $\chi(\xi)$ be a smooth function on $[0,\infty)$, which is identically equal to $1$ for $0\le \xi \le 1$ and identically equal to $0$ for $\xi \ge 2$. Later on we will see that it is crucial to allow the possibility of $\chi(\xi)$ being {\it complex} for $1<\xi<2$. For $k\in\rho(L_u)\cup(\real^{+}\pm 0i)$, let
\eqn\label{def: l}
l(k)=\frac1{2\pi}\int_0^{\infty}\frac{\chi(\xi)}{\xi-k}~d\xi,
\eeqn
and let
\begin{align}
G_k^0(x)&=G_{k}(x)-l(k)=\frac1{2\pi}\int_0^{\infty}\frac{e^{ix\xi}-\chi(\xi)}{\xi-k}~d\xi, \label{def: Gk0}\\
T_k^0(\varphi)&=G_k^0*(u\varphi) = T_k(\varphi)-l(k)\langle \varphi, u\rangle. \label{def: Tk0}
\end{align}
We also define 
\eqn\label{def: G00}
G_0^0(x) = \frac{1}{2\pi}\int_0^{\infty}\frac{e^{ix\xi}-\chi(\xi)}{\xi}~d\xi, \quad T_0^0(\varphi) = G_0^0*(u\varphi).
\eeqn
We define the modified Jost functions $m_1^0(x,k)$ and $m_e^0(x,\lambda\pm 0i)$ to be solutions (if exist) to the integral equations
\begin{align}
m_1^0(k)& = 1+T_k^0(m_1^0(k)) = 1+G_k^0*(um_1^0(k)), \label{eq: mod m1}\\
m_e^0(\lambda\pm 0i) &= \e(\lambda)+T_{\lambda\pm 0i}^0(m_e^0(\lambda\pm 0i)) = \e(\lambda)+G_{\lambda\pm 0i}^0*(um_e^0(\lambda\pm 0i)). \label{eq: mod me}
\end{align}
Using \cref{def: Tk0}, we obtain the relation between the original and the modified Jost functions:
\eqn\label{eq: m1 vs m1 mod}
m_1(k)=\frac{m_1^0(k)}{1-l(k)\langle m_1^0(k),u\rangle},
\eeqn
%\eqn
%m_1^0(k) = \frac{m_1(k)}{1+l(k)\langle m_1(k),u\rangle}.
%\eeqn
\begin{align}
m_e(\lambda\pm )&=m_e^0(\lambda\pm )+l(\lambda\pm)\langle m_e^0(\lambda\pm ),u\rangle m_1(\lambda\pm )\notag\\
&= \frac{m_e^0(\lambda\pm)+l(\lambda\pm)\big(\langle m_e^0(\lambda\pm),u\rangle m_1^0(\lambda\pm )-\langle m_1^0(\lambda\pm),u\rangle m_e^0(\lambda\pm )\big)}{1-l(\lambda\pm)\langle m_1^0(\lambda\pm ),u\rangle}.\label{eq: me vs me mod}
\end{align}
To prove existence of the modified Jost functions when $k$ is near $0$, we carry out the plan introduced at the beginning of this section. The first step is to estimate the modified convolution kernel $G_k^0$.

\begin{lem}\label{lem: Gk0-G00}
There exists $k_0>0$ such that for every $\epsilon\in (0,1)$, there is $C>0$ such that for all $k\in (\comp\setminus[0,\infty))\cup (\real^+\pm 0i)$ with $|k|<k_0$, 
\eqn\label{est: Gk-G0}
|G_k^0(x)-G_0^0(x)|\le C|k|^{\epsilon}(1+|x|)^{\epsilon}.
\eeqn
\end{lem}
\begin{proof}
By the definition of $\chi(\xi)$, we write
\begin{equation}\label{eq: Gk-G0 split}
G_k^0(x)-G_0^0(x) = \frac1{2\pi}\int_0^1\frac{e^{ix\xi}-1}{\xi}\frac{k}{\xi-k}~d\xi + \frac1{2\pi}\int_1^{\infty}\frac{e^{ix\xi}-\chi(\xi)}{\xi}\frac{k}{\xi-k}~d\xi.
\end{equation}
We first estimate the second term in \cref{eq: Gk-G0 split}. Assuming $|k|<k_0<\frac12$, we obtain $|\xi-k|\ge \xi-|k|\ge \frac12\xi$ for $\xi\ge 1$. Thus
\eqn
\left|\int_1^{\infty}\frac{e^{ix\xi}-\chi(\xi)}{\xi}\frac{k}{\xi-k}~d\xi\right|\le C|k|\int_1^{\infty}\frac1{\xi^2}~d\xi \le C|k|.
\eeqn
We are left to estimate the first term in \cref{eq: Gk-G0 split}. Let us first consider the case $x>0$. Make a change of variable to rewrite the integral as
\eqn\label{eq: Gk-G0 1}
kx\int_0^x \frac{e^{i\xi}-1}{\xi}\frac{1}{\xi-kx}~d\xi
\eeqn
Notice that $k\in (\comp\setminus[0,\infty))\cup (\real^+\pm 0i)$ means $kx$ can get arbitrarily close to the interval $(0,x)$. We deform the contour of integration when estimating \cref{eq: Gk-G0 1}. The work is split into two cases: when $|k|x<1$, or when $|k|x\ge1$. If $|k|x<1$, we split the integral \cref{eq: Gk-G0 1} as follows
\eqn
\int_{\Gamma_1}+\int_{2|k|x}^2+\int_2^x.
\eeqn
Here $\Gamma_1$ is a semicircle centered at $|k|x$ with radius $|k|x$. $\Gamma_1$ is in the lower half plane if $kx$ is in the upper half plane, and vice versa. With this choice, we have $|\xi-kx|\ge|k|x$ and $\left|\frac{e^{i\xi}-1}{\xi}\right|\le C$ when $\xi\in \Gamma_1$. Hence
\eqn
|k|x\left|\int_{\Gamma_1}\frac{e^{i\xi}-1}{\xi}\frac{1}{\xi-kx}~d\xi\right|\le C|k|x\le C|k|^{\epsilon}|x|^{\epsilon}.
\eeqn
We used $|k|x<1$ to get the last inequality. For $2|k|x<\xi<2$, we have $\left|\frac{e^{i\xi}-1}{\xi}\right|\le C$ and $|\xi-kx|\ge \xi-|k|x$. Hence
\begin{align}
|k|x\left|\int_{2|k|x}^2\frac{e^{i\xi}-1}{\xi}\frac{1}{\xi-kx}~d\xi\right|&\le C|k|x\int_{2|k|x}^2\frac{1}{\xi-|k|x}~d\xi\notag\\
&=C|k|x\log\left(\frac{2}{|k|x}-1\right)\notag\\
&\le C|k|^{\epsilon}|x|^{\epsilon}.
\end{align}
For $\xi$ between $2$ and $x$ (either could be the larger of the two), $\left|\frac{e^{i\xi}-1}{\xi}\right|\le \frac{2}{\xi}$, and $|\xi-kx|\ge \xi-|k|x$. Thus
\begin{align}
|k|x\left|\int_2^x\frac{e^{i\xi}-1}{\xi}\frac{1}{\xi-kx}~d\xi\right|&\le C|k|x\left|\int_2^x\frac{1}{\xi(\xi-|k|x)}~d\xi\right|\notag\\
&=C\left|\log\left(\frac{2(1-|k|)}{2-|k|x}\right)\right|\notag\\
&\le C\log[(1+|k|)(1+|k|x)]\notag\\
&\le C|k|^{\epsilon}(1+|x|)^{\epsilon}.
\end{align}
Now let's suppose $|k|x\ge 1$, we split the integral \cref{eq: Gk-G0 1} into the following pieces:
\eqn
\int_0^{\frac12}+\int_{\frac12}^{|k|x-\frac12}+\int_{\Gamma_2}+\int_{|k|x+\frac12}^x.
\eeqn
Here $\Gamma_2$ is a semicircle centered at $|k|x$ with radius $\frac12$. Again, $\Gamma_2$ is in the lower half plane if $kx$ is in the upper half plane, and vice versa. For $0<\xi<\frac12$, $\left|\frac{e^{i\xi}-1}{\xi}\right|\le C$, and $|\xi-kx|\ge |k|x-\xi$. Hence
\begin{align}
|k|x\left|\int_0^{\frac12}\frac{e^{i\xi}-1}{\xi}\frac{1}{\xi-kx}~d\xi\right|&\le C|k|x\int_0^{\frac12}\frac{1}{|k|x-\xi}~d\xi\notag\\
&= C|k|x\log\left(\frac{|k|x}{|k|x-\frac12}\right)\notag\\
&=C|k|x\log\left(1+\frac{1}{2|k|x-1}\right)\notag\\
&\le C|k|x\log\left(1+\frac{1}{|k|x}\right)\le C|k|^{\epsilon}|x|^{\epsilon}.
\end{align}
For $\frac12<\xi<|k|x-\frac12$, $\left|\frac{e^{i\xi}-1}{\xi}\right|\le \frac{2}{\xi}$, and $|\xi-kx|\ge |k|x-\xi$. Hence
\begin{align}
|k|x\left|\int_{\frac12}^{|k|x-\frac12}\frac{e^{i\xi}-1}{\xi}\frac{1}{\xi-kx}~d\xi\right|&\le C|k|x\int_{\frac12}^{|k|x-\frac12}\frac{1}{\xi(|k|x-\xi)}~d\xi\notag\\
&\le C\log(2|k|x-1)\notag\\
&\le C|k|^{\epsilon}|x|^{\epsilon}.
\end{align}
For $\xi\in \Gamma_2$, $\left|\frac{e^{i\xi}-1}{\xi}\right|\le \frac{C}{|\xi|}\le \frac{C}{|k|x-\frac12}$, $|\xi-kx|\ge\frac12$. Thus
\begin{equation}
|k|x\left|\int_{\Gamma_2}\frac{e^{i\xi}-1}{\xi}\frac{1}{\xi-kx}~d\xi\right|\le \frac{C|k|x}{|k|x-\frac12}\le C\le C|k|^{\epsilon}|x|^{\epsilon}.
\end{equation}
Of course we used $|k|x\ge 1$. Finally, for $|k|x+\frac12<\xi<x$, $\left|\frac{e^{i\xi}-1}{\xi}\right|\le \frac{2}{\xi}$, and $|\xi-kx|\ge \xi-|k|x$. Thus
\begin{align}
|k|x\left|\int_{|k|x+\frac12}^{x}\frac{e^{i\xi}-1}{\xi}\frac{1}{\xi-kx}~d\xi\right|&\le C|k|x\int_{|k|x+\frac12}^{x}\frac{1}{\xi(\xi-|k|x)}~d\xi\notag\\
&\le C\log[(1-|k|)(2|k|x+1)]\notag\\
&\le C|k|^{\epsilon}(1+|x|)^{\epsilon}.
\end{align}
This finishes the proof of \cref{est: Gk-G0} when $x>0$. The proof for $x< 0$ is completely analogous. The case $x=0$ is trivial.
\end{proof}

\begin{lem}\label{lem: G00}
\eqn\label{eq: G00 form}
G_0^0(x)=\lim_{N\to \infty}\frac1{2\pi}\int_0^N \frac{e^{ix\xi}-\chi(\xi)}{\xi}~d\xi = -\frac{1}{2\pi}\log|x|+\frac1{2\pi}\begin{cases}c_1 &\text{ if }x>0,\\ c_2 &\text{ if }x<0,\end{cases}
\eeqn
%where
%\eqn
%c_1 = -\int_1^2\frac{\chi(\xi)}{\xi}~d\xi +\int_0^{\infty}\frac{\cos\xi - \chi_{\{\xi<1\}}(\xi)}{\xi}~d\xi+ i\int_0^{\infty}\frac{\sin\xi }{\xi}~d\xi,
%\eeqn
%and
%\eqn
%c_2 = -\int_1^2\frac{\chi(\xi)}{\xi}~d\xi +\int_0^{\infty}\frac{\cos\xi - \chi_{\{\xi<1\}}(\xi)}{\xi}~d\xi- i\int_0^{\infty}\frac{\sin\xi }{\xi}~d\xi.
%\eeqn
and there is $C>0$ such that
\eqn\label{est: G00 N}
\left|\frac1{2\pi}\int_0^N\frac{e^{ix\xi}-\chi(\xi)}{\xi}~d\xi \right|\le C+C|\log|x||+\frac{C}{|x|^{\frac12}}\chi_{\{|x|\le 1\}}
\eeqn
for all $N>2$.
\begin{proof}
We write $G_0^0$ as
\eqn\label{eq: G00 split}
G_0^0(x) = \frac1{2\pi}\int_0^1\frac{e^{ix\xi}-1}{\xi}~d\xi + \frac1{2\pi}\int_1^{\infty}\frac{e^{ix\xi}}{\xi}~d\xi-\frac1{2\pi}\int_1^{2}\frac{\chi(\xi)}{\xi}~d\xi.
\eeqn
When $x>0$, make a change of variable to get
\begin{align}
G_0^0(x) = &~-\frac1{2\pi}\int_1^x \frac1{\xi}~d\xi+ \frac1{2\pi}\int_0^1\frac{e^{i\xi}-1}{\xi}~d\xi+\frac{1}{2\pi}\int_1^{\infty}\frac{e^{i\xi}}{\xi}~d\xi\notag\\
&\quad -\frac1{2\pi}\int_1^{2}\frac{\chi(\xi)}{\xi}~d\xi\notag\\
=&~-\frac1{2\pi}\log|x|+\frac{c_1}{2\pi}.
\end{align}
When $x<0$, a change of variable gives
\begin{align}
G_0^0(x) = &~-\frac1{2\pi}\int_{-1}^x \frac1{\xi}~d\xi+ \frac1{2\pi}\int_0^{-1}\frac{e^{i\xi}-1}{\xi}~d\xi+\frac{1}{2\pi}\int_{-1}^{-\infty}\frac{e^{i\xi}}{\xi}~d\xi\notag\\
&\quad -\frac1{2\pi}\int_1^{2}\frac{\chi(\xi)}{\xi}~d\xi\notag\\
=&~-\frac1{2\pi}\log|x|+\frac{c_2}{2\pi}.
\end{align}
%and estimate the three terms under the assumption that $x>0$. The second term in \cref{eq: G00 split} is bounded by a constant, which is bounded by $C\log\left(|x|+\frac1{|x|}\right)$. The last integral can be written as $\int_{2x}^{\infty}\frac{e^{i\xi}}{\xi}~d\xi$. This is the well-known exponential integral, which is easily seen to be bounded by $C+C\left|\log\left(\frac1{|x|}\right)\right|\le C\log\left(|x|+\frac1{|x|}\right)$. Write the first integral in \cref{eq: G00 split} as
%\eqn
%\int_0^x\frac{e^{i\xi}-1}{\xi}~d\xi=\int_0^1\frac{e^{i\xi}-1}{\xi}~d\xi+\int_1^x\frac{e^{i\xi}-1}{\xi}~d\xi.
%\eeqn
%The first term is bounded by a constant, and the second term is bounded by $C+C|\log(|x|)|$. The proof for $x<0$ is analogous. 
Next we assume $x>0$, $N>2$, and write the integral in \cref{est: G00 N} as 
\eqn
G_0^0(x) - \frac{1}{2\pi}\int_N^{\infty}\frac{e^{ix\xi}}{\xi}~d\xi=G_0^0(x)-\frac1{2\pi}\int_{Nx}^{\infty}\frac{e^{i\xi}}{\xi}~d\xi,
\eeqn
where $\int_{Nx}^{\infty}\frac{e^{i\xi}}{\xi}~d\xi$ is easily seen to be bounded by
\begin{align}
C+C\log\left(\frac{1}{Nx}\right)\chi_{\{Nx<1\}}&\le C+\frac{C}{|Nx|^{\frac12}}\chi_{\{Nx<1\}}\notag\\
&\le C+\frac{C}{|x|^{\frac12}}\chi_{\{|x|\le 1\}}
\end{align}
for $N>2$. The proof for $x<0$ is similar.
\end{proof}
\end{lem}

The pointwise estimates established in the above lemmas imply estimates on the $L^{\infty}_{-(s-s_1)}$ operator norm of $T_k^0$.

\begin{lem}\label{lem: Tk0-T00}
Let $s>s_1>\frac12$, and $u\in L^2_s(\real)$. Let $k\in (\comp\setminus [0,\infty))\cup(\real^+\pm 0i)\cup\{0\}$. Then all $T_k^0$ are compact on $L^{\infty}_{-(s-s_1)}(\real)$, and there exist $\epsilon\in(0,1)$, $k_0>0$ and $C>0$ such that for all $|k|<k_0$, 
\eqn\label{est: Tk0-T00}
\|T_k^0-T_0^0\|_{L^{\infty}_{-(s-s_1)}\to L^{\infty}_{-(s-s_1)}}\le C|k|^{\epsilon}.
\eeqn
\end{lem}
\begin{proof}
When $k\ne 0$, $T_k^0$ is compact since it is a rank one perturbation of $T_k$, which was shown to be compact in \Cref{lem: compact}. The compactness of $T_0^0$ follows from \cref{est: Tk0-T00}, which we now show. In fact, by \Cref{lem: Gk0-G00}
\begin{align}
&~w^{s_1-s}(x)|T_k^0(\varphi)(x)-T_0^0(\varphi)(x)|\notag\\
\le &~w^{s_1-s}(x)\int_{\real}|G_k^0(x-y)-G_0^0(x-y)||u(y)\varphi(y)|~dy\notag\\
\le &~Cw^{s_1-s}(x)\int_{\real}|k|^{\epsilon}w^{\epsilon}(x-y)w^{-p_2}(y)|u_2(y)|\|\psi\|_{L^{\infty}_{-(s-s_1)}}~dy\notag\\
\le &~C|k|^{\epsilon}\|u_2\|_1\|\psi\|_{L^{\infty}_{-(s-s_1)}},
\end{align}
where $p_2 = \frac12(s_1-\frac12)>0$ and $u_2 = uw^{s-s_1+p_2}\in L^1$. To get the last step above, we used $w(x-y)\le w(x)w(y)$, and $\epsilon<\min(s-s_1,p_2)$.
\end{proof}

The key to proving existence of modified Jost functions is to show invertibility of $I-T_0^0$, which by \Cref{lem: Tk0-T00},  reduces to showing triviality of its kernel. We accomplish this in several steps. First we show an identity that is crucial for later developments. It is for this identity that the complexity of $\chi(\xi)$ is needed. Recall that $\chi(\xi)$ is the cutoff function in the definition of $G_0^0$ and $T_0^0$ in \cref{def: G00}.

\begin{lem}\label{lem: vanish int 1}
Suppose $s>s_1>\frac12$, $u\in L^2_s(\real)$, $\varphi\in L^{\infty}_{-(s-s_1)}(\real)$, and $\varphi = T_0^0\varphi$. If
\eqn\label{eq: chi nonzero im}
\text{Im}\int_1^2\frac{\chi(\xi)}{\xi}~d\xi\ne 0,
\eeqn
then
\eqn\label{eq: vanish int 1}
\int_\real u(y)\varphi(y)~dy=0.
\eeqn
\end{lem}
\begin{proof}
By \cref{def: G00},
\eqn
G_0^0(x) = \overline{G_0^0(-x)}-\frac{i}{\pi}~\text{Im}\int_1^2\frac{\chi(\xi)}{\xi}~d\xi.
\eeqn
Therefore by $\varphi=T_0^0\varphi = G_0^0*(u\varphi)$,
\begin{align}
&~\langle \varphi, u\varphi\rangle = \langle G_0^0*(u\varphi) ,u\varphi\rangle =\int_\real \int_\real G_0^0(x-y) u(y)\varphi(y)u(x)\overline{\varphi(x)}~dy~dx\notag\\
= &~\int_\real \int_\real \left(\overline{G_0^0(y-x)}-\frac{i}{\pi}~\text{Im}\int_1^2\frac{\chi(\xi)}{\xi}~d\xi\right) u(y)\varphi(y)u(x)\overline{\varphi(x)}~dy~dx\notag\\
= &~-\frac{i}{\pi}\left(\text{Im}\int_1^2\frac{\chi(\xi)}{\xi}~d\xi\right)|\langle \varphi, u\rangle|^2 + \int_\real\int_\real u(y)\varphi(y) \overline{G_0^0(y-x)u(x)\varphi(x)}~dx~dy\notag\\
= &~-\frac{i}{\pi}\left(\text{Im}\int_1^2\frac{\chi(\xi)}{\xi}~d\xi\right)|\langle \varphi, u\rangle|^2 + \langle u\varphi, G_0^0*(u\varphi)\rangle\notag\\
 = &~-\frac{i}{\pi}\left(\text{Im}\int_1^2\frac{\chi(\xi)}{\xi}~d\xi\right)|\langle \varphi, u\rangle|^2 + \langle u\varphi, \varphi\rangle. \label{eq: vanish int 2}
\end{align}
Since $\langle \varphi, u\varphi\rangle = \langle u\varphi, \varphi\rangle$, \cref{eq: vanish int 2} and \cref{eq: chi nonzero im} imply $\langle \varphi, u\rangle = 0$, which is \cref{eq: vanish int 1}. We provide the estimates needed to apply Fubini's theorem in the calculation above. In fact, by \cref{eq: G00 form},
\begin{align}
&~\int_\real\int_\real |G_0^0(x-y)||u(y)\varphi(y)||u(x)\varphi(x)|~dy~dx\notag\\
\le &~C\|u\varphi\|_1^2 + C\int_\real\int_\real |\log|x-y|||u(y)\varphi(y)||u(x)\varphi(x)|~dy~dx.\label{eq: fubini split 1}
\end{align}
We split the integral $\int_\real |\log|x-y|||u(y)\varphi(y)|~dy$ at $|x-y|=1$ and estimate
\eqn
\int_{|x-y|<1}|\log|x-y|||u(y)\varphi(y)|~dy \le \|\chi_{\{|x|<1\}}\log x\|_2\|u\varphi\|_2,
\eeqn
and
\begin{align}
\int_{|x-y|>1}\log|x-y||u(y)\varphi(y)|~dy&\le \int_{|x-y|>1}(1+|x-y|)^{\epsilon}|u(y)\varphi(y)|~dy\notag\\
&\le (1+|x|)^{\epsilon}\int_\real(1+|y|)^{\epsilon}w^{-s_1}(y)|u_1(y)\varphi_1(y)|~dy\notag\\
&\le C\|u_1\|^2 (1+|x|)^{\epsilon}.
\end{align}
Here $u_1 = w^su\in L^2$, $\varphi_1=w^{-(s-s_1)}\varphi\in L^{\infty}$, and $\epsilon>0$ is chosen so that $s_1-\epsilon>\frac12$. The estimates above imply the finiteness of \cref{eq: fubini split 1}.
\end{proof}

The key vanishing integral \cref{eq: vanish int 1} implies the following decay estimate for functions in the kernel of $I-T_0^0$.

\begin{lem}
Suppose $s>s_1>\frac12$, $u\in L^2_s(\real)$ and that \cref{eq: chi nonzero im} holds. If $\varphi\in L^{\infty}_{-(s-s_1)}(\real)$, and $\varphi = T_0^0\varphi$, then there exists $C=C( u,s,s_1)$ such that
\eqn\label{est: T00 phi decay}
|\varphi(x)|\le Cw^{-1}(x).
\eeqn
In particular $\varphi\in L^2(\real)$.
\end{lem}
\begin{proof}
%Let us take the Fourier transform of \cref{eq: diff T00}:
%\eqn
%\xi\hat\varphi(\xi) = \chi_{\real^+}(\xi)\widehat{u\varphi}(\xi).
%\eeqn
%The conditions on $u$ and $\varphi$ imply $u\varphi\in L^1\cap L^2$. Hence $\widehat{u\varphi}\in C(\real)$. We observe that $\hat\varphi$ is a tempered distribution such that $\hat\varphi=0$ on $(-\infty,0)$, and $\hat\varphi=\frac{\widehat{u\varphi}(\xi)}{\xi}$ on $(0,\infty)$. \Cref{lem: distribution} in the Appendix shows that $\widehat{u\varphi}(0)$ must be $0$. In other words:

By \cref{eq: G00 form},
\begin{align}
&~\varphi(x) = T_0^0\varphi(x)=\frac1{2\pi}\int_\real G_0^0(x-y)u(y)\varphi(y)~dy\notag\\
=&~ \frac{1}{2\pi}\int_{-\infty}^x(c_1-\log(x-y))u(y)\varphi(y)~dy \notag\\
&\quad +\frac{1}{2\pi}\int_x^{\infty}(c_2-\log(y-x))u(y)\varphi(y)~dy\notag\\
=&~- \frac{1}{2\pi}\int_x^{\infty}\log(y-x)u(y)\varphi(y)~dy \label{eq: T00 phi split}\\
  &\quad -\frac{1}{2\pi}\int_{-\infty}^x\log(x-y)u(y)\varphi(y)~dy + R(x). \notag
\end{align}
By \cref{eq: vanish int 1}, $R(x)$ can be written in two ways:
\eqn\label{eq: R(x)}
R(x) = \frac1{2\pi}\int_x^{\infty}(c_2-c_1)u(y)\varphi(y)~dy = \frac1{2\pi}\int_{-\infty}^x(c_1-c_2)u(y)\varphi(y)~dy.
\eeqn

We now start a bootstrap argument, assuming $\varphi\in L^{\infty}_{r}$ for some $r\ge -(s-s_1)$. We have $u=w^{-s}u_1$ for some $u_1\in L^2$ and $\varphi = w^{-r}\varphi_1$ for some $\varphi_1\in L^{\infty}$. Hence $u\varphi=w^{-(r+s)}u_1\varphi_1$. Since $r\ge -(s-s_1)$, $r+s\ge s_1>\frac12$. Letting $p_2= \frac12(s_1-\frac12)>0$, we get 
\eqn\label{eq: u phi}
u\varphi = w^{-(r+\Delta r)}w^{-\frac12-p_2}u_1\varphi_1 = w^{-(r+\Delta r)}u_2\varphi_1.
\eeqn
Here 
\eqn
\Delta r=s-s_1+p_2>p_2>0, 
\eeqn
\eqn
r+\Delta r =r+s-s_1+p_2\ge p_2>0,
\eeqn
and $u_2=w^{-\frac12-p_2}u_1\in L^1\cap L^2$.

We assume $x>0$. Use the first expression in \cref{eq: R(x)} for $R(x)$ to get
\eqn
|R(x)|\le C\int_x^{\infty}w^{-(r+\Delta r)}(y)|u_2(y)\psi_1(y)|~dy\le Cw^{-(r+\Delta r)}(x).
\eeqn
Next we write the first integral in \cref{eq: T00 phi split} as
\eqn
\int_x^{\infty}\log(y-x)w^{-(r+\Delta r)}(y)u_2(y)\varphi_1(y)~dy
\eeqn
and split the integral at $x+1$:
\begin{align}
&~\left|\int_x^{x+1}\log(y-x)w^{-(r+\Delta r)}(y)u_2(y)\varphi_1(y)~dy\right|\notag\\
\le&~ Cw^{-(r+\Delta r)}(x)\big\|\chi_{\{|y|\le 1\}}\log y \big\|_2\|u_2\varphi_1\|_2
\end{align}
When $y-x>1$, there exists $C=C(\epsilon)$ for every $\epsilon>0$ such that $\log(y-x)\le Cw^{\epsilon}(x-y)\le Cw^{\epsilon}(x)w^{\epsilon}(y)$. Take $\epsilon=\frac{p_2}{4}$. We have $r+\Delta r -\epsilon \ge p_2-\epsilon>0$ and $\Delta r >p_2>4\epsilon$. Thus
\begin{align}
&~\left|\int_{x+1}^{\infty}\log(y-x)w^{-(r+\Delta r)}(y)u_2(y)\varphi_1(y)~dy\right|\notag\\
\le &~Cw^{\epsilon}(x)\int_{x+1}^{\infty}w^{\epsilon-(r+\Delta r)}(y)|u_2(y)\varphi_1(y)|~dy\notag\\
\le &~Cw^{2\epsilon - (r+\Delta r)}(x)\|u_2\varphi_1\|_1\le Cw^{-(r+\frac{\Delta r}{2})}(x).
\end{align}
In summary, the first integral in \cref{eq: T00 phi split} is bounded as follows:
\eqn
\left|\int_x^{\infty}\log(y-x)u(y)\varphi(y)~dy\right|\le Cw^{-(r+\frac{\Delta r}{2})}(x).
\eeqn
We now focus on the second integral in \cref{eq: T00 phi split}. When $0<x<1$, we split the integral at $-1$ and estimate
\eqn
\left|\int_{-\infty}^{-1}\log(x-y)u(y)\varphi(y)~dy\right|\le C\int_{-\infty}^{-1}\log(1+|y|)|u(y)\varphi(y)|~dy\le C,
\eeqn
\eqn
\left|\int_{-1}^{x}\log(x-y)u(y)\varphi(y)~dy\right|\le C\big\|\chi_{\{|y|\le 2\}}\log y\big\|_2\|u\varphi\|_2.
\eeqn
When $x>1$, we use \cref{eq: vanish int 1} again to rewrite the second integral in \cref{eq: T00 phi split} as
\begin{align}\label{eq: T00 phi split 2}
&~\int_{-\infty}^x(\log(x-y)-\log x)w^{-(r+\Delta r)}(y)u_2(y)\varphi_1(y)~dy\notag\\
&\quad -\log x \int_{x}^{\infty}w^{-(r+\Delta r)}(y)u_2(y)\varphi_1(y)~dy.
\end{align}
The last term in \cref{eq: T00 phi split 2} is easily seen to be bounded by
\eqn
( \log x)w^{-(r+\Delta r)}(x)\|u_2\varphi_1\|_1\le Cw^{-(r+\frac{\Delta r}{2})}(x).
\eeqn
We split the first integral in \cref{eq: T00 phi split 2} at $\frac x2$, and estimate as follows. When $y<\frac x2$, we have
\eqn
\left|\log\left(1-\frac yx\right)\right|\le C\left|\frac y x\right|^{p},
\eeqn
where $p=\min(1,r+\frac{\Delta r}2)$. Thus
\begin{align}
&~\left|\int_{-\infty}^{\frac x2}\log\left(1-\frac yx\right)w^{-(r+\Delta r)}(y)u_2(y)\varphi_1(y)~dy\right|\notag\\
\le &~ C|x|^{-p}\int_{-\infty}^{\frac x2}|y|^pw^{-(r+\Delta r)}(y)|u_2(y)\varphi_1(y)|~dy \notag\\
\le &~C    w^{-p}(x)\|u_2\varphi_1\|_1.
\end{align}
To estimate the $y>\frac x2$ piece of the first integral in \cref{eq: T00 phi split 2}, we use the first expression for $u\varphi$ in \cref{eq: u phi}, and get
\begin{align}
&~\left|\int_{\frac x2}^x \log\left(1-\frac yx\right)w^{-(r+\Delta r)-(\frac12+p_2)}(y)u_1(y)\varphi_1(y)~dy\right|\notag\\
\le &~Cw^{-(r+\Delta r)-(\frac12+p_2)}(x)\left(\int_{\frac x2}^x \left|\log\left(1-\frac yx\right)\right|^2~dy\right)^{\frac12}\|u_1\varphi_1\|_2\notag\\
\le &~Cw^{-(r+\Delta r)-(\frac12+p_2)}(x)\left(\int_0^{\frac 12} \left|\log z\right|^2~dz\right)^{\frac12}x^{\frac12}\notag\\
\le &~Cw^{-(r+\Delta r)}(x).
\end{align}
This completes the estimation of \cref{eq: T00 phi split} when $x>0$. The arguments for $x<0$ are completely analogous, as long as one uses the second expression in \cref{eq: R(x)} for $R(x)$. In summary, we get from the above estimates that $\varphi\in L^{\infty}_{r+\frac{\Delta r}2}$ if $r+\frac{\Delta r}2<1$ and $\varphi\in L^{\infty}_1$ if $r+\frac{\Delta r}2\ge 1$. The result thus follows from finitely many iterations of the above estimates.
\end{proof}

Next we show that any function in the kernel of $I-T_0^0$ satisfies the same type of eigenvalue equation as do the Jost functions, regardless of the choice of $\chi(\xi)$.
\begin{lem}\label{lem: diff T00}
Let $s>s_1>\frac12$, and $u\in L^2_s(\real)$. If $\varphi\in L^{\infty}_{-(s-s_1)}(\real)$ satisfies $\varphi = T_0^0\varphi$, then
\eqn\label{eq: diff T00}
\frac1i\partial_x\varphi - C_+(u\varphi)=0
\eeqn
in the sense of tempered distributions.
\end{lem}
\begin{proof}
Let $\psi(x)$ be any test function in $C_0^{\infty}(\real)$. Let $M>1$ be such that $[-M,M]$ contains the support of $\psi$. There exists $C=C(M,\psi)$ such that
\begin{align}
&~\int_{\real}\int_{\real}\left|1+|\log|x-y||+\frac{1}{|x-y|^{\frac12}}\chi_{\{|x-y|\le 1\}}\right||u(y)\varphi(y)\psi'(x)|~dy~dx\notag\\
\le &~C\int_{\real}\int_{-M}^M \left(1+|\log|x-y||+\frac{1}{|x|^{\frac12}}\right)~dx ~|u(y)\varphi(y)|~dy\notag\\
\le &~C\int_{\real}(1+\log(|y|+M))|u(y)\varphi(y)|~dy<\infty.\label{est: dct}
\end{align}
Therefore, by \cref{est: G00 N}, \cref{est: dct} and the dominated convergence theorem,
\begin{align}
\int_\real \varphi (x)\psi'(x)~dx &=\int_{\real}\int_{\real}G_0^0(x-y)u(y)\varphi(y)~dy~\psi'(x)~dx\notag\\
&=\lim_{N\to \infty}\frac1{2\pi}\int_{\real}\int_\real\int_0^N\frac{e^{i(x-y)\xi}-\chi(\xi)}{\xi}~d\xi~u(y)\varphi(y)\psi'(x)~dy~dx. \label{eq: weak diff varphi 1}
\end{align}
We want to use the Fubini theorem to change the order of integration. To that end, we observe that there is $C=C(M,\psi,N)$ such that
\begin{align}
&~\int_{\real}\int_{\real}\int_0^N\left|\frac{e^{i(x-y)\xi}-\chi(\xi)}{\xi}\right||\psi'(x)||u(y)\varphi(y)|~d\xi~dx~dy\notag\\
\le &~C\int_\real\int_{-M}^M\left(\int_0^1\left|\frac{e^{i(x-y)\xi}-1}{\xi}\right|~d\xi+C\right)|u(y)\varphi(y)|~dx~dy\notag\\
\le &~C\int_\real\int_{-M}^M(1+|\log|x-y||)|u(y)\varphi(y)|~dx~dy\notag\\
\le &~C\int_\real (1+\log(|y|+M))|u(y)\varphi(y)|~dx~dy<\infty.
\end{align}
Hence \cref{eq: weak diff varphi 1} equals to 
\begin{align}
&~\lim_{N\to \infty}\frac1{2\pi}\int_{\real}\int_0^N\int_\real\left(\frac{e^{i(x-y)\xi}-\chi(\xi)}{\xi}\right)\psi'(x)~dx~d\xi~u(y)\varphi(y)~dy\notag\\
=&~-i\lim_{N\to \infty}\int_\real\frac1{2\pi}\int_0^N\int_\real e^{i(x-y)\xi}\psi(x)~dx~d\xi~u(y)\varphi(y)~dy\notag\\
=&~-i\lim_{N\to \infty}\int_\real\psi(x)\frac1{2\pi}\int_0^N e^{ix\xi}\int_\real e^{-iy\xi}u(y)\varphi(y)~dy~d\xi~dx\notag\\
=&~-i\int_\real \psi(x) C_+(u\varphi)(x)~dx.
\end{align}
The last step follows from the fact that $\frac1{2\pi}\int_0^N e^{ix\xi}\int_\real e^{-iy\xi}u(y)\varphi(y)~dy~d\xi$ converges to $C_+(u\varphi)(x)$ in $L^2$. The above calculation shows
\eqn
\int_\real \varphi (x)\psi'(x)~dx =-i\int_\real \psi(x) C_+(u\varphi)(x)~dx,
\eeqn
which gives \cref{eq: diff T00}.
\end{proof}

We are now ready to prove the key vanishing lemma.
\begin{lem}\label{lem: mod L2 vanishing}
Suppose $s>s_1>\frac12$, $u\in L^2_s(\real)$, and that \cref{eq: chi nonzero im} holds. If $\varphi\in L^{\infty}_{-(s-s_1)}(\real)$ and $\varphi=T_0^0\varphi$, then $\varphi=0$.
\end{lem}
\begin{proof}
At this point, we can basically repeat the proof of \Cref{lem: crucial iden} for the case $k<0$. Only in the present case, $k=0$. All calculations can be justified now that we know the decay estimate \cref{est: T00 phi decay}. One has from \Cref{lem: diff T00} that
\eqn
\chi_{\real^+}\widehat{u\varphi}=\xi\hat\varphi.
\eeqn
By \cref{est: T00 phi decay} and the conditions on $u$ and $\varphi$, we have $(1+|x|)u\varphi\in L^1$, Hence $\widehat{u\varphi}\in C^1(\real)$. Recall that $\widehat{u\varphi}(0)=0$ by \Cref{lem: vanish int 1}. Hence
\eqn
\hat\varphi(0+) = \lim_{\xi\to 0^+}\frac{\widehat{u\varphi}(\xi)}{\xi}=\widehat{u\varphi}'(0).
\eeqn
We repeat the argument in \Cref{lem: crucial iden} to get \cref{eq: varphi L2}, which now becomes
\eqn
2\pi \int_\real |\varphi|^2~dx=0.
\eeqn
\end{proof}

We can now prove existence of the modified Jost functions.

\begin{thm}\label{thm: existence mod Jost}
Let $s>s_1>\frac12$, $u\in L^2_s(\real)$, $k\in (\comp\setminus [0,\infty))\cup(\real^+\pm 0i)\cup\{0\}$, $\lambda\ge 0$, and $\chi(\xi)$ satisfy \cref{eq: chi nonzero im}. Then there is $k_0>0$ such that for all $|k|,\lambda<k_0$, there exist unique solutions $m_1^0(x,k), m_e^0(x,\lambda\pm 0i)\in L^{\infty}_{-(s-s_1)}(\real)$ to \cref{eq: mod m1} and \cref{eq: mod me}. Furthermore, there are $C>0$ and $\epsilon\in (0,1)$ such that
\eqn\label{eq: m1 asym}
\|m_1^0(k)-m_1^0(0)\|_{L^{\infty}_{-(s-s_1)}}\le C|k|^{\epsilon},
\eeqn
\eqn\label{eq: me asym}
\|m_e^0(\lambda\pm 0i)-m_1^0(0)\|_{L^{\infty}_{-(s-s_1)}}\le C\lambda^{\epsilon}.
\eeqn
\end{thm}
\begin{proof}
\Cref{lem: Tk0-T00}, \Cref{lem: mod L2 vanishing} and the Fredholm alternative theorem imply the invertibility of $I-T_k^0$, from which we obtain existence and uniqueness of $m_1^0$ and $m_e^0$. The asymptotic bounds \cref{eq: m1 asym}, \cref{eq: me asym} follow from \Cref{lem: Tk0-T00}, and the fact that $\|e^{i\lambda x}-1\|_{L^{\infty}_{-(s-s_1)}}\le C\lambda^{\epsilon}$.
\end{proof}

We can obtain asymptotic formulas for the original Jost functions and scattering coefficients as $k$ approaches $0$, since the original Jost functions can be expressed in terms of the modified Jost functions as in \cref{eq: m1 vs m1 mod} and \cref{eq: me vs me mod}. At this point, it is useful to make a division between two distinct cases.

\begin{definition}
Let $u\in L^2_s(\real)$, and let $m_1^0(x,0)$ be constructed as in \Cref{thm: existence mod Jost}. $u$ is called a generic potential if $\int_\real u(x) m_1^0(x,0)~dx\ne 0$, or a non-generic potential if $\int_\real u(x) m_1^0(x,0)~dx=0$.
\end{definition}

Notice that $m_1^0(x,0)$ actually depends on the choice of the cutoff function $\chi(\xi)$ when we regularize $T_k$ to $T_k^0$. However, the definition of genericity does not depend on the choice of $\chi(\xi)$, as is shown in the following lemma. To state the lemma, let $\chi^{(1)}(\xi)$ and $\chi^{(2)}(\xi)$ be smooth functions on $[0,\infty)$, which are identically equal to $1$ on $[0,1]$, and identically equal to $0$ on $[2,\infty)$. We use the notation $G_0^{0(1)}$, $G_0^{0(2)}$, etc. to denote the corresponding objects constructed using $\chi^{(1)}(\xi)$ and $\chi^{(2)}(\xi)$.
\begin{lem}
Let $\chi^{(1)}(\xi)$ and $\chi^{(2)}(\xi)$ be given as above, and $u$ be given as in \Cref{thm: existence mod Jost}. Suppose $\chi^{(1)}(\xi)$ satisfy \cref{eq: chi nonzero im}, and let $m_1^{0(1)}(x,0)$ be the Jost solution constructed in \Cref{thm: existence mod Jost}. If 
\eqn\label{eq: generic 1}
\int_{\real}u(x)m_1^{0(1)}(x,0)~dx =0, 
\eeqn
then
\begin{enumerate}[(a)]
\item $I-T_0^{0(2)}$ is invertible on $L^{\infty}_{-(s-s_1)}(\real)$.
\item $m_1^{0(2)}(x,0)=(I-T_0^{0(2)})^{-1}1$ is the same as $m_1^{0(1)}(x,0)$.
\end{enumerate}
\end{lem}
\begin{proof}
Part (a) is of course already established in \Cref{thm: existence mod Jost} if $\chi^{(2)}(\xi)$ satisfies \cref{eq: chi nonzero im}. The interesting point, however, is that when $u$ is non-generic, $I-T_0^{0(2)}$ must still be invertible when $\int_1^2\frac{\chi^{(2)}(\xi)}{\xi}~d\xi$ is real. To prove that, we need to show that any $\varphi\in L^{\infty}_{-(s-s_1)}$ satisfying $\varphi = T_0^{0(2)}\varphi$ must be zero. Examining the sequence of lemmas before \Cref{thm: existence mod Jost}, we find that the only place \cref{eq: chi nonzero im} was used was to establish the key vanishing integral \cref{eq: vanish int 1}, which we now show by different means. In fact, we observe that
\eqn
G_0^{0(2)}(x) = G_0^{0(1)}(x)+\frac{1}{2\pi}\int_1^2\frac{\chi^{(1)}(\xi)-\chi^{(2)}(\xi)}{\xi}~d\xi = G_0^{0(1)}(x) + c. 
\eeqn
We have
\eqn
\varphi=T_0^{0(2)}\varphi = G_0^{0(2)}*(u\varphi) = c\langle \varphi, u\rangle + G_0^{0(1)}*(u\varphi)=c\langle \varphi, u\rangle + T_0^{0(1)}\varphi.
\eeqn
Thus $\varphi = c\langle\varphi, u\rangle m_1^{0(1)}(0)$. Since $\langle m_1^{0(1)}(0),u\rangle=0$ by \cref{eq: generic 1}, $\langle \varphi, u\rangle =0$, which is the key vanishing integral \cref{eq: vanish int 1}. Part (a) can be proven by the same arguments following \cref{eq: vanish int 1}.

To show part (b), we observe that
\begin{align*}
G_0^{0(2)}*(um_1^{0(1)}(0)) &= G_0^{0(1)}*(um_1^{0(1)}(0)) + c\langle m_1^{0(1)}(0),u\rangle\\
&= G_0^{0(1)}*(um_1^{0(1)}(0))\\
& = m_1^{0(1)}(0)-1.
\end{align*}
The result now follows by uniqueness.
\end{proof}

We are now ready to state and compute the asymptotics of the Jost functions and scattering coefficients as $k$ approaches $0$.

\begin{thm}\label{thm: lim k 0}
Let $s>s_1>\frac12$, $u\in L^2_s(\real)$, $k\in (\comp\setminus[0,\infty))\cup (\real^+\pm 0i)$, and $\lambda>0$. Let $m_1(x,k)$ and $m_e(x,\lambda\pm 0i)$ be constructed as in \Cref{thm: existence of Jost soln}, and let $\Gamma(\lambda)$ and $\beta(\lambda)$ be defined as in \Cref{lem: def scattering coeff}. Let $m_0^0(x,0)$ be constructed as in \Cref{thm: existence mod Jost}. Then there exists $\epsilon\in (0,1)$ such that
as $k$ approaches $0$ and $\lambda$ approaches $0^+$, \begin{enumerate}[(a)]
\item if $u$ is a generic potential,
\eqn
m_1(k) = \frac{2\pi}{\langle m_1^0(0),u\rangle \log k}m_1^0(0) + O\left(\frac{1}{|\log^2 k|}\right),
\eeqn
\eqn
m_e(\lambda\pm 0i) =\frac{2\pi}{\langle m_1^0(0),u\rangle \log (\lambda\pm 0i)}m_1^0(0) + O\left(\frac{1}{|\log^2 \lambda|}\right),
\eeqn
\eqn
\Gamma(\lambda) = 1+\frac{2\pi i}{\log(\lambda+0i)} + O\left(\frac{1}{|\log^2 \lambda|}\right),
\eeqn
\eqn
\beta(\lambda) = \frac{2\pi i}{\log(\lambda+0i)} + O\left(\frac{1}{|\log^2 \lambda|}\right);
\eeqn
\item if $u$ is a non-generic potential,
\eqn
m_1(k) = m_1^0(0) + O(|k|^\epsilon |\log k|),
\eeqn
\eqn
m_e(\lambda\pm 0i) = m_1^0(0) + O(\lambda^\epsilon |\log \lambda|),
\eeqn
\eqn
\Gamma(\lambda) = 1 + O(\lambda^{\epsilon} |\log \lambda|),
\eeqn
\eqn
\beta(\lambda) = O(\lambda^{\epsilon} |\log \lambda|).
\eeqn
\end{enumerate}
Here the function $\log$ takes the principle branch, with a branch cut on $[0,\infty)$. The big $O$ notation has the usual meaning in equations involving $\Gamma$ and $\beta$, but holds in the sense of $L^{\infty}_{-(s-s_1)}(\real)$ norm in equations involving $m_1$ and $m_e$.
\end{thm}
\begin{proof}
The proof is a straightforward calculation using \cref{eq: m1 vs m1 mod}, \cref{eq: me vs me mod}, \cref{eq: m1 asym}, \cref{eq: me asym}, and the definitions of $\Gamma(\lambda)$ and $\beta(\lambda)$. We only need to observe that
\begin{align}
l(k) &= \frac1{2\pi}\int_0^{\infty}\frac{\chi(\xi)}{\xi-k}~d\xi\notag\\
&= \frac{1}{2\pi}\int_0^1\frac{1}{\xi-k}~d\xi + \frac1{2\pi}\int_1^2\frac{\chi(\xi)}{\xi-k}~d\xi\notag\\
&= \frac1{2\pi}\log k + h(k),
\end{align}
where $\log$ takes the principle branch with branch cut $[0,\infty)$, and $h(k)$ is analytic around $k=0$.
\end{proof}

%%%%%%%%%%%%%%%%%%%%%%%%%%%%%%%%%%%%%%%%%%%%%%%%%%%%%%%%%%%%%%%%%%%%%%
\section{Asymptotic behavior near $k=\infty$}\label{sec: k infty limit}
%%%%%%%%%%%%%%%%%%%%%%%%%%%%%%%%%%%%%%%%%%%%%%%%%%%%%%%%%%%%%%%%%%%%%%

In this section, we obtain asymptotic formulas for the Jost functions and scattering coefficients as $k$ approaches $\infty$ in the cut plane. The situation of large $k$ limit is very different from that of small $k$ limit discussed in \Cref{sec: k 0 limit}. As we will see in the following, the operator $I-T_k$ can be inverted explicitly when $|k|$ is sufficiently large. This allows explicit calculation and estimation of error. Similar to the situation of the Fourier transform, high regularity and decay of the potential $u$ implies high regularity and decay of the scattering coefficients as $k$ tends to $\infty$. The precise assumptions on $u$ and the corresponding decay estimates on the scattering coefficients may vary according to the needs in application. As an example, we work in this section with the following three types of assumptions on $u$: $u\in L^2_s(\real)$ with $s>\frac12$, $u\in H^s_s(\real)$ with $s>\frac12$, and $u$ in the Schwartz class $\mathcal{S}$. The first type of spaces is to keep the same assumption on $u$ as in the previous sections. The second type of spaces will provide the proper assumption to obtain a higher order term for $m_1(x,k)$. Finally, the choice of the Schwartz class will allow us to see how rapid decay of the scattering coefficients may be obtained, without having to formulate the regularity and decay assumptions on $u$ too carefully.

To begin, let's use the weakest of the three types of assumptions: $u\in L^2_s(\real)$ and show how $I-T_k$ can be inverted explicitly on $L^{\infty}_{-(s-s_1)}(\real)$,
when $s>s_1>\frac12$. First assume $k$ is in a fixed Stolz angle away from the positive real line. In other words, there exists $\alpha \in (0,\frac\pi 2)$ such that $$|\text{Im }k|\ge(\tan \alpha)\text{Re }k.$$ For any such $k$ and any $\xi>0$, $|\xi-k|$ and $\xi+|k|$ are comparable: $$0<\frac1{C_{\alpha}}\le \frac{|\xi-k|}{\xi+|k|}\le C_{\alpha}.$$ Therefore by the definition of $T_k$ and $G_k$ given in \cref{def: T lambda} and \cref{def: G},
\eqn
\|T_k \varphi\|_{\infty}\le \|G_k\|_2\|u\varphi\|_2\le C\|G_k\|_2\|\varphi\|_{L^{\infty}_{-(s-s_1)}},
\eeqn
where
\eqn
\|G_k\|_2 \le C\left(\int_0^{\infty}\frac{1}{(\xi+|k|)^2}~d\xi\right)^{\frac12}\le \frac{C}{\sqrt {|k|}}.
\eeqn
It follows that $\|T_k\|_{L^{\infty}_{-(s-s_1)}\to L^{\infty}}\le \frac{C}{\sqrt {|k|}}$. Therefore $(I-T_k)^{-1}=\sum_{n=0}^{\infty}T_k^n$ when $|k|$ is large. To invert $I-T_k$ when $k$ is close to the positive real line, we write $T_k = S_k - \widetilde{T}_k$ by \cref{eq: G_lambda ep decomp}, where for $k=\lambda\pm \mu i$, with $\lambda>0$, $\mu\ge 0$:
\eqn
S_k\varphi(x) = i\int_{\mp \infty}^x e^{ik(x-y)}u(y)\varphi(y)~dy,
\eeqn
and
\eqn\label{def: tilde T}
\widetilde{T}_k\varphi = \widetilde{G}_{k}*(u\varphi),\quad \widetilde{G}_{k}=\frac1{2\pi}\int_{-\infty}^0\frac{e^{ix\xi}}{\xi-k}~d\xi.
\eeqn
Now that $k=\lambda\pm \mu i$ with $\lambda>0$, $k$ is in a fixed Stolz angle away from the negative real line. By the same argument as above, we have $\|\widetilde{T}_k\|_{L^{\infty}_{-(s-s_1)}\to L^{\infty}}\le \frac{C}{\sqrt {|k|}}$. On the other hand, $I-S_k$ can be inverted explicitly by solving an ODE. In fact, we can rewrite
\eqn
\varphi = S_k\varphi + g = g+i\int_{\mp \infty}^x e^{ik(x-y)}u(y)\varphi(y)~dy
\eeqn
as
\eqn\label{eq: Sk inv ode 1}
(\varphi-g)e^{-ikx} = i\int_{\mp \infty}^x e^{-iky}u(y)\varphi(y)~dy.
\eeqn
Differentiating with respect to $x$ and rearranging terms using an integrating factor, we get
\eqn\label{eq: Sk inv ode 2}
[e^{-i\int_{\mp\infty}^x u(t)~dt}e^{-ikx}(\varphi-g)]_x = ie^{-i\int_{\mp\infty}^xu(t)~dt}e^{-ikx}ug
\eeqn
By \cref{eq: Sk inv ode 1}, $\varphi(x)-g(x)\to 0$ as $x\to \mp \infty$. Hence we may integrate \cref{eq: Sk inv ode 2} from $\mp \infty$ and get
\eqn\label{eq: I-Sk inv}
\varphi(x) = g(x) +i\int_{\mp \infty}^x e^{ik(x-y)}e^{i\int_y^x u(t)~dt}u(y)g(y)~dy.
\eeqn
The right hand side of \cref{eq: I-Sk inv} is $(I-S_k)^{-1}g$. It is easy to see that the operator norm of $(I-S_k)^{-1}$ is bounded uniformly in $k$ for $|k|$ large.
Combining the calculation above, we may write
\begin{align}
(I-T_k)^{-1}&=(I-S_k+\widetilde{T}_k)^{-1}=(I+(I-S_k)^{-1}\widetilde{T}_k)^{-1}(I-S_k)^{-1}\notag\\
&= \sum_{n=0}^{\infty}(-(I-S_k)^{-1}\widetilde{T}_k)^{n}(I-S_k)^{-1}.
\end{align}
We have thus proved
\begin{lem}\label{lem: inv forms}
Let $s>s_1>\frac12$ and $u\in L^2_s(\real)$. Let $k\in (\comp\setminus[0,\infty))\cup (\real^+\pm 0i)$. There exists $k_0>0$ such that for $|k|>k_0$, $I-T_k$ is invertible on $L^{\infty}_{-(s-s_1)}(\real)$, and
\begin{enumerate}[(a)]
\item if $k$ is in a fixed Stolz angle away from the positive real line, i.e. there exists $\alpha \in (0,\frac\pi 2)$ such that $$|\text{Im }k|\ge(\tan \alpha)\text{Re }k,$$ then 
\eqn\label{eq: Tk left plane bound}
\|T_k\|_{L^{\infty}_{-(s-s_1)}\to L^{\infty}}\le \frac{C_{\alpha}}{\sqrt{|k|}},
\eeqn 
and
\eqn\label{eq: I-Tk inv 1}
(I-T_k)^{-1} = \sum_{n=0}^{\infty}T_k^n,
\eeqn
\item if $k=\lambda\pm i\mu$, with $\lambda>0$, $\mu\ge 0$, then for $\widetilde{T}_k$ given in \cref{def: tilde T}, and 
\eqn
R_k\varphi (x)= i\int_{\mp \infty}^x e^{ik(x-y)}e^{i\int_y^x u(t)~dt}u(y)\varphi(y)~dy,
\eeqn
we have 
\eqn\label{eq: Tk Rk bound}
\|\widetilde T_k\|_{L^{\infty}_{-(s-s_1)}\to L^{\infty}}\le \frac{C}{\sqrt{|k|}},\quad \|R_k\|_{L^{\infty}_{-(s-s_1)}\to L^{\infty}}\le C,
\eeqn
and 
\eqn\label{eq: I-Tk inv 2}
(I-T_k)^{-1}=\sum_{n=0}^{\infty}(-(I+R_k)\widetilde{T}_k)^n (I+R_k).
\eeqn
\end{enumerate}
\end{lem}

The calculation of the scattering coefficients will be simplified by the following lemma.

\begin{lem}\label{lem: int simplify}
Let $u$ and $R_{\lambda+0i}$ be given as in \Cref{lem: inv forms}. If $\varphi\in L^{\infty}_{-(s-s_1)}(\real)$, then
\eqn\label{eq: int simplify}
\int_\real u(x)e^{-i\lambda x}[(I+R_{\lambda+0i})\varphi ](x)~dx = \int_\real u(x)e^{-i\lambda x}e^{i\int_x^\infty u(t)~dt}\varphi (x)~dx.
\eeqn
\end{lem}
\begin{proof}
Recall that $u\in L^1$ and $u\varphi\in L^1$ by the conditions on $u$ and $\varphi$. Therefore by Fubini's theorem
\begin{align}
&~\int_\real u(x)e^{-i\lambda x}[R_{\lambda+0i}\varphi](x)~dx\notag\\
= &~\int_\real u(x)e^{-i\lambda x}i\int_{-\infty}^x e^{i\lambda(x-y)}e^{i\int_y^x u(t)~dt}u(y)\varphi(y)~dy~dx\notag\\
=&~ \int_\real u(y)\varphi(y) e^{-i\int_{-\infty}^y u(t)~dt}e^{-i\lambda y}\int_y^\infty \left(e^{i\int_{-\infty}^x u(t)~dt}\right)_x~dx~dy\notag\\
=&~\int_\real u(y)e^{-i\lambda y}e^{i\int_y^\infty u(t)~dt}\varphi(y)~dy - \int_\real u(y)e^{-i\lambda y}\varphi(y)~dy.
\end{align}
\Cref{eq: int simplify} thus follows.
\end{proof}

We want to use the inversion formulas in \Cref{lem: inv forms} to compute asymptotics of $m_1(x,k)$, $m_e(x,\lambda\pm 0i)$, $\Gamma(\lambda)$, $\beta(\lambda)$, and $f(\lambda)$. By relations \cref{eq: me pm rel} and \cref{eq: f beta rel}, we only need to study $m_1(x,k)$, $m_e(x,\lambda + 0i)$, $\Gamma(\lambda)$, and $\beta(\lambda)$.

\begin{thm}\label{thm: lim infty}
Let $s>\frac12$, $u\in L^2_s(\real)$, $k\in (\comp\setminus[0,\infty))\cup (\real^+\pm 0i)$, and $\lambda>0$. Then
\eqn\label{eq: lim m1 infty}
\lim_{k\to \infty}m_1(x,k) = 1, 
\eeqn
\eqn\label{eq: lim me infty}
\lim_{\lambda \to \infty}m_e(x,\lambda+0i) - e^{i\lambda x}e^{i\int_{-\infty}^x u(t)~dt} = 0,
\eeqn
\eqn\label{eq: lim Gamma infty}
\Gamma(\lambda) - e^{i\int_\real u(t)~dt} =O\left(\frac1{\lambda}\right) \text{ as }\lambda\to \infty,
\eeqn
and
\eqn\label{eq: lim beta infty}
\beta(\lambda) \in L^2(a,\infty),~\lim_{\lambda\to \infty}\beta(\lambda)=0.
\eeqn
Here the limits for $m_1$ and $m_e$ hold in $L^{\infty}(\real)$ norm, and $a>0$ is any fixed number.
\end{thm}
\begin{proof}
We first work on $m_1(k) = (I-T_k)^{-1}1$. If $k$ is in the left half plane, we use \cref{eq: I-Tk inv 1}, and the fact that $\|\sum_{n=1}^{\infty}T_k^n 1\|_{\infty}\le \frac{C}{\sqrt{|k|}}$ to conclude \cref{eq: lim m1 infty}. If $k$ is in the right half cut plane, we use \cref{eq: I-Tk inv 2} to write
\eqn\label{eq: m1 limit infty 1}
m_1(k) = (I+R_k)1 + \sum_{n=1}^{\infty}(-(I+R_k)\widetilde{T}_k)^n (I+R_k)1,
\eeqn
and use \cref{eq: Tk Rk bound} to conclude that the infinite sum in \cref{eq: m1 limit infty 1} has $L^{\infty}$ norm bounded by $\frac{C}{\sqrt{|k|}}$. What is left to show is that $\|R_k1\|_{\infty}\to 0$ as $k$ approaches $\infty$ in the right half cut plane. For simplicity of presentation, let us work only with the case $k=\lambda+i\mu$ with $\lambda>0$, $\mu\ge0$. In the following proof, this is always assumed. Thus
\eqn
R_k1 (x)= i\int_{-\infty}^x e^{ik(x-y)}e^{i\int_y^x u(t)~dt}u(y)~dy,
\eeqn
Recall that $u\in L^1$ if $u\in L^2_s$ with $s>\frac12$, and $|e^{ik(x-y)}|\le 1$ when $x-y\ge 0$. So $\lim_{x\to -\infty}R_k1(x)=0$, and
\eqn
\lim_{x\to \infty}R_k1(x) = \begin{cases}0 &\text{ if }\mu>0,\\ i\int_\real e^{i\lambda(x-y)}e^{i\int_y^x u(t)~dt}u(y)~dy &\text{ if }\mu =0,\end{cases}
\eeqn
by the dominated convergence theorem. By the Riemann-Lebesgue lemma, 
\eqn
\lim_{\lambda\to \infty}i\int_\real e^{i\lambda(x-y)}e^{i\int_y^x u(t)~dt}u(y)~dy=0.
\eeqn
Therefore for every $\epsilon>0$, there is $k_1>0$ such that if $|k|>k_1$, $|\lim_{x\to \infty}R_k1(x)|< \epsilon$. Since $u\in L^1$, there exist finitely many points $\{x_n\}_{n=1}^N$ such that $|R_k1(x)-R_k1(y)|<\epsilon$ if none of the $x_n$'s is between $x$ and $y$. As we have already controlled $R_k1(x)$ when $x$ is at $\pm \infty$, it remains to control $R_k1(x)$ if $x$ is one of $\{x_n\}_{n=1}^N$. For each fixed $x_n$, we have
\eqn
R_k1(x_n) = i\int_{-\infty}^{x_n} e^{ik(x_n-y)}e^{i\int_y^{x_n} u(t)~dt}u(y)~dy.
\eeqn
We claim that $\lim_{k\to \infty}R_k1(x_n)=0$. In fact, one can mimic the proof of the Riemann-Lebesgue lemma, and approximate $u$ in $L^1$ by a $C_0^{\infty}$ function $g$, while integrating 
\eqn
i\int_{-\infty}^{x_n} e^{ik(x_n-y)}e^{i\int_y^{x_n} u(t)~dt}g(y)~dy
\eeqn
by parts to get
\eqn\label{eq: R-L control}
-\frac{1}{k}g(x_n)+\frac{1}{k}\int_{-\infty}^{x_n} e^{ik(x_n-y)}\left(e^{i\int_y^{x_n} u(t)~dt}g(y)\right)_y~dy,
\eeqn
which obviously tends to $0$ as $k$ tends to $\infty$. Thus by enlarging $k_1$ finitely many times, we get for $|k|>k_1$, $|R_k(x)|<2\epsilon$ for all $x$. This completes the proof of \cref{eq: lim m1 infty}. By a similar argument as above, the asymptotic behavior of $m_e(\lambda+0i)$ is given by $(I+R_{\lambda+0i})\e$, which in this case can be computed explicitly, as
\begin{align}
[R_{\lambda+0i}\e](x) &=i\int_{- \infty}^x e^{i\lambda(x-y)}e^{i\int_y^x u(t)~dt}u(y)e^{i\lambda y}~dy\notag\\
&= -e^{i\lambda x}e^{i\int_{-\infty}^xu(t)~dt} \int_{-\infty}^x \left(e^{-i\int_{-\infty}^y u(t)~dt}\right)_y~dy\notag\\
&= -e^{i\lambda x}+ e^{i\lambda x}e^{i\int_{-\infty}^xu(t)~dt}.
\end{align}
Hence $[(I+R_{\lambda+0i})\e](x)=e^{i\lambda x}e^{i\int_{-\infty}^xu(t)~dt}$. This finishes the proof of \cref{eq: lim me infty}.

In order to obtain enough decay estimates of the scattering coefficients, we need to expand $m_1(\lambda+0i)$ and $m_e(\lambda+0i)$ by one more order. By \cref{eq: Tk Rk bound}, we have
\begin{align}
&~m_1(\lambda+0i) = (I-T_{\lambda+0i})^{-1}1 \notag\\
=&~ (I+R_{\lambda+0i})1 - (I+R_{\lambda+0i})\widetilde{T}_{\lambda+0i}(I+R_{\lambda+0i})1 + O\left(\frac1{\lambda}\right),
\end{align}
and
\begin{align}
&~m_e(\lambda+0i) = (I-T_{\lambda+0i})^{-1}\e \notag\\
=&~ (I+R_{\lambda+0i})\e - (I+R_{\lambda+0i})\widetilde{T}_{\lambda+0i}(I+R_{\lambda+0i})\e + O\left(\frac1{\lambda}\right).
\end{align}
By the definition of $\Gamma(\lambda)$ and $\beta(\lambda)$ given in \cref{def: Gamma} and \cref{def: beta}, we have
\begin{align}
\Gamma(\lambda) = &~1+i\int_{\real}u(x)e^{-i\lambda x}[(I+R_{\lambda+0i})\e](x)~dx \notag\\ 
&~- i\int_{\real}u(x)e^{-i\lambda x}[(I+R_{\lambda+0i})\widetilde T_{\lambda+0i}(I+R_{\lambda+0i})\e](x)~dx +O\left(\frac1{\lambda}\right),
\end{align}
and
\begin{align}
\beta(\lambda) = &~i\int_{\real}u(x)e^{-i\lambda x}[(I+R_{\lambda+0i})1](x)~dx \notag\\ 
&~- i\int_{\real}u(x)e^{-i\lambda x}[(I+R_{\lambda+0i})\widetilde T_{\lambda+0i}(I+R_{\lambda+0i})1](x)~dx +O\left(\frac1{\lambda}\right).
\end{align}
We first work on $\Gamma(\lambda)$. By \Cref{lem: int simplify}, 
\begin{align}
i\int_{\real}u(x)e^{-i\lambda x}[(I+R_{\lambda+0i})\e](x)~dx &= \int_\real iu(x) e^{i\int_x^\infty u(t)~dt}~dx\notag\\
&= -\int_\real \left(e^{i\int_x^\infty u(t)~dt}\right)_x~dx\notag\\
&= e^{i\int_\real u(t)~dt}-1,
\end{align}
and
\begin{align}
&~\int_{\real}u(x)e^{-i\lambda x}[(I+R_{\lambda+0i})\widetilde T_{\lambda+0i}(I+R_{\lambda+0i})\e](x)~dx\notag\\
=&~\int_{\real}u(x)e^{-i\lambda x}e^{i\int_x^\infty u(t)~dt}[\widetilde T_{\lambda+0i}(I+R_{\lambda+0i})\e](x)~dx,
\end{align}
which is bounded by $\|u\|_2\|\widetilde T_{\lambda+0i}(I+R_{\lambda+0i})\e\|_2$. By the Plancherel identity, 
\begin{align}
\|\widetilde T_{\lambda+0i}(I+R_{\lambda+0i})\e\|_2&\le C \bigg\|\frac{\chi_{\real^-}(\xi)}{\xi-\lambda}F(u(I+R_{\lambda+0i})\e)\bigg\|_2\notag\\
&\le \frac C{\lambda}\|F(u(I+R_{\lambda+0i})\e)\|_2\notag\\
&\le \frac C{\lambda}\|u(I+R_{\lambda+0i})\e\|_2 \le \frac C{\lambda}\|u\|_2.
\end{align}
Hence
\eqn
\Gamma(\lambda) = 1+e^{i\int_\real u(t)~dt}-1 + O\left(\frac1{\lambda}\right)=e^{i\int_\real u(t)~dt}+ O\left(\frac1{\lambda}\right).
\eeqn
This proves \cref{eq: lim Gamma infty}. The calculation of $\beta(\lambda)$ differs basically only in the main term
\eqn
\int_\real u(x)e^{-i\lambda x}[(I+R_{\lambda+0i})1](x)~dx=\int_\real u(x)e^{-i\lambda x}e^{i\int_x^\infty u(t)~dt}~dx.
\eeqn
The result \cref{eq: lim beta infty} follows from the fact that $u\in L^1\cap L^2$.
\end{proof}

Our next result shows that a little more information on the asymptotic behavior of $m_1(x,k)$ may be obtained by imposing slightly stronger regularity assumptions on $u$. For this result, $k$ is allowed to approach $\infty$ in a fixed Stolz angle away from the positive real line.

\begin{thm}\label{thm: lim m1 hot 1}
Let $s>\frac12$, and $u\in H^s_s(\real)$. Suppose there exists $\alpha\in (0,\frac\pi2)$ such that $|\text{Im }k|>(\tan \alpha)\text{Re }k$. Then there exists $\epsilon>0$ such that
\eqn\label{eq: lim m1 hot 1} 
m_1(x,k) = 1-\frac{C_+u(x)}{k}+O\left(\frac1{|k|^{1+\epsilon}}\right) \text{ as }k\to\infty.
\eeqn
Here the big $O$ notation holds in the sense of $L^{\infty}(\real)$.
\end{thm}
\begin{proof}
If $k$ is in the left half plane and $|k|$ is sufficiently large, we use \cref{eq: I-Tk inv 1} to get
\eqn
m_1(k)= (I-T_k)^{-1}1 = \sum_{n=0}^{\infty}T_k^n1.
\eeqn
Since $u\in H^s_s(\real)$, $w^{s}\hat u\in L^2$, and $\hat u\in L^1$. It follows that the $L^{\infty}$ norm of
\eqn
(T_k1)(x) = \frac1{2\pi}\int_0^{\infty}\frac{e^{ix\xi}}{\xi-k}\hat u(\xi)~d\xi
\eeqn
is bounded by $\frac C{|k|}$. Therefore by \cref{eq: Tk left plane bound},
\eqn
m_1(k) = 1 + T_k 1 + O\left(\frac1{|k|^{\frac32}}\right).
\eeqn
We now write $T_k1$ as
\begin{align}\label{eq: Tk1 1}
(T_k1)(x) &= -\frac{1}{2\pi k}\int_0^\infty e^{ix\xi}\hat u(\xi)~d\xi + \frac1{2\pi k}\int_0^\infty\frac{\xi e^{ix\xi}}{\xi-k}\hat u(\xi)~d\xi\notag\\
&= -\frac{C_+u(x)}{k} + \frac1{2\pi k}\int_0^{\infty}\frac{\xi^{1-\epsilon}e^{ix\xi}}{\xi-k}[\xi^{\epsilon}\hat u(\xi)]~d\xi.
\end{align}
We estimate the last integral as follows. Since $w^s\hat u \in L^2$, by choosing $\epsilon>0$ sufficiently small, we can make $\xi^{\epsilon}\hat u(\xi)\in L^1$. The integral is therefore bounded by
\eqn
C\bigg\|\frac{\chi_{\real^+}(\xi)\xi^{1-\epsilon}}{\xi-k}\bigg\|_{\infty}\le C\bigg\|\frac{\chi_{\real^+}(\xi)\xi^{1-\epsilon}}{\xi+|k|}\bigg\|_{\infty}\le \frac{C}{|k|^{\epsilon}}.
\eeqn
Thus $T_k1=-\frac{C_+u}{k}+O(\frac1{|k|^{1+\epsilon}})$, and the result follows. 
\end{proof}

Our last result exemplifies how fast decay of the scattering coefficients can be obtained when $u$ is assumed to be smooth with rapid decay.

\begin{thm}\label{thm: lim infty schwartz}
Suppose $u\in \mathcal{S}$, the Schwartz class of rapidly decaying functions, and $k\in (\comp\setminus [0,\infty))\cup (\real^+\pm 0i)$. Then there exists $k_0>0$ and $C>0$ such that  
\eqn\label{eq: lim m1 hot 2}
\bigg\|m_1(x,k)-1+\frac{C_+u(x)}{k}\bigg\|_{\infty}\le \frac C{|k|^2}
\eeqn
for $|k|>k_0$, 
and for every positive integer $N$, there exists $C_N>0$ such that
\eqn\label{eq: lim me hot 2}
\|m_e(x,\lambda+0i)-e^{i\lambda x}e^{i\int_{-\infty}^x u(t)~dt}\|_{\infty}\le \frac{C_N}{\lambda^N},
\eeqn
\eqn\label{eq: lim Gamma beta hot 2}
|\Gamma(\lambda)-e^{i\int_\real u(t)~dt}|\le \frac{C_N}{\lambda^N},\text{ and }|\beta(\lambda)|\le \frac{C_N}{\lambda^N},
\eeqn
for all $\lambda >k_0$.
\end{thm}
\begin{proof}
The improvement from \cref{eq: lim m1 hot 1} to \cref{eq: lim m1 hot 2} is twofold: $\epsilon$ is improved to $1$, and the restriction on the Stolz angle is removed. We first assume $k$ is in the left half plane. The choice of $\epsilon$ in the proof of \Cref{thm: lim m1 hot 1} is used only to make $\xi^{\epsilon}\hat u(\xi)\in L^1$. It is clear that we may choose $\epsilon=1$ now that $u\in \mathcal{S}$.

To remove the restriction on the Stolz angle, let's assume $k$ is in the right half plane with $|k|$ sufficiently large. This time we use \cref{eq: I-Tk inv 2} to write
\eqn
m_1(k)= (I-T_k)^{-1}1 = \sum_{n=0}^{\infty}(-(I+R_k)\widetilde{T}_k)^n (I+R_k)1.
\eeqn
We again work only with the case $k=\lambda+i\mu$ with $\lambda>0$, $\mu\ge 0$ and compute
\begin{align}
&~[(I+R_k)1](x) \notag\\
= &~1+i\int_{-\infty}^x e^{ik(x-y)}e^{i\int_y^x u(t)~dt}u(y)~dy\notag\\
= &~ 1-\frac{u(x)}{k} +\frac1k\int_{-\infty}^x e^{ik(x-y)}\left(e^{i\int_y^x \int u(t)~dt}u(y)\right)_y~dy\notag\\
=&~1-\frac{u(x)}{k} +O\left(\frac1{|k|^2}\right).%-\frac1{ik^2}\left(e^{i\int_y^x \int u(t)~dt}u(y)\right)_y\bigg|_{y=x}+\frac1{ik^2}\int_{-\infty}^x e^{ik(x-y)}\left(e^{i\int_y^x \int u(t)~dt}u(y)\right)_{yy}~dy
\end{align}
Here we have used integration by parts to compute the integral and used it one more time to estimate the remainder. It follows that
\begin{align}
&~[\widetilde{T}_k(I+R_k)1](x) \notag\\
=&~ \left[\widetilde T_k \left(1-\frac u k\right) \right](x)+ O\left(\frac1{|k|^2}\right) \notag\\
=&~\frac1{2\pi}\int_{-\infty}^0 \frac{e^{ix\xi}}{\xi-k}F\left(u\left(1-\frac u k\right)\right)(\xi)~d\xi + O\left(\frac1{|k|^2}\right) \notag\\
=&~\frac1{2\pi}\int_{-\infty}^0 \frac{e^{ix\xi}}{\xi-k}\hat u(\xi)~d\xi + O\left(\frac1{|k|^2}\right)\notag\\
=&~ -\frac{C_-u(x)}k+ O\left(\frac1{|k|^2}\right).
\end{align}
Therefore 
\eqn
m_1(k) = 1-\frac u k +\frac{C_-u}k+O\left(\frac1{|k|^2}\right)=1-\frac{C_+u}k +O\left(\frac1{|k|^2}\right). 
\eeqn
This completes the proof of \cref{eq: lim m1 hot 2}.

Next, we use \cref{eq: I-Tk inv 2} to write
\eqn
m_e(\lambda+0i)= (I-T_{\lambda+0i})^{-1}\e = \sum_{n=0}^{\infty}(-(I+R_{\lambda+0i})\widetilde{T}_{\lambda})^n (I+R_{\lambda+0i})\e,
\eeqn
and recall from the proof of \Cref{thm: lim infty} that $[(I+R_{\lambda+0i})\e](x) = e^{i\lambda x}e^{i\int_{-\infty}^x u(t)~dt}$. Thus 
\begin{align}
&~[\widetilde T_{\lambda}(I+R_{\lambda+0i})\e](x)\notag\\
= &~\frac1{2\pi}\int_{-\infty}^0\frac{e^{ix\xi}}{\xi-\lambda}F\left(u(y)e^{i\lambda y}e^{i\int_{-\infty}^y u(t)~dt}\right)(\xi)~d\xi\notag\\
= &~\frac1{2\pi}\int_{-\infty}^0\frac{e^{ix\xi}}{\xi-\lambda}F\left(u(y)e^{i\int_{-\infty}^y u(t)~dt}\right)(\xi-\lambda)~d\xi\notag\\
= &~\frac1{2\pi}\int_{-\infty}^{-\lambda}\frac{e^{ix\xi}}{\xi}F\left(u(y)e^{i\int_{-\infty}^y u(t)~dt}\right)(\xi)~d\xi.
\end{align}
Since $F\left(u(y)e^{i\int_{-\infty}^y u(t)~dt}\right)$ is also in the Schwartz class, we have $\|\widetilde T_{\lambda}(I+R_{\lambda+0i})\e\|_{\infty}\le\frac{C_N}{\lambda^N}$, and \cref{eq: lim me hot 2} follows. The asymptotic bound on $\Gamma(\lambda)$ follows immediately from \cref{eq: lim me hot 2}. Finally, to find the bound on $\beta(\lambda)$, we write
\eqn\label{eq: beta as infinite sum}
\beta(\lambda) = \sum_{n=0}^{\infty}\int_\real u(x) e^{-i\lambda x}[(-(I+R_{\lambda+0i})\widetilde{T}_{\lambda})^n (I+R_{\lambda+0i})1](x)~dx.
\eeqn
By \Cref{lem: int simplify},
\begin{align}
&~\int_\real u(x) e^{-i\lambda x}[((I+R_{\lambda+0i})\widetilde{T}_{\lambda})\varphi](x)~dx\notag\\
=&~ \int_\real u(x) e^{-i\lambda x}e^{i\int_x^\infty u(t)~dt}[\widetilde{T}_{\lambda}\varphi](x)~dx\notag\\
=&~ F\left(u(x)e^{i\int_x^\infty u(t)~dt}[\widetilde{T}_{\lambda}\varphi](x)\right)(\lambda)\notag\\
= &~\frac1{2\pi}\left[F\left(u(x)e^{i\int_x^\infty u(t)~dt}\right)*F([\widetilde{T}_{\lambda}\varphi](x))\right](\lambda)\notag\\
= &~\frac1{2\pi}\int_{-\infty}^0 F\left(u(x)e^{i\int_x^\infty u(t)~dt}\right)(\lambda-\xi)\frac{1}{\xi-\lambda}\widehat{u\varphi}(\xi)~d\xi,
\end{align}
whose $L^{\infty}$ norm is bounded by 
\eqn
C\sup_{\xi\ge \lambda}\left|F\left(u(x)e^{i\int_x^\infty u(t)~dt}\right)(\xi)\right|\|u\varphi\|_2\le \frac{C_N}{\lambda^N}.
\eeqn
Taking $\varphi$ to be $((I+R_{\lambda+0i})\widetilde{T}_{\lambda})^{n-1} (I+R_{\lambda+0i})1$, we easily obtain $|\beta(\lambda)|\le \frac{C_N}{\lambda^N}$ from \cref{eq: beta as infinite sum}.
\end{proof}

%%%%%%%%%%%%%%%%%%%%%%%%%%%%%%%%%%%%%%%%%%%%%%%%%%%%%%%%%%%%%%%%%%%%%%
\section{Time evolution of scattering data}\label{sec: time evo}
%%%%%%%%%%%%%%%%%%%%%%%%%%%%%%%%%%%%%%%%%%%%%%%%%%%%%%%%%%%%%%%%%%%%%%

In this section, we present a formal derivation of the time evolution of the Jost functions and scattering coefficients given in \cite{fokas1983inverse}, assuming $u=u(x,t)$ is sufficiently smooth with sufficiently rapid decay, and evolves with the BO equation \cref{eq: BO}. We spend no effort in justifying the change of order of derivatives with asymptotic notations. The reason that we don't try to make the steps rigorous is as follows. If our goal is to construct solutions to the Cauchy problem of the BO equation using IST, the shorter path is to evolve the scattering data by the formulas obtained formally in this section, and prove that the solution constructed by the IST indeed solves the BO equation. Therefore, although it may be possible to prove the time evolution of scattering data using the $H^s$ solution to the BO equation constructed in the PDE literature, we do not pursue that path here.

The derivation is done in two steps. In the first step, we argue that the Jost functions $m_1(k)$, $m_e(\lambda\pm 0i)$, and the eigenfunctions $\phi_j$ defined in \Cref{sec: FA IST} satisfy the following evolution equations:
\begin{align}
\partial_t \phi_j &= B_u\phi_j, \label{eq: evo phi}\\
\partial_t m_1(k) &=B_um_1(k), \label{eq: evo m1}\\
\partial_t m_e(\lambda\pm 0i) &= B_um_e(\lambda\pm 0i)-i\lambda^2m_e(\lambda\pm 0i), \label{eq: evo me}
\end{align}
where $B_u$ is defined as \cref{def: Bu}.
In the second step, we will use \cref{eq: evo phi}, \cref{eq: evo m1}, \cref{eq: evo me} to show the following time evolution for the eigenvalues $\{\lambda_j\}_{j=1}^N$, phase constants $\{\gamma_j\}_{j=1}^N$, and scattering coefficients $\Gamma(\lambda)$, $\beta(\lambda)$:
\begin{align}
\partial_t \lambda_j &=0, \label{eq: evo lambda}\\
\partial_t \gamma_j &= 2\lambda_j, \label{eq: evo gamma}\\
\partial_t \Gamma(\lambda)&=0, \label{eq: evo Gamma}\\
\partial_t \beta(\lambda) &=i\lambda^2\beta(\lambda). \label{eq: evo beta}
\end{align}

We want to use the Lax equation $\partial_t L_u + [L_u,B_u]=0$ to derive \cref{eq: evo phi}, \cref{eq: evo m1}, \cref{eq: evo me}. However, as is pointed out in \Cref{sec: FA IST}, the equivalence of the Lax equation with the BO equation has only been derived if $L_u$ and $B_u$ are regarded as operators on $\hardy^+$. We prove in the following lemma that the eigenfunctions and Jost functions are in fact boundary values of bounded analytic functions on the upper half plane. This is enough to justify the Lax equation.

\begin{lem}\label{lem: hardy infty}
Let $\phi_j(x)$ be an eigenfunction of $L_u$ corresponding to a negative eigenvalue $\lambda_j$. Let $m_1(x,k)$, $m_e(x,\lambda\pm 0i)$ be given as in \Cref{lem: equiv int diff}, and satisfy either condition (a) or (b). Then $\phi_j, m_1(x,k),m_e(x,\lambda\pm 0i)\in \hardy^{\infty,+}$ for fixed $k$ and $\lambda$.
\end{lem}
\begin{proof}
We first work on $m_1(x,k)$. If $k$ is not on $\real^+\pm 0i$, we can repeat the calculation in \Cref{lem: equiv int diff} to get \cref{eq: hat m1 hardy}, or
\eqn
m_1(x)-1 = \frac1{2\pi}\int_0^{\infty}\frac{e^{i\xi x}}{\xi-k}\widehat{um_1}(\xi)~d\xi.
\eeqn
For $z=x+iy$ with $y>0$, define
\eqn
F(z) = \frac1{2\pi}\int_0^{\infty}\frac{e^{i\xi z}}{\xi-k}\widehat{um_1}(\xi)~d\xi = \frac1{2\pi}\int_0^{\infty}\frac{e^{i\xi x}e^{-y\xi}}{\xi-k}\widehat{um_1}(\xi)~d\xi .
\eeqn
Since $\widehat{um_1}\in L^{q'}$ for some $2\le q'<\infty$, $F(z)$ is obviously bounded and analytic in the upper half plane. Furthermore, $F(x+iy)$ converges uniformly to $m_1(x)-1$ as $y\searrow 0$. This shows $m_1(x,k)\in \hardy^{p,+}$. The eigenfunction $\phi_j(x)$ can be treated in a similar way.

Next we work on the cases $k=\lambda\pm 0i$. We provide arguments only for $m_1(x,\lambda+0i)$. The other functions can be treated similarly. We abbreviated $m_1(x,\lambda+0i)$ simply as $m_1$. Since
\eqn
\frac1i\partial_x m_1-C_+(um_1) = \lambda(m_1-1),
\eeqn
and $m_1(x)-1\to 0$ as $x\to -\infty$, we get
\eqn
m_1(x) = 1 + i\int_{-\infty}^x e^{i\lambda(x-s)}C_+(um_1)(s)~ds.
\eeqn
For $z=x+iy$ with $y>0$, define
\eqn\label{def: F}
F(z) = m_1(0) + i\int_{0}^z e^{i\lambda(z-s)}C_+(um_1)(s)~ds.
\eeqn
Here the integral is taken along any smooth contour in the upper half plane with end points at $0$ and $z$. We have used the analytic extension of $C_+(um_1)$ into the upper half plane as $C_+(um_1)\in \hardy^{p,+}$. $F(z)$ is obviously analytic in the upper half plane. We now estimate $F(x+iy)-m_1(x)$. To do that, we take the contour of integration in \cref{def: F} to be the straight line from $0$ to $x$, followed by the straight line from $x$ to $x+iy$. It follows that
\eqn
F(x+iy)-m_1(x) = i\int_{0}^{y}e^{-\lambda(y-s)}C_+(um_1)(x+is)~ds.
\eeqn
Using the elementary estimate on $\hardy^{p,+}$ functions (see Lemma 2.12 in \cite{stein2016introduction})
\eqn
|C_+(um_1)(x+is)|\le Cs^{-\frac1p}\|C_+(um_1)\|_{\hardy^{p,+}},
\eeqn
we get
\begin{align}
|F(x+iy)-m_1(x)|&\le C\|C_+(um_1)\|_{\hardy^{p,+}}\int_0^y e^{-\lambda(y-s)}s^{-\frac{1}{p}}~ds\notag\\
&= C\|C_+(um_1)\|_{\hardy^{p,+}}\int_0^1 y^{1-\frac1p}e^{-\lambda y(1-s)}s^{-\frac{1}{p}}~ds
\end{align}
For $\lambda>0$, $y>0$, $p>1$, and $0<s<1$, we have the elementary estimate
\eqn
y^{1-\frac1p}e^{-\lambda y(1-s)}\le \min\left(y^{1-\frac1p}, \left[\frac{1-\frac1p}{\lambda(1-s)}\right]^{1-\frac1p}e^{-(1-\frac1p)}\right).
\eeqn
Therefore for some constant $C=C(u,m_1,p,\lambda)>0$
\eqn
|F(x+iy)-m_1(x)| \le C\min\left(1,y^{1-\frac1p}\right).
\eeqn
This implies that $F(z)$ is bounded and $F(x+iy)$ converges to $m_1(x)$ uniformly as $y\searrow 0$. In other words, $m_1\in \hardy^{\infty,+}$.
\end{proof}

Next, we show that $(\partial_t L_u + [L_u,B_u])\varphi=0$ if $\varphi\in \hardy^{\infty,+}$ and is suitably smooth. In fact, repeating the derivation of the Lax pair in \cite{wu2016simplicity} using the modified $L_u$ and $B_u$ given in \cref{def: Lu} and \cref{def: Bu} provides
\begin{equation}
[L_u,B_u] \varphi = \frac{2}{i}(C_+u_{xx})\varphi-\frac{1}{i}C_+(u_{xx}\varphi)-2C_+(u_xu\varphi).
\end{equation}
Using the BO equation \cref{eq: BO}, we get 
\begin{equation}
(\partial_t L_u)\varphi = -C_+(u_t\varphi) = 2C_+(uu_x\varphi)+\frac{1}{i}C_+([(u_{xx}-2(C_+u_{xx})]\varphi)
\end{equation}
Hence
\begin{align}\label{eq: modified lax}
(\partial_t L_u+[L_u,B_u])\varphi &= \frac{2}{i}(C_+u_{xx})\varphi- \frac2i C_+((C_+u_{xx})\varphi) \notag\\
&= \frac{2}{i}C_-\left[(C_+u_{xx})\varphi\right].
\end{align}
Since $\varphi\in \hardy^{\infty,+}$, we get $(C_+u_{xx})\varphi \in \hardy^+$, and $C_-\left[(C_+u_{xx})\varphi\right]=0$. Thus we may use the Lax equation on all eigenfunctions $\phi_j$ and Jost functions $m_1$ and $m_e$, by \Cref{lem: hardy infty}. 

The standard argument of a Lax pair shows that all eigenvalues $\{\lambda_j\}_{j=1}^N$ do not change with time. We take the time derivative of $L_u\phi_j = \lambda_j\phi_j$ to get
\eqn\label{eq: phi evo 1}
(\partial_t L_u)\phi_j + L_u(\partial_t \phi_j) = \lambda_j \partial_t\phi_j.
\eeqn
Using the Lax equation $(\partial_t L_u+[L_u,B_u])\phi_j=0$, \cref{eq: phi evo 1} becomes
\eqn
(L_u-\lambda_j)(\partial_t \phi_j - B_u\phi_j)=0.
\eeqn
In other words, $\partial_t \phi_j - B_u\phi_j$ is an eigenfunction corresponding to $\lambda_j$. By the simplicity of the eigenvalues proven in \cite{wu2016simplicity}, $\partial_t \phi_j - B_u\phi_j$ is a multiple of $\phi_j$. To find out the multiplicity constant, we compare the asymptotics when $x\to\pm\infty$. By the normalization used in \cite{wu2016simplicity}, $\phi_j(x) \sim \frac1x$ as $x\to\pm\infty$. On the other hand, we argue that
\eqn
\partial_t\phi_j - B_u\phi_j = \partial_t\phi_j - \frac1i\partial_x^2\phi_{j}-2[(C_+u_x)\phi_j-C_+((u\phi_j)_x)]=o\left(\frac1x\right),
\eeqn
which implies \cref{eq: evo phi} as a consequence.
In fact, $\partial_t\phi_j - \frac1i\partial_x^2\phi_{j}=o\left(\frac1x\right)$ if we formally exchange derivatives with asymptotics. $xC_+u_x \to 0$ as $x\to\pm \infty$ because
\eqn
F(xC_+u_x) (\xi)= -\partial_\xi(\chi_{\real^+}\xi\hat{u})=-\chi_{\real^+}\hat u + i\chi_{\real^+}\xi \widehat{xu}\in L^1.
\eeqn
For a similar reason $xC_+((u\phi_j)_x)\to 0$ as $x\to \pm\infty$.

We can show \cref{eq: evo m1} and \cref{eq: evo me} similarly. A few differences in the arguments are noted: the step where we used simplicity of eigenvalues is now replaced by uniqueness of Jost solutions; $\partial_t m_e(\lambda \pm 0i)-B_um_e(\lambda \pm 0i)\sim -i\lambda^2e^{i\lambda x}$ as $x\to \mp\infty$.

We now derive the time evolution of the scattering coefficients, starting with \cref{eq: evo gamma}. Taking the time derivative of \cref{eq: Laurent}, we obtain
\eqn\label{eq: evo gamma 1}
\partial_t m_1 = -\frac{i}{k-\lambda_j}\partial_t \phi_j + (\partial_t\gamma_j)\phi_j + (x+\gamma_j)\partial_t\phi_j + (k-\lambda_j)\partial_t h(k,\lambda_j).
\eeqn
Letting $B_u$ act on \cref{eq: Laurent},  we obtain
\eqn\label{eq: evo gamma 2}
B_u m_1 = -\frac{i}{k-\lambda_j}B_u \phi_j  + B_u[(x+\gamma_j)\phi_j ]+ (k-\lambda_j)B_u h(k,\lambda_j).
\eeqn
Take the difference of \cref{eq: evo gamma 1} with \cref{eq: evo gamma 2}, use \cref{eq: evo phi}, \cref{eq: evo m1}, and evaluate at $k=\lambda_j$ to get
\eqn
(\partial_t\gamma_j)\phi_j + [x+\gamma_j,B_u]\phi_j =0.
\eeqn
We compute the commutator term and get
\begin{align}
[x+\gamma_j,B_u]\phi_j &= -\frac2i\partial_x\phi_j + 2C_+(u\phi_j) -2[xC_+(u\phi_j)_x-C_+(x(u\phi_j)_x)]\notag\\
&=-\frac2i\partial_x\phi_j + 2C_+(u\phi_j)\notag\\
&=-2L_u\phi_j=-2\lambda_j\phi_j.
\end{align}
The terms in the square brackets vanish as is easily seen by taking its Fourier transform. It follows that
\eqn
(\partial_t\gamma_j-2\lambda_j)\phi_j=0,
\eeqn
from which we get \cref{eq: evo gamma}. To obtain \cref{eq: evo Gamma}, we take the time derivative of \cref{eq: me pm rel} and also act on it by $B_u$. We get
\begin{align}
\partial_tm_e(\lambda+0i)&=\Gamma\partial_t m_e(\lambda-0i)+(\partial_t\Gamma )m_e(\lambda-0i),\\
B_u m_e(\lambda+0i)&=\Gamma B_u  m_e(\lambda-0i).
\end{align}
Take the difference and use \cref{eq: evo me} to get
\eqn
-i\lambda^2 m_e(\lambda+0i) = (\partial_t\Gamma-i\lambda^2\Gamma)m_e(\lambda-0i).
\eeqn
Now use \cref{eq: me pm rel} again to get \cref{eq: evo Gamma}. Finally to obtain \cref{eq: evo beta}, we perform a similar calculation using \cref{eq: m1 jump}. We first get
\begin{align}
\partial_tm_1(\lambda+0i) - \partial_tm_1(\lambda-0i) &= \beta\partial_t m_e(\lambda-0i) + (\partial_t\beta)m_e(\lambda-0i),\\
B_u m_1(\lambda+0i) - B_u m_1(\lambda-0i) &= \beta B_um_e(\lambda-0i).
\end{align}
Take the difference and use \cref{eq: evo m1} and \cref{eq: evo me} to get
\eqn
(\partial_t\beta -i\lambda^2\beta) m_e(\lambda-0i)=0,
\eeqn
from which \cref{eq: evo beta} follows.

%%%%%%%%%%%%%%%%%%%%%%%%%%%%%%%%%%%%%%%%%%%%%%%%%%%%%%%%%%%%%%%%%%%%%%%%%%%%%%%%%%%%%%%%%%%%%%%%%%%%%%%%%%%%%%%%%%%%%%%%%%%%%%%%%%%%%%%%%%%%%%%%
\bibliographystyle{siamplain}
\bibliography{BObiblio}

\begin{thebibliography}{10}

\bibitem{ablowitz1991solitons}
{\sc M.~J. Ablowitz and P.~A. Clarkson}, {\em Solitons, nonlinear evolution
  equations and inverse scattering}, vol.~149, Cambridge university press,
  1991.

\bibitem{benjamin1967internal}
{\sc T.~B. Benjamin}, {\em Internal waves of permanent form in fluids of great
  depth}, Journal of Fluid Mechanics, 29 (1967), pp.~559--592.

\bibitem{bock1979two}
{\sc T.~Bock and M.~Kruskal}, {\em A two-parameter {M}iura transformation of
  the {B}enjamin-{O}no equation}, Physics Letters A, 74 (1979), pp.~173--176.

\bibitem{choi1999fully}
{\sc W.~Choi and R.~Camassa}, {\em Fully nonlinear internal waves in a
  two-fluid system}, Journal of Fluid Mechanics, 396 (1999), pp.~1--36.

\bibitem{coifman1990scattering}
{\sc R.~R. Coifman and M.~V. Wickerhauser}, {\em The scattering transform for
  the {B}enjamin-{O}no equation}, Inverse Problems, 6 (1990), p.~825.

\bibitem{davis1967solitary}
{\sc R.~E. Davis and A.~Acrivos}, {\em Solitary internal waves in deep water},
  Journal of Fluid Mechanics, 29 (1967), pp.~593--607.

\bibitem{fokas1983inverse}
{\sc A.~Fokas and M.~Ablowitz}, {\em The inverse scattering transform for the
  {B}enjamin-{O}no equation -- {A} pivot to multidimensional problems.}, Stud.
  Appl. Math., 68 (1983), pp.~1--10.

\bibitem{jose1986cauchy}
{\sc R.~Jos{\'e}~I{\'o}rio, Jr}, {\em On the {C}auchy problem for the
  {B}enjamin-{O}no equation}, Communications in partial differential equations,
  11 (1986), pp.~1031--1081.

\bibitem{kaup1998inverse}
{\sc D.~Kaup and Y.~Matsuno}, {\em The inverse scattering transform for the
  {B}enjamin-{O}no equation}, Studies in applied mathematics, 101 (1998),
  pp.~73--98.

\bibitem{kenig2003local}
{\sc C.~E. Kenig and K.~D. Koenig}, {\em On the local well-posedness of the
  {B}enjamin-{O}no and modified {B}enjamin-{O}no equations}, Mathematical
  Research Letters, 10 (2003), pp.~879--896.

\bibitem{koch2003local}
{\sc H.~Koch and N.~Tzvetkov}, {\em On the local well-posedness of the
  {B}enjamin-{O}no equation in ${H}^s (\mathbb{R})$}, International Mathematics
  Research Notices, 2003 (2003), pp.~1449--1464.

\bibitem{nakamura1979direct}
{\sc A.~Nakamura}, {\em A direct method of calculating periodic wave solutions
  to nonlinear evolution equations. {I}. {E}xact two-periodic wave solution},
  Journal of the Physical Society of Japan, 47 (1979), pp.~1701--1705.

\bibitem{ono1975algebraic}
{\sc H.~Ono}, {\em Algebraic solitary waves in stratified fluids}, Journal of
  the Physical Society of Japan, 39 (1975), pp.~1082--1091.

\bibitem{ponce1991global}
{\sc G.~Ponce et~al.}, {\em On the global well-posedness of the
  {B}enjamin-{O}no equation}, Differential and Integral Equations, 4 (1991),
  pp.~527--542.

\bibitem{saut1979sur}
{\sc J.-C. Saut}, {\em Sur quelques g{\'e}n{\'e}ralisations de l'{\'e}quation
  de {K}orteweg-de {V}ries}, Journal de Math{\'e}matiques Pures et
  Appliqu{\'e}es, 58 (1979), pp.~21--61.

\bibitem{stein2016introduction}
{\sc E.~M. Stein and G.~Weiss}, {\em Introduction to Fourier analysis on
  Euclidean spaces (PMS-32)}, vol.~32, Princeton university press, 2016.

\bibitem{tao2004global}
{\sc T.~Tao}, {\em Global well-posedness of the {B}enjamin-{O}no equation in
  ${H}^1 (\mathbb{R})$}, Journal of Hyperbolic Differential Equations, 1
  (2004), pp.~27--49.

\bibitem{wu2016simplicity}
{\sc Y.~Wu}, {\em Simplicity and finiteness of discrete spectrum of the
  {B}enjamin--{O}no scattering operator}, SIAM Journal on Mathematical
  Analysis, 48 (2016), pp.~1348--1367.

\bibitem{xu2010asymptotic}
{\sc Z.~Xu}, {\em Asymptotic analysis and numerical analysis of the
  {B}enjamin-{O}no equation}, PhD thesis, The University of Michigan, 2010.

\end{thebibliography}

\end{document}